\documentclass[a4paper,11pt]{article}
\usepackage[utf8]{inputenc} 
\usepackage[T1]{fontenc}
\usepackage{lmodern}

\usepackage{amsthm,amsmath,amsfonts,amssymb,stmaryrd,bbm,mathrsfs}
\usepackage{mathtools}
\usepackage{enumitem}
\usepackage[dvipsnames]{xcolor}
\usepackage[colorlinks=true, linkcolor=NavyBlue, citecolor=NavyBlue]{hyperref}

\usepackage[top=2.7cm, bottom=2.7cm, left=2.5cm, right=2.5cm]{geometry}

\usepackage[french,english]{babel}
\usepackage{caption,tikz,color}
\usetikzlibrary{shapes}
\usetikzlibrary{patterns}

\makeatletter

\@addtoreset{equation}{section}
\makeatother

\setlist[enumerate,1]{label=(\roman*), font = \normalfont} 

\let\originalleft\left
\let\originalright\right
\renewcommand{\left}{\mathopen{}\mathclose\bgroup\originalleft}
\renewcommand{\right}{\aftergroup\egroup\originalright}

\newlength{\bibitemsep}
\newlength{\bibparskip}
\let\oldthebibliography\thebibliography
\renewcommand\thebibliography[1]{\oldthebibliography{#1}
  \setlength{\parskip}{\bibitemsep}
  \setlength{\itemsep}{\bibparskip}}

\newcommand{\N}{\mathbb{N}}

\newcommand{\Q}{\mathbb{Q}}
\newcommand{\R}{\mathbb{R}}
\newcommand{\C}{\mathbb{C}}
\newcommand{\T}{\mathbb{T}}
\renewcommand{\P}{\mathbb{P}}
\newcommand{\E}{\mathbb{E}}

\newcommand{\Ec}[1]{\mathbb{E} \left[#1\right]}
\newcommand{\Pp}[1]{\mathbb{P} \left(#1\right)}
\newcommand{\Ecsq}[2]{\mathbb{E} \left[#1\middle|#2\right]}
\newcommand{\Ppsq}[2]{\mathbb{P} \left(#1\middle|#2\right)}
\newcommand{\Varsq}[2]{\Var \left(#1\middle|#2\right)}
\newcommand{\Eci}[2]{\mathbb{E}_{#1} \left[#2\right]}

\newcommand{\1}{\mathbbm{1}}

\newcommand{\cN}{\mathcal{N}}

\newcommand{\cC}{\mathcal{C}}

\newcommand{\cL}{\mathcal{L}}

\newcommand{\cS}{\mathcal{S}}

\newcommand{\sF}{\mathscr{F}}

\newcommand{\ep}{\varepsilon}
\newcommand{\e}{\mathrm{e}}
\newcommand{\diff}{\mathop{}\mathopen{}\mathrm{d}}
\DeclareMathOperator{\Var}{Var}
\DeclareMathOperator{\sinhc}{sinhc}

\newcommand{\abs}[1]{\left\lvert#1\right\rvert}
\newcommand{\abso}[1]{\lvert#1\rvert}
\newcommand{\norme}[1]{\left\lVert#1\right\rVert}

\newcommand\relphantom[1]{\mathrel{\phantom{#1}}}

\theoremstyle{plain}  
\newtheorem{thm}{Theorem}[section]
\newtheorem{prop}[thm]{Proposition}
\newtheorem{lem}[thm]{Lemma}
\newtheorem{cor}[thm]{Corollary}
\newtheorem{conj}[thm]{Conjecture}

\theoremstyle{definition}

\theoremstyle{remark}
\newtheorem{rem}[thm]{Remark}

\title{1-stable fluctuations in branching Brownian motion at critical temperature II: general functionals}

\author{Pascal \bsc{Maillard}%
\footnote{Univ Toulouse, Institut Universitaire de France, INUC, UT2J, INSA Toulouse, TSE, CNRS, IMT, Toulouse, France. PM acknowledges partial support from Institut Universitaire de France, the MITI interdisciplinary program 80PRIME GEx-MBB and from ANR through the MBAP-P (ANR-24-CE40-1833) project.}~
and Michel \bsc{Pain}%
\footnote{Univ Toulouse, INUC, UT2J, INSA Toulouse, TSE, CNRS, IMT, Toulouse, France. MP acknowledges partial support from the MITI interdisciplinary program 80PRIME GEx-MBB and from ANR through the MBAP-P (ANR-24-CE40-1833) project.}}
\date{February 4, 2026}

\begin{document}

\maketitle

\begin{abstract}
\noindent Let $\mu_t$ denote the critical derivative Gibbs measure of branching Brownian motion at time~$t$.
It has been proved by Madaule (\textit{Stochastic Process. Appl.} \textbf{126} (2016), no. 2, 470--502) and Maillard and Zeitouni (\textit{Ann. Inst. Henri Poincaré Probab. Stat.} \textbf{52} (2016), no. 3, 1144--1160) that $\mu_t$ converges weakly to the random measure $Z_\infty \sqrt{2/\pi} x^2 \e^{-x^2/2} \1_{x >0} \diff x$, where $Z_\infty$ is the limit of the derivative martingale.
In this paper, we are interested in the fluctuations that occur in this convergence and prove for a large class of functions $F$ that
\begin{align*}
	& \sqrt{t} \left( 
	\int_\R F \diff \mu_t 
	- Z_\infty \int_0^\infty F(x) \sqrt{\frac{2}{\pi}} x^2 \e^{-x^2/2} \diff x
	- \frac{c(F) \log t}{\sqrt{t}} Z_\infty \right) \\
	& \xrightarrow[t\to\infty]{} S(F),
	\quad \text{in law},
\end{align*}
where $c(F)$ is a constant depending on $F$ and, given $Z_\infty$, $S(F)$ has an explicit 1-stable distribution.
Moreover, we extend this result to a functional convergence, and we identify precisely the particles responsible for the fluctuations.
In particular, this proves the following result for the critical additive martingale $(W_t)_{t\geq 0}$: 
\[
\sqrt{t} \left( \sqrt{t} W_t - \sqrt{\frac{2}{\pi}} Z_\infty \right)
\xrightarrow[t\to\infty]{} C Z_\infty,
\quad \text{in law},
\]
where here $C$ is a Cauchy variable independent of $Z_\infty$, 
confirming a conjecture by Mueller and Munier (\textit{Phys. Rev. E} \textbf{90} (2014), 042143) in the physics literature.
\end{abstract}


{
	\hypersetup{linkcolor=black}
	\tableofcontents
}

\section{Introduction}

In this paper, we are interested in fluctuations arising in the front of branching Brownian motion, where here the front refers to the particles supporting the critical Gibbs measure or equivalently those which are at a distance of order $\sqrt{t}$ above the minimum.
This is the sequel of \cite{maillardpain2019} by the authors, where the fluctuations of the derivative martingale, a key quantity describing the front, were obtained. 
Here, we investigate the fluctuations of general functionals of the front and show that they are 1-stable with possibly any value of the asymmetry parameter, whereas for the derivative martingale they were totally asymmetric.

\subsection{Definitions and results}

\paragraph{Definitions.}
Branching Brownian motion (BBM) is a branching Markov process defined as follows. 
Initially, there is a single particle at the origin. 
Each particle moves according to a Brownian motion with variance $\sigma^2>0$ and drift $\rho \in \R$, during an exponentially distributed time of parameter $\lambda > 0$ and then splits into a random number of new particles, chosen accordingly to a reproduction law~$\mu$.
These new particles start the same process from their place of birth, behaving independently of the others. 
The system goes on indefinitely, unless there is no particle at some time.


Let $L$ denote a random variable on $\N \coloneqq \{ 0, 1, \dots \}$ with law $\mu$.
Our assumptions in this paper concerning the reproduction law are
\begin{equation} \label{hypothese 1}
\Ec{L} >1
\quad \text{ and } \quad
\Ec{L^2} < \infty.
\end{equation}
The first inequality implies that the underlying Galton-Watson tree $\T$ is supercritical and the event 
of non-extinction of the population has positive probability.

Let $\cN(t)$ be the set of particles alive at time $t$ and $X_u(t)$ the position of particle $u$ at time~$t$.
More generally, for a particle $u \in \cN(t)$ and $s \leq t$, we write $X_u(s)$ for the position at time $s$ of the single ancestor of $u$ alive at time $s$.
We denote by $(\sF_t)_{t\geq 0}$ the canonical filtration associated with the BBM.
As in \cite{maillardpain2019} and in the branching random walk literature \cite{aidekon2013,aidekonshi2014}, we take $\sigma = \rho = 1$ and $\lambda = 1/(2\E[L-1])$ in order to have, for any $t >0$,
\begin{align*}
\Ec{\sum_{u \in \cN(t)} \e^{-X_u(t)}} = 1,
\quad \Ec{\sum_{u \in \cN(t)} X_u(t) \e^{-X_u(t)}} = 0
\quad \text{and} \quad
\Ec{\sum_{u \in \cN(t)} X_u(t)^2 \e^{-X_u(t)}} = t.
\end{align*}
Moreover, we define the \textit{critical additive martingale}
\begin{align*}
W_t \coloneqq \sum_{u \in \cN(t)} \e^{-X_u(t)},
\quad t \geq 0,
\end{align*}
which converges a.s.\@ to 0 \cite{neveu87}, and the \textit{derivative martingale}
\begin{align*}
Z_t \coloneqq \sum_{u \in \cN(t)} X_u(t) \e^{-X_u(t)}, 
\quad t \geq 0.
\end{align*}

\paragraph{Background.}
It has been proved by Lalley and Sellke \cite{lalleysellke87} for binary branching and then by Yang and Ren \cite{yangren2011} under the optimal assumption $\E[L \log_+^2 L] < \infty$ that
the derivative martingale converges a.s.\@ to a limit $Z_\infty$ which is positive on the non-extinction event.
Moreover, it has been proved for the branching random walk by Aïdékon and Shi \cite{aidekonshi2014} that after a proper renormalization, the critical additive martingale converges also to the same limit:
\begin{equation} \label{eq:cv-of-W_t}
\sqrt{t} W_t \xrightarrow[t\to\infty]{} \sqrt{\frac{2}{\pi}} Z_\infty,
\quad \text{in probability},
\end{equation}
and this result extends to the case of BBM.

The study of BBM has been mainly focused on the behavior of the extremal particles (those which are at a distance of order $1$ from the minimum of BBM): Bramson \cite{bramson78,bramson83} and Lalley and Sellke \cite{lalleysellke87} proved the following convergence for the position of the lowest particle of BBM
\begin{align} \label{eq:cv-du-min-du-BBM}
\Pp{\min_{u\in\cN(t)} X_u(t) \geq \frac{3}{2} \log t + x} 
\xrightarrow[t\to\infty]{} 
\Ec{\e^{-c^* \e^x Z_\infty}},
\end{align}
for some positive constant $c^*$, and Aïdékon, Berestycki, Brunet and Shi \cite{abbs2013}, Arguin, Bovier and Kistler \cite{abk2013} described the limit of the whole extremal process.
One can note that the derivative martingale plays a role in the behavior of the extremal particles. 
Indeed, for any $t \leq s$ such that $t$ and $s-t$ are large, extremal particles at time $s$ descend of the particles at time $t$ that mainly contribute to $Z_t$: these are the particles with a position of order $\sqrt{t}$ at time $t$ and we call them the \textit{front} of the BBM. 
A way to pick a typical particle of the front is to consider the \textit{critical derivative Gibbs measure} defined by
\[
\sum_{u \in \cN(t)} X_u(t) \e^{-X_u(t)} \delta_{X_u(t)/\sqrt{t}}.
\]
Note that this is not a probability measure, but it has total weight $Z_t$ which is of order $1$ on the non-extinction event.
This measure converges weakly as $t \to \infty$: more precisely, Madaule \cite{madaule2016} (for the branching random walk) and Maillard and Zeitouni \cite{maillardzeitouni2016} proved that, for any $F \colon \R \to \R$ continuous and bounded, we have
\begin{equation} \label{eq:cv-Z_t(F)}
Z_t(F) 
\coloneqq \sum_{u \in \cN(t)} X_u(t) \e^{-X_u(t)} 
	F \left( \frac{X_u(t)}{\sqrt{t}} \right)
\xrightarrow[t\to\infty]{} \rho(F) Z_\infty,
\quad \text{in probability},
\end{equation}
where $\rho$ is the law at time $1$ of a 3-dimensional Bessel process starting from 0, i.e.
\begin{equation}
	\label{eq:def_rho}
\rho(\diff x) = \sqrt{\frac 2 \pi} x^2 e^{-x^2/2}\1_{x>0} \diff x,
\end{equation}
and we use, for $F\in L^1(\rho)$, the shorthand notation
\[
\rho(F) \coloneqq \int_0^\infty F(x)\,\rho(\diff x).
\]
The convergence in \eqref{eq:cv-Z_t(F)} shows the derivative martingale is the key quantity to describe the front of BBM: given $Z_\infty$, the distribution of particles in the front is deterministic at first order.
Note that it is equivalent to work with the (rescaled) critical Gibbs measure 
$\sqrt{t} \sum_{u \in \cN(t)} \e^{-X_u(t)} \delta_{X_u(t)/\sqrt{t}}$, up to a replacement of $F(x)$ by $xF(x)$.
Moreover, the convergence \eqref{eq:cv-Z_t(F)} generalizes \eqref{eq:cv-of-W_t} for the critical additive martingale.
In a previous paper \cite{maillardpain2019}, we studied the fluctuations of the derivative martingale around its limit. To state the result, following Samorodnitsky and Taqqu~\cite{samorodnitskytaqqu1994}, we define the $1$-stable distribution $\cS_1(\sigma,\beta,\mu)$, with scale parameter $\sigma > 0$, skewness $\beta\in[-1,1]$ and shift parameter $\mu \in \R$, as the distribution on the real line having characteristic function $\Psi_{\sigma,\beta,\mu}(\lambda) \coloneqq \exp(-\psi_{\sigma,\beta,\mu}(\lambda))$, 
where
\begin{align}
	\label{eq:Psi_sigma_mu}
	\psi_{\sigma,\beta,\mu}(\lambda) 
	\coloneqq 
	\sigma \left(\abs{\lambda} 
	+ i \beta \frac{2}{\pi} \lambda \log \abs{\lambda}
	\right)
	- i \mu \lambda
	,\quad \lambda\in\R.
\end{align}
We recall the main result of \cite{maillardpain2019}  (see Appendix A in \cite{maillardpain2019} for the notion of weak convergence in probability):
\begin{thm}[\cite{maillardpain2019}]\label{theorem-previous}
Conditioned on $\sF_t$, the finite-dimensional distributions of the stochastic process $$\sqrt{t} \left( Z_{at} - Z_\infty - \frac{\log t}{\sqrt{2 \pi t}} Z_\infty \right)_{a\ge 1}$$
converge in the sense of weak convergence in probability to the finite-dimensional distributions of $(-S_{Z_\infty/\sqrt{a}})_{a\ge 1}$ given $Z_\infty$, where $(S_t)_{t\ge 0}$ is a Lévy process independent of $Z_\infty$, starting at 0 and such that $S_1$ has distribution $$\cS_1\left(\sqrt{\frac \pi 2},1,\sqrt{\frac 2 \pi}\mu_Z\right),$$
where the constant $\mu_Z$ is defined as follows\footnote{In \cite{maillardpain2019}, the $+1$ in the definition of $\mu_Z$ was mistakenly missing, as pointed out to us by the authors of \cite{buraczewskiiksanovmallein2021}.\label{foot:Zinfty}}:
\[
\mu_Z \coloneqq \lim_{x\to\infty} \E[Z_\infty\boldsymbol 1_{Z_\infty \le x}] - \log x - \gamma + 1,
\]
where $\gamma$ is the Euler-Mascheroni constant.
\end{thm}

In this paper, we generalize this convergence to $Z_t(F)$ for a very general class of functions $F$. In particular, we study the fluctuations of the additive martingale $W_t$ around its limit, which corresponds to the case $F \colon x \mapsto 1/x$.
The functional limit is best described in terms of \emph{(1-)stable integrals}, whose definition we now recall.

\paragraph{1-stable integrals.}

We recall the classical notion of $\alpha$-stable integrals from Samorodnitsky and Taqqu \cite[Chapter 3]{samorodnitskytaqqu1994},  restricting to the case $\alpha = 1$ and to spectrally positive stable integrals, i.e.~with skewness equal to 1. Let $(E,\mathcal E,m)$ be a measure space, with $m$ a $\sigma$-finite measure, and define the linear space of functions
\[
\mathcal G(m) = \{g:E\to \R\text{ measurable:} \int_E \abs{g(r)}(1+\abs{\log\abs{g(r)}})m(\diff r) < \infty\},
\]
where we use throughout the convention $0\log 0 = 0$.
The \emph{spectrally positive 1-stable integral with control measure $m$} is a stochastic process $(I(g))_{g\in\mathcal G(m)}$ which has the following properties:
\begin{enumerate}
	\item $I$ is linear, i.e. $I(a_1 g_1+\cdots + a_n g_n) = a_1I(g_1) + \cdots + a_n I(g_n)$ almost surely for $a_1,\ldots,a_n\in\R$ and $g_1,\ldots,g_n\in \mathcal G(m)$
	\item For $g\in \mathcal G(m)$, we have
\begin{align} \label{eq:characteristic_function_M}
	\Ec{e^{i I(g)}}	& = \exp \left( - \int_0^\infty 
	\left[ \abs{g(r)} + i g(r) \frac{2}{\pi} \log \abs{g(r)}  \right] 
	m(\diff r)\right).
\end{align}	
\end{enumerate}
Note that these two properties uniquely define the finite-dimensional distributions of the process $(I(g))_{g\in\mathcal G(m)}$. Furthermore, using \eqref{eq:characteristic_function_M} with $\lambda g$ instead of $g$, we see that $I(g)$ is distributed according to $\cS_1(\sigma_g,\beta_g,\mu_g)$, with
\begin{equation}
	\label{eq:params_g}
	\sigma_g = \int_0^\infty\abs{g(r)}m(\diff r),\quad \beta_g = \frac{\int_0^\infty g(r)m(\diff r)}{\int_0^\infty \abs{g(r)}m(\diff r)}, \quad \mu_g = \int_0^\infty g(r)\frac{2}{\pi} \log \abs{g(r)} 
	m(\diff r).
\end{equation}
We also remark that $I(g_1),\ldots,I(g_n)$ are independent if $g_1,\ldots,g_n$ have mutually disjoint supports, as can be easily checked from \eqref{eq:characteristic_function_M}.

The term ``integral'' comes from the fact that $I(g)$ can be constructed analogously to the Lebesgue integral, but against the ``stable noise'' $M$, defined as the map
\[
M(A) \coloneqq I(\1_A),\quad\text{for $A\in \mathcal E$ such that $m(A)<\infty$.}
\]
$M$ is called a \emph{spectrally positive 1-stable noise with control measure $m$}.
It is not a random (signed) measure because it has infinite total variation (unless $m$ is the null measure). Nevertheless, it is customary to use the notation
\[
\int_E g(r)M(\diff r) \coloneqq I(g),
\]
and we will do so in the remainder of the article. We will also use standard notation for the domain of integration such as
\[
\int_A g(r)M(\diff r) \coloneqq \int_E g(r)\1_A(r)M(\diff r),
\]
and, in the case where $E\subset \R$ is an interval and $a,b\in E$, $a\le b$,
\[
\int_a^b g(r)M(\diff r) 
\coloneqq \int_E g(r) \1_{(a,b)}(r)M(\diff r).
\]

\paragraph{Main result.}
Let $F\colon \R\to\R$ be measurable and such that $\int_0^\infty \abs{F(sx)}\,\rho(\diff x) < \infty$ for all $s\in (0,1]$. We define a function $\mathscr R F \colon (0,\infty)\to\R$, by
\begin{align}
\nonumber
	\mathscr R F(r) 
	&\coloneqq 
	\int_0^\infty \left[F(\sqrt{1-r}\cdot x)\1_{r<1} - F(x)\right]\,\rho(\diff x)\\
	&= \rho(F(\sqrt{1-r}\,\cdot))\1_{r<1} - \rho(F),\quad r>0, 	\label{eq:def R F}
\end{align}
with $\rho$ defined in \eqref{eq:def_rho}. Note that the above assumption on $F$ is satisfied in particular if $\abs{F(x)} \le Cx^{-\alpha}e^{Cx}$ for some $\alpha < 3$ and $C<\infty$, for all $x>0$. 

Next, conditioned on $Z_\infty$, let $M'_{Z_\infty}$ denote a spectrally positive 1-stable noise with control measure 
$$ m(\diff r) = Z_\infty \frac{\sqrt \pi}{2\sqrt{2}} \frac{\diff r}{r^{3/2}},$$
We formally set
\begin{equation}
	\label{eq:def_M_Z_infty}
M_{Z_\infty} \coloneqq M'_{Z_\infty} + \mu_Z \frac{Z_\infty}{\sqrt{2\pi}}  \frac{\diff r}{r^{3/2}}.
\end{equation}
What this means is that for every $g_1,\ldots,g_n\in \mathcal G(r^{-3/2}\diff r)$, conditioned on $Z_\infty$,
\[
\left(\int g_i(r)M_{Z_\infty}(\diff r)\right)_{i=1,\ldots,n} \stackrel{\mathrm{(law)}}{=} \left(\int g_i(r)M'_{Z_\infty}(\diff r) + \mu_Z \int g_i(r) \frac{Z_\infty}{\sqrt{2\pi}}  \frac{\diff r}{r^{3/2}}\right)_{i=1,\ldots,n} .
\] 

We can now state the main result of the paper:
\begin{thm} \label{theorem-complete}
	Let $F \colon \R \to \R$.
	Assume $F$ is twice differentiable on $(0,\infty)$ and, for any $x>0$, $\abs{F''(x)} \leq C x^{-\alpha-2} \e^{C x}$ for some $\alpha \in [0,2)$ and $C>0$.
	Then, $\mathscr R F(\cdot/a)\in \mathcal G(r^{-3/2}\diff r)$ for all $a\in(0,\infty)$ and conditionally on $\sF_{\ep t}$, as $t\to\infty$ and then $\ep \to 0$, the finite-dimensional distributions of
	\begin{equation} \label{eq:process_at_t}
		\sqrt t \cdot \left( 
		Z_{at}(F) - \rho(F) Z_\infty
		+ \frac{\log t}{2 \sqrt{t}} 
		\int_0^\infty \mathscr R F\left( \frac{r}{a} \right) \frac{Z_\infty}{\sqrt{2\pi}}  \frac{\diff r}{r^{3/2}}
		\right)_{a\in(0,\infty)}
	\end{equation}
	converge weakly in probability to the finite-dimensional distributions, conditionally on $Z_\infty$, of
	\begin{equation} \label{eq:process_limit}
		\left( \int_0^\infty \mathscr R F \left( \frac{r}{a} \right) M_{Z_\infty}(\diff r)\right)_{a\in(0,\infty)}.
	\end{equation}		
	More generally, if $F_1,\ldots,F_n$ are functions as above, and we consider the processes defined as in \eqref{eq:process_at_t} with $F = F_i$ for each $i=1,\ldots,n$, then the above convergence holds jointly in $i=1,\ldots,n$.
\end{thm}

\begin{rem}
	As mentioned above, see Appendix A in \cite{maillardpain2019} for a brief introduction to weak convergence in probability for random probability measures. 
	In particular, by \cite[Proposition A.1]{maillardpain2019}, the convergence stated in the theorem is equivalent to 
	the pointwise convergence in probability of the conditional characteristic function of any finite-dimensional marginal of \eqref{eq:process_at_t} given $\sF_{\ep t}$ 
	toward the conditional characteristic function of the corresponding finite-dimensional marginal of \eqref{eq:process_limit} given $Z_\infty$ as $t\to\infty$ and then $\ep \to 0$.
\end{rem}

\begin{rem}
	\label{rem:one-dimensional}
	Let us express the one-dimensional marginals of the process in \eqref{eq:process_limit}. 
	Assume $F$ is not almost everywhere zero on $(0,\infty)$, then one readily checks that $\mathscr R F$ is not almost everywhere zero, neither. A quick calculation using \eqref{eq:params_g} shows that for a given $a\in (0,\infty)$, conditionally on $Z_\infty$, the random variable in \eqref{eq:process_limit} follows the law 
	\[
	\cS_1\left(\frac{Z_\infty}{\sqrt a} \sigma_F,\beta_F,\frac{Z_\infty}{\sqrt a}\mu_F\right),
	\]
	where
	\begin{align*}
	\sigma_F &= \int_0^\infty \abs{\mathscr R F(r)} \frac{\sqrt \pi}{2\sqrt{2}} \frac{\diff r}{r^{3/2}},\\
	\beta_F &= \frac{\int_0^\infty \mathscr R F(r) \frac{\diff r}{r^{3/2}}}{\int_0^\infty  \abs{\mathscr R F(r)} \frac{\diff r}{r^{3/2}}},\\
	\mu_F &= \int_0^\infty \mathscr R F(r)\left(\log\abs{\mathscr R F(r)}+\mu_Z\right)\frac 1 {\sqrt{2\pi}}\frac{\diff r}{r^{3/2}}.
	\end{align*}
	Hence, the logarithmic correction term in \eqref{eq:process_at_t} is equal to
	\[
	\frac{\log t}{2\sqrt {at}} \frac{2}{\pi}Z_\infty\sigma_F \beta_F.
	\]
	In particular, this term vanishes if and only if $\beta_F = 0$.
\end{rem}

\begin{rem}
	\label{rem:scaling}
	The stable noise $M_{Z_\infty}$ satisfies a scaling property, due to the fact that its control measure $m$ satisfies $m(\diff r/\ep) = \sqrt \ep m(\diff r)$ for $\ep>0$. Indeed, we can formally write (conditioned on $Z_\infty$),
	\[
	M_{Z_\infty}(\diff r/\ep) \stackrel{\text{(law)}}{=} \sqrt \ep\left(M_{Z_\infty}(\diff r) + \frac{\log \ep}{2} \frac{Z_\infty}{\sqrt{2\pi}}\frac{\diff r}{r^{3/2}}\right).
	\]
	What this means is that for every $g_1,\ldots,g_n\in \mathcal G(r^{-3/2}\diff r)$, conditioned on $Z_\infty$, for every $\ep>0$,
	\[
	\left(\int g_i(\ep r)M_{Z_\infty}(\diff r)\right)_{i=1,\ldots,n} \stackrel{\mathrm{(law)}}{=} \sqrt \ep \left(\int g_i(r)M_{Z_\infty}(\diff r) + \frac{\log \ep}{2} \int g_i(r) \frac{Z_\infty}{\sqrt{2\pi}}  \frac{\diff r}{r^{3/2}}\right)_{i=1,\ldots,n} .
	\]
	This can be easily checked using \eqref{eq:characteristic_function_M}.
\end{rem}

\begin{rem}
	Our proof identifies precisely the particles responsible for the fluctuations of $Z_t(F)$ around its limit $\rho(F)Z_\infty$. These are the particles descending to $\frac{1}{2} \log t + O(1)$ after time $\ep t$ for $\ep$ small enough (recall that the minimum at time $t$ is typically near $\frac 3 2 \log t$). In fact, the particles reaching $\frac{1}{2} \log t + O(1)$ between times $r_1t$ and $r_2t$, for $0<r_1<r_2<\infty$, are those contributing to $M_{Z_\infty}([r_1,r_2])$, see the end of Section~\ref{subsection:motivations-comments} for more details. 
	The situation is more complex compared to the special case of the derivative martingale $Z_t = Z_t(1)$, where only the particles \emph{after} time $t$ are responsible for the fluctuations, see \cite[Section 1.3]{maillardpain2019}. In particular, this explains the different conditioning in the statement of Theorem~\ref{theorem-complete} compared to Theorem~\ref{theorem-previous}. 
\end{rem}

\begin{rem} \label{rem:thm-2}
	As an intermediate step towards the proof of Theorem~\ref{theorem-complete}, we have Theorem~\ref{theorem-2}, another main result of this paper. This theorem is similar to Theorem~\ref{theorem-complete}, but instead of conditioning on $\sF_{\varepsilon t}$ and letting $\varepsilon\to 0$, we condition on $\sF_t$ and obtain a limit result for the quantity $Z_{at}(F) - Z_t(G)$, $a> 1$, for some $G$ depending on $F$ and $a$. In that case, only the particles descending to $\frac 1 2 \log t + O(1)$ between times $t$ and $at$ play a role and therefore only the restriction of $M_{Z_\infty}$ to the interval $[1,a]$ appears in the description of the limit.
\end{rem}

\begin{rem}
	In the case of the derivative martingale ($F\equiv 1$), we have $\mathscr R F(r) = -\1_{r \ge 1}$, so that $\mathscr R F(r/a) = -\1_{r\ge a}$. It follows that 	
	 the limiting process defined in \eqref{eq:process_limit} has independent increments, in fact, it can be written as $(-S_{Z_\infty/\sqrt{a}})_{a \in (0, \infty)}$, with $(S_t)_{t\ge 0}$ the Lévy process defined above, matching with Theorem~\ref{theorem-previous}. In general, the process defined in \eqref{eq:process_limit} does not have independent increments.
\end{rem}

\begin{rem}
	In \cite{maillardpain2019}, the fluctuations of the derivative martingale (see above) have been obtained under the assumption $\E[L (\log_+ L)^3] < \infty$, instead of $\Ec{L^2} < \infty$. 
	We believe that Theorem~\ref{theorem-complete} also holds under this weaker assumption and that it is optimal.
	Nevertheless, the proof of fluctuations of the derivative Gibbs measure in this paper involves several additional technicalities and, therefore, we chose to avoid the truncation arguments needed to work under the assumption $\E[L (\log_+ L)^3] < \infty$.
\end{rem}
\paragraph{The additive martingale.}
We can apply Theorem~\ref{theorem-complete} in order to obtain the fluctuations of the additive martingale $W_t$. Defining $F \colon x \mapsto 1/x$, a quick calculation shows that
\[
\sqrt{t} W_t = Z_t(F)
\quad \text{and} \quad 
\mathscr R F(r) = \sqrt{\frac 2 \pi} \left(\frac{\1_{r<1}}{\sqrt{1-r}}-1\right).
\]
In particular, we have $\mathscr R F(r) >0$ if $r < 1$ and $\mathscr R F(r) < 0$ if $r\ge 1$. We can determine the quantities in Remark~\ref{rem:one-dimensional} as
\[
\sigma_F = 2,\quad 
\beta_F = 0 \quad \text{and} \quad 
\mu_F = -\frac{2\log 2}{\pi},
\]
see Appendix~\ref{app:additive_martingale} for detailed calculations.
Theorem~\ref{theorem-complete} and Remark~\ref{rem:one-dimensional} now yield the following result.
\begin{cor} \label{cor:fluctuations-for-W}
	We have the following convergence in distribution
	\begin{align*}
	\sqrt{t} \left( \sqrt{t} W_t - \sqrt{\frac{2}{\pi}} Z_\infty \right)
	\xrightarrow[t\to\infty]{} 
	Z_\infty C,
	\end{align*}
	where $C$ is a Cauchy variable independent of $Z_\infty$ with 
	$\E[\e^{i \xi C}]
	= \exp(- [2 \abs{\xi} + \frac{2 \log 2}{\pi} i \xi] )$ for $\xi \in \R$.
\end{cor}

Corollary~\ref{cor:fluctuations-for-W} can be seen as a rigorous formulation and proof of claims from the physics literature, see Mueller and Munier~\cite{muellermunier2014}, although they rather seem to have in mind $O(1)$ fluctuations of $\sqrt t W_t$ instead of the $O(1/\sqrt t)$ fluctuations considered here.

\subsection{Motivations and comments}
\label{subsection:motivations-comments}

The last decades have seen tremendous activity on the extremes of branching Brownian motion and other systems belonging to the so-called \emph{F-KPP universality class}, see for example the references in~\cite{maillardpain2019}. The \emph{fluctuations} of the extremes of branching Brownian motion have arosen recent interest \cite{muellermunier2014,maillardpain2019,mytnikroquejoffreryzhik2022}  because of the manifestation of a certain universality of the fluctuations, witnessed by the ubiquity of 1-stable distributions. The results of \cite{maillardpain2019} on the fluctuations of the derivative martingale have been adapted to branching random walk by Buraczewski, Iksanov and Mallein \cite{buraczewskiiksanovmallein2021} and Hou, Ren and Song \cite{hrs2024}.
The 1-stable limit contrasts with the $\alpha$-stable distributions with $\alpha >1$ appearing 
in the fluctuations of the \emph{subcritical} additive martingales of the branching random walk, which have been studied by Rösler, Topchii and Vatutin~\cite{rtv2002}, Iksanov and Kabluchko~\cite{iksanovkabluchko2016}, Iksanov, Kolesko and Meiners~\cite{ikm2019,ikm2020} and Hartung and Klimovsky~\cite{hartungklimovsky2018}, also in the case of complex BBM. Note also that the \emph{supercritical} additive martingales of branching random walk can be renormalized to converge to a limit distributed as an $\alpha$-stable subordinator, $\alpha<1$, taken at time $Z_\infty$, see \cite{brv2012}. 

This article is a continuation of the article \cite{maillardpain2019}, in which we studied the fluctuations of the derivative martingale $Z_t$. This martingale, which plays a crucial role in the description of the extremal particles, can be seen as the partition function of the derivative Gibbs measure defined above. This Gibbs measure is supported by particles at distance of order $\sqrt t$ from the minimum. It is natural to consider the more general functionals $Z_t(F)$ of the Gibbs measure, supported by the same particles. We find it striking that all possible values of the asymmetry parameter can be obtained through varying $F$, which is a new feature compared to previous results. A specific example is the critical additive martingale $W_t$, which was studied by Mueller and Munier \cite{muellermunier2014}, as mentioned above. In this example, there is no correction term of order $\log t/\sqrt t$, since $W_t$ is invariant by shifts, as explained at the end of Section~\ref{section:strategy} (see Footnote~\ref{foot:shift}). Since the role of this correction term is to compensate for small jumps, the fact that it vanishes suggests that the limiting distribution should be a (symmetric) Cauchy distribution, a reasoning which is confirmed by Remark~\ref{rem:one-dimensional}.

Studying the whole family $Z_t(F)$ allows also to explore what happens when the function $F$ becomes more and more singular at the origin. Indeed, Theorem~\ref{theorem-complete} allows to consider functions $F$ satisfying $F(x) = O(x^{-\alpha})$ as $x\to0$, for some $\alpha < 2$ and gives that the fluctuations of $Z_t(F)$ are still of order $1/\sqrt t$. We believe that the critical value $\alpha = 2$ is sharp, in fact, one checks that if $F(x) \sim x^{-\alpha}$ for some $\alpha \ge 2$, then the function $\mathscr R F$ is either not integrable near $r=1$ (if $\alpha\in [2,3)$), or not defined (if $\alpha \ge 3$). In fact, when $\alpha > 3$, we expect the functional $Z_t(F)$ to be supported by particles at distance $O(\log t)$ of the maximum\footnote{In a previous version, we conjectured that $Z_t(F)$ is supported by the extremal particles if $\alpha > 3$. We thank Xinxin Chen for showing us arguments that rather suggest this more accurate picture.}. This motivates the following conjecture:
\begin{conj}
\label{conj:alpha23}
If $F(x) \sim x^{-\alpha}$ as $x\to 0$ with $\alpha > 2$, the fluctuations of $Z_t(F)$ are of order $t^{-\beta(\alpha) + o(1)}$ with $\beta(\alpha)\in (0,\frac 1 2)$ for $\alpha\in(2,3)$ and $\beta(\alpha) = 0$ for $\alpha \ge 3$.
\end{conj}

In fact, Conjecture~\ref{conj:alpha23} should be related to recent work by Mytnik, Roquejoffre and Ryzhik \cite{mytnikroquejoffreryzhik2022}. In that work, the authors study the fluctuations of the number of particles in the limiting extremal process. They show that the number of particles below a high level $x$ has fluctuations of order $1/x$. We expect that the limiting process describes well the particles at time $t$ as long as $x \ll \sqrt t$. Indeed, the density of the limiting process grows as $xe^x$ \cite{cortineshartunglouidor2019}, which is also the asymptotic behaviour of the density predicted by \eqref{eq:cv-Z_t(F)}. Hence, it seems reasonable to assume that the number of particles in branching Brownian motion at distance $t^\beta$ from the minimum at time $t$, $0 < \beta < 1/2$,  also has fluctuations of order $O(t^{-\beta})$, and that it is these particles that contribute to the fluctuations of $Z_t(F)$ when $2<\alpha<3$.

In contrast, the current article treats the particles at distance of order $\sqrt t$ from the minimum. We expect this to be the regime where the comparison with the extremal process breaks down, i.e.\@ where the fluctuations predicted by \cite{mytnikroquejoffreryzhik2022} should cease to hold.
%
%

\paragraph{Proof ideas.}
Let us briefly describe the main ideas of the proof. 
A more detailed overview is provided in Section~\ref{section:strategy}.
We start by recalling the idea from \cite{maillardpain2019} to study the fluctuations of the derivative martingale $Z_t$: we introduce a ``barrier'' from time $t$ on at the point $\frac{1}{2} \log t + \beta_t$, where $\beta_t \to \infty$ slowly enough. The goal is to treat specially the particles  that hit the barrier, which are exactly those contributing to the fluctuations of $Z_t$. Indeed, the contribution to $Z_\infty$ of the descendants of such a particle is equal in law to $e^{-\beta_t}Z_\infty /\sqrt t$, and the number of such particles is of the order $e^{\beta_t}$. Since $Z_\infty$ is in the domain of attraction of a 1-stable distribution, this leads to the 1-stable limit.

The above-mentioned proof idea relied on the fact that a suitably shifted version of $Z_t$ is a martingale when the particles are killed at the barrier. Indeed, when restricting to the particles that do not hit the barrier, the conditional expectation of $Z_\infty$ conditioned on $\sF_t$ is equal to this shifted version of $Z_t$, allowing the use of Chebychev's inequality in order to bound the difference between $Z_t$ and $Z_\infty$, restricted to these particles. For general functionals $Z_t(F)$, this strategy breaks down and we have to compare $Z_t(F)$ with $Z_s$ at a time $s$ \emph{before} the time $t$. However, when adapting the ``barrier strategy''  mentioned above, it turns out that one cannot get optimal estimates in one step. We therefore implement a \emph{multiscale}, or \emph{bootstrap strategy} which can be described as follows: Suppose we introduce a barrier between times $t_0$ and $t$, for some $t_0 \le t$, for example $t_0 = t^a$ for some $a\in(0,1)$. We can estimate the particles not touching the barrier by first and second moments, and the particles hitting the barrier by first moments. This yields a first bound on the fluctuations of $Z_t(F)$ of order $t^{-b}$, for some $b$ depending on $t_0$. Then we can iterate: we take another $t_0'\le t$ and introduce again a barrier between times $t_0'$ and $t$. However, we can now make use of our previous estimate on fluctuations at two places: first, to get a better estimate on the conditional expectation of $Z_t(F)$ conditioned on $\sF_{t_0'}$, and second, in order to estimate more precisely the contribution to $Z_t(F)$ of the descendants of particles hitting the barrier. This is the basic strategy we implement for a finite number of steps, with variations at every step. See Section~\ref{section:strategy} for details.
%


This idea makes it clear that the fluctuations come again from the contributions to $Z_t(F) - Z_\infty \int F(x) \rho(\diff x)$ of the killed particles and helps to read the definition of the limit in Theorem~\ref{theorem-complete}.
Indeed, for $r > 0$, the number of particles hitting the barrier between times $rt$ and $(r+\diff r) t$ is approximately $\e^{\beta_t} (2\pi)^{-1/2} Z_\infty r^{-3/2} \diff r$.
Each of them contributes to $Z_\infty$ as an independent copy of $(\e^{-\beta_t}/\sqrt{t}) Z_\infty$, as mentioned above. Moreover, it contributes to $Z_t(F)$ only if $r<1$ and, in that case, one expects that the contribution is approximately an independent copy of 
\[
\frac{\e^{-\beta_t}}{\sqrt t} Z_{t-rt} \left(x \mapsto F(\sqrt{1-r}\cdot x) \right)
\simeq \frac{\e^{-\beta_t}}{\sqrt t} Z_\infty \int_0^\infty F(\sqrt{1-r}\cdot x)\,\rho(\diff x),
\]
by \eqref{eq:cv-Z_t(F)}.
Since $Z_\infty$ is in the domain of attraction of a 1-stable law, this explains that the limit can be written as an integral of the function $\mathscr R F$ against a $1$-stable random measure with control measure proportional to $Z_\infty r^{-3/2}\diff r$.
Note that the fluctuations of the front are due to the particles that come extremely low, around $\frac{1}{2} \log t$, which is far below the usual minimum of BBM at a time of order $t$.

\subsection{Organization of the paper and notation}

The paper is organized as follows. 
In Section~\ref{section:strategy}, we present the strategy for the proof of Theorem~\ref{theorem-complete}, which is obtained through several steps in Sections \ref{section:a-first-rough-approach-to-the-result} to \ref{section:proof_theorem}.
Section \ref{section:preliminary-results} is devoted to preliminary results concerning branching Brownian motion, in particular when particles are killed at 0.
Other technical results concerning the 3-dimensional Bessel process are shown in Section \ref{section:technical-results} in the appendix. Sections~\ref{section:two-constants-are-identical} and \ref{app:additive_martingale} contain the proofs of some identities involving the function $\mathscr R F$.
In the sequel, we denote by $C$ a positive constant that can change between occurrences and depend implicitly on some parameters indicated in the statement of each result. 
Moreover, for $x \in \R$, we set $x_+ = \max(x,0)$.
For $\varphi_1 \colon \R_+ \to \R$ and $\varphi_2 \colon \R_+ \to \R_+^*$, we say that $\varphi_1(t) = o(\varphi_2(t))$ as $t \to \infty$ if $\lim_{t\to\infty} \varphi_1(t)/\varphi_2(t) = 0$ and that $\varphi_1(t) = O(\varphi_2(t))$ as $t \to \infty$ if $\limsup_{t\to\infty} \abs{\varphi_1(t)}/\varphi_2(t) < \infty$.
We let $(B_t)_{t\geq0}$ denote a standard Brownian motion and $(R_t)_{t\geq 0}$ a 3-dimensional Bessel process.
For $u,v \in \T$, we say that $u \leq v$ if $v$ is a descendant of $u$.
If $n \geq 1$ is an integer, we set $\llbracket 1,n \rrbracket \coloneqq \{1,2,\dots,n\}$.

\section{Strategy of the proof}
\label{section:strategy}

The basic strategy of the proof of Theorem~\ref{theorem-complete} is a multiscale argument which we outline here. We first present in Section~\ref{subsection:scheme_barrier} a general decomposition of the BBM particles into ``bulk'' and ``fluctuations''. This decomposition is then used at various scales and, combined with first and second moment estimates as well as tail estimates of certain random variables, is used to obtain finer and finer results on the additive functional $Z_t(F)$.

\subsection{General scheme of decomposition using the barrier}
\label{subsection:scheme_barrier}

We introduce a barrier killing particles which go below a level $\gamma$ during the time interval $[t_0,\infty)$. These parameters $t_0$ and $\gamma$ will be chosen as various appropriate functions of $t$ throughout the paper. We always choose $t_0 < t$, that is the barrier is introduced before the time $t$ at which we study $Z_t(F)$.

Let $\cL^{t_0,\gamma}$ denote the stopping line of killed particles and, for $u \in \cL^{t_0,\gamma}$, $T_u$ the killing time of~$u$:
\begin{align*}
\cL^{t_0,\gamma} 
& \coloneqq 
\{u \in \T : \text{there exists } s \geq t_0 \text{ s.t.\@ } 
u \in \cN(s), X_u(s) \leq \gamma \text{ and } \forall r \in [t_0,s), X_u(r) > \gamma \}, \\
T_u
& \coloneqq 
\inf \{s \geq t_0 : u \in \cN(s) \text{ and } X_u(s) \leq \gamma \},
\end{align*}
see Figure \ref{figure:stopping-line} for an illustration.
\begin{figure}[t]
\centering
\begin{tikzpicture}[scale=1.2]
\draw[->,>=latex] (0,0) -- (8.7,0) node[right]{time};
\draw[->,>=latex] (0,-0.8) -- (0,4) node[left]{space};
\draw (0,0) node[left]{$0$};
\draw[dashed] (3,1.6) -- (0,1.6) node[left]{$\gamma$};
\draw (3,0) node[below left]{$t_0$};
\draw[dashed] (8,3.9) -- (8,-0.1) node[below]{$t$};
\draw[dashed] (4.7,1.6) -- (4.7,-0.1) node[below]{$T_u$};
\draw[NavyBlue,very thick] (3,-0.8) -- (3,1.6) -- (8,1.6);
\draw[NavyBlue,very thick,dashed] (8,1.6) -- (8.7,1.6);
\draw[NavyBlue] (3,-0.6) node[right]{$\cL^{t_0,\gamma}$};
\draw[smooth] plot coordinates 
{(0,0) (0.1,-0.04) (0.3,0.6) (0.4,0.62) (0.6,1) 
(0.7,1.05) (0.8,1.4) (0.9,1.35) (1,1.5) 
(1.1,1.6) (1.2,1.65) (1.4,1.7) (1.5,1.9) (1.8,1.85) (2.1,2.3) 
(2.3,2.4) (2.5,2.4) (2.8,2.3) (3,2.6) (3.3,2.7) (3.5,2.6) (3.8,2.9) (4,2.9) 
(4.3,2.6) (4.5,2.75) (4.8,2.65) (5,2.7) (5.4,2.4) 
(5.6,2.55) (5.8,2.5) (6,2.2) (6.2,2.25) (6.4,1.9) (6.6,1.85) (6.8,1.6) 
};
\draw[smooth,NavyBlue] plot coordinates 
{(6.8,1.6) (7,1.5) (7.3,1) (7.5,1.1) 
(7.7,1) (8,1.3)
};
\draw[smooth] plot coordinates 
{(0.6,1) (0.7,1.05) (0.9,1.2) (1.1,1.15) (1.3,1.3)
(1.5,1.15) (1.7,1.2) (1.9,1.25) (2.1,1) (2.4,1.2) (2.7,0.85) (3,0.8)}; 
\draw[smooth,NavyBlue] plot coordinates 
{(3,0.8) (3.3,0.7) (3.5,0.85) (3.7,0.9) (4,0.8) (4.3,1.1) (4.5,1.15) (4.7,1.1) (5,1.4) (5.2,1.45) (5.5,1.8) (5.7,1.75) (5.9,1.9) (6.1,1.85) (6.4,2.3) (6.6,2.35) (6.8,2.8) (7,3) 
(7.2,3.15) (7.5,3) (7.7,3.05) (8,2.6)
};
\draw[smooth] plot coordinates 
{(1,1.5) (1.2, 1.9) (1.3,1.95) (1.4,1.95) (1.5,2.15) (1.8,2.2) (2,2) (2.1,1.8) (2.3,1.85)
(2.5,2) (2.6,1.95) (2.8,2.1) (3,1.9) 
(3.15,1.95) (3.3,2.1) (3.5,2.2) (3.7,2) (4,1.9) (4.1,2) (4.3,1.95) (4.7,1.6)}; 
\draw[smooth,NavyBlue] plot coordinates 
{(4.7,1.6) (4.8,1.5) (5.1,1.3) (5.2,1.35) (5.4,1.1) (5.5,1.05) 
(5.8,1.1) (6,0.85) (6.2,0.9) (6.4,1.15) (6.6,1.1) (6.9,1.15)  
(7.2,1.6) (7.4,1.7) (7.6,2.1) (7.8,2.1) (8,2.2)};
\draw[smooth] plot coordinates 
{(2.1,2.3) (2.3,2.35) (2.5,2.7) (2.7,2.75) (3,3.3) (3.2,3.25) (3.5,3.8) 
};
\draw[smooth] plot coordinates 
{(3,1.9) (3.1,1.85) (3.3,1.75) (3.4,1.8) (3.6,1.6)}; 
\draw[smooth,NavyBlue] plot coordinates 
{(3.6,1.6) (3.8,1.5) (3.9,1.4) (4.2,1.55) (4.3,1.55) (4.4,1.75) (4.55,1.8) (4.8,2.2) 
(5,2.2) (5.3,2.3) (5.5,2.1) (5.7,2.15) (6.1,2.6) (6.3,2.65) (6.6,2.4) (6.8,2.45) (6.9,2.5) (7.1,2.3) (7.3,2.3) (7.5,1.9) (7.8,1.95) (8,1.8)
};
\draw[smooth] plot coordinates 
{(4,2.9) (4.3,2.75) (4.5,2.9) (4.8,2.95) (5,3.4) (5.3,3.35) 
(5.6,3.8) 
};
\draw[smooth,NavyBlue] plot coordinates 
{(4.8,2.2) (5,2.4) (5.2,2.45) (5.5,2.8) (5.7,2.75) (6,3.3) (6.2,3.4) (6.5,3.3) (6.7,3.6) (6.9,3.55) (7.1,3.8) 
};
\draw[smooth] plot coordinates 
{(5.3,3.35) (5.4,3.4) (5.7,3.3) (6,3.5) (6.15,3.3) (6.4,3.25) (6.5,3.15) (6.8,3.25)
(7,3.4) (7.2,3.4)  
(7.5,3.6) (7.8,3.45) (8,3.5)
};
\draw[smooth] plot coordinates 
{(5.4,2.4) (5.6,2.4) (5.8,2.7) (6,2.7) (6.2,2.5) (6.4,2.75) (6.7,2.6) (6.9,2.7) (7.2,2.65) (7.4,2.4) (7.6,2.6) 
(7.8,2.7) (8,2.7)
};
\draw[smooth,NavyBlue] plot coordinates 
{(6.9,1.15) (7,1.2) (7.2,1.2) (7.5,1.5) (7.7,1.6) (7.9,1.9) (8,1.9)
};
\draw[smooth] plot coordinates 
{(7.2,3.4) (7.5,3.5) (7.8,3.2) (8,3.2)
};
\draw[smooth,NavyBlue] plot coordinates 
{(7,3) (7.2,3.05) (7.4,3.3) (7.7,3.4) (8,3.8) 
};
\draw[smooth,NavyBlue] plot coordinates 
{(7.5,1.1) (7.7,1.1) (8,0.8)
};
\draw[smooth] plot coordinates 
{(7.6,2.6) (7.8,2.65) (7.9,2.5) (8,2.45)
};
\draw (3,0.8) node{$\bullet$};
\draw (3.6,1.6) node{$\bullet$};
\draw (4.7,1.6) node{$\bullet$};
\draw (4.7,1.6) node[below left]{$u$};
\draw (6.8,1.6) node{$\bullet$};
\end{tikzpicture}
\caption{Representation of a BBM with binary branching. 
The killing barrier that defines $\cL^{t_0,\gamma}$ is the thick blue line. 
At time $t$, only the black particles contribute to $\widetilde{Z}^{t_0,\gamma}_t (F,\gamma)$.
For each $u$ in the stopping line, the blue particles descending from $u$ at time $t$ contribute to $\Omega^{(u)}_t$.}
\label{figure:stopping-line}
\end{figure}
Let $\sF_{\cL^{t_0,\gamma}}$ be the $\sigma$-algebra associated with the stopping line $\cL^{t_0,\gamma}$, which is generated by all the events depending only on the behavior of particles before they hit the barrier (see Chauvin \cite{chauvin91} for formal definitions).

As in \cite{maillardpain2019}, which studied the fluctuations of the derivative martingale, it will be more convenient to consider a translated version of $Z_t(F)$ defined by
\begin{align} \label{eq:def_Z_t(F,gamma)}
Z_t (F,\gamma)
& \coloneqq 
\sum_{u \in \cN(t)} (X_u(t)-\gamma)_+ \e^{-X_u(t)} 
	F \left( \frac{X_u(t)-\gamma}{\sqrt{t}} \right),
\end{align}
recalling that we write $x_+ = \max(x,0)$.
We decompose this quantity into two parts. 
Firstly, we consider the part where only particles which have not been killed by the barrier are taken into account:
\begin{align} \label{eq:def_Z_t^t_0,gamma(F,gamma)}
\widetilde{Z}_t^{t_0,\gamma} (F,\gamma)
& \coloneqq 
\sum_{u \in \cN(t)} (X_u(t)-\gamma)_+ \e^{-X_u(t)} 
	F \left( \frac{X_u(t)-\gamma}{\sqrt{t}} \right)
	\1_{\forall s \in [t_0,t], X_u(s) > \gamma}.
\end{align}
Secondly, we consider the contribution $\Omega^{(u)}_t$ of the progeny of each killed particle $u \in \cL^{t_0,\gamma}$ in $Z_t(F,\gamma)$, defined by
\begin{align} \label{eq:def_Omega}
\Omega^{(u)}_t \coloneqq 
\sum_{v \in \cN(t) \text{ s.t.\@ } u \leq v} 
(X_v(t)-\gamma)_+ \e^{-X_v(t)} 
F \left( \frac{X_v(t)-\gamma}{\sqrt{t}} \right),
\qquad \text{if } T_u \leq t,
\end{align}
where ``$u \leq v$'' means that $u$ is an ancestor of $v$,
and $\Omega^{(u)}_t \coloneqq 0$ if $T_u>t$.
Then, we have the following decomposition
\begin{align} 
Z_t(F,\gamma) 
& = \widetilde{Z}_t^{t_0,\gamma} (F,\gamma) 
+ \sum_{u \in \cL^{t_0,\gamma}} \Omega^{(u)}_t. \label{eq:decompo-Z_t(F,gamma)}
\end{align}
The decomposition~\eqref{eq:decompo-Z_t(F,gamma)} will be used several times throughout the paper, in the spirit of a multiscale argument.
We repetitively argue that if the barrier is chosen high enough, then $\widetilde{Z}_t^{t_0,\gamma} (F,\gamma)$ is well-concentrated around its conditional expectation given~$\sF_{t_0}$, using a second moment calculation.
Then, we either bound or estimate precisely the contributions of killed particles. The details vary depending on the scale.

\subsection{Main steps of the proof}
\label{subsection:main_steps}

\begin{figure}[h!]
	\centering
	\begin{tikzpicture}
		\small
		\def\SecTro{15.6}
		\node[draw,rounded corners,text width=2.3cm,text centered] (31) at (9,\SecTro) 
		{Proposition~\ref{prop:control-rate-of-cv-Z} \\ \footnotesize Rate of convergence for $Z_t$};
		\node[draw,rounded corners,text width=1.7cm,text centered] (32) at (-4,\SecTro) 
		{Lemma~\ref{lem:first-moment-Z_s(F,Delta))} \\ \footnotesize 1st moment estimates of $\widetilde{Z}_s (F,\Delta)$};
		\node[draw,rounded corners,text width=1.8cm,text centered] (34) at (-1.3,\SecTro) 
		{Lemma~\ref{lem:second-moment-Z_s(F,t,Delta)} \\ \footnotesize 2nd moment bounds of $\widetilde{Z}_s (F)$ for regular $F$};
		\node[draw,rounded corners,text width=1.8cm,text centered] (35) at (1.5,\SecTro) 
		{Lemma~\ref{lem:f_alpha} \\ \footnotesize 2nd moment bound of $\widetilde{Z}_s (F)$ for diverging $F$};
		\node[draw,rounded corners,text width=3.85cm,text centered] (36) at (5.2,\SecTro+.9) 
		{Lemma~\ref{lem:first-moment-number-of-killed-particles} \\ \footnotesize 1st moment bound of nb of particles killed on a line};
		\node[draw,rounded corners,text width=3.85cm,text centered] (37) at (5.2,\SecTro-.9) 
		{Lemma~\ref{lem:second-moment-number-of-killed-particles} \\ \footnotesize 2nd~moment bound of nb of particles killed on a line};
		
		\def\SecQua{12.2}
		\node[draw,rounded corners,text width=2.45cm,text centered] (41) at (-3.5,\SecQua) 
		{Lemma~\ref{lem:first-second-moment-Z_s^(gamma_s)(F,Delta_s)} \\ \footnotesize Conditional 1st and 2nd moment of $\widetilde{Z}_t^{t^a,\gamma} (F,\gamma)$ for $a\in(0,1)$};
		\node[draw,rounded corners,text width=2.05cm,text centered] (42) at (0.2,\SecQua) 
		{Lemma~\ref{lem:control-rate-of-CV-killed-Z_s^(s,gamma_s)(F,Delta_s)} \\ \footnotesize Concentration of $\widetilde{Z}_t^{t^a,\gamma} (F,\gamma)$ for $a\in(0,1)$};
		\node[draw,rounded corners,text width=2.5cm,text centered] (43) at (4,\SecQua) 
		{Lemma~\ref{lem:control-fluctuations-Z_s(F,t,Delta_s)} \footnotesize Concentration of the contribution of killed particles};
		\node[draw=BurntOrange,very thick,text width=3cm,text centered] (21) at (9,\SecQua) 
		{Proposition~\ref{prop:rate-of-CV-Z_t(F,Delta,t)} Corollary~\ref{cor:rate_of_convergence_as_expectation} \\ \footnotesize Rate of convergence for $Z_t(F,\Delta)$ for regular $F$\vspace{-.25cm} \rule{3cm}{0.4pt} $t_0=t^{2/3}$, $\gamma = \frac{1}{6} \log t$};
		
		\def\SecCin{8.5}
		\node[draw,rounded corners,text width=3cm,text centered] (52) at (8.2,\SecCin) 
		{Lemma~\ref{lem:convergence-of-fluctuations-part-3-bis}  Corollary~\ref{cor:convergence-of-fluctuations-part-2} \\ \footnotesize Convergence of sums along a stopping line starting at $t^a$};
		\node[draw,rounded corners,text width=2.5cm,text centered] (51) at (2,\SecCin) 
		{Lemma~\ref{lem:convergence-of-fluctuations} \\ \footnotesize Convergence of the contribution of killed particles};
		\node[draw=BurntOrange,very thick,text width=3.35cm,text centered] (22) at (-3.1,\SecCin) 
		{Proposition~\ref{prop:convergence-with-translation-gamma_t} \\ \footnotesize 1st version of Thm~\ref{theorem-complete}: one-dimensional, non-conditional, with a shift and for regular $F$\vspace{-.25cm} \rule{3.35cm}{0.4pt} $t_0=t^a$, $a\in (1/3,1/2)$ $\gamma = \frac{1}{2} \log t + \beta_t$};
		
		\def\SecSix{5}
		\node[draw,rounded corners,text width=2.9cm,text centered] (61) at (-.9,\SecSix) 
		{Proposition~\ref{prop:cv_rate_for_bad_functions} Corollary~\ref{cor:rate_of_convergence_as_expectation_for bad_functions} \\ \footnotesize Rate of convergence for $Z_t(F,\Delta)$ for diverging $F$};
		\node[draw,rounded corners,text width=1.7cm,text centered] (62) at (-4,\SecSix) 
		{Lemma~\ref{lem:cv_rate_for_bad_functions_1} \\ \footnotesize Truncation procedure};
		\node[draw,rounded corners,text width=2.4cm,text centered] (64) at (2.9,\SecSix) 
		{Proposition~\ref{prop:from-translation-gamma_t-to-0} \\ \footnotesize Effect of the shift in $Z_t(F,\Delta)$};
		\node[draw,rounded corners,text width=2.2cm,text centered] (65) at (6.4,\SecSix) 
		{Lemma~\ref{lem:sums_along_stopping_line} Corollary~\ref{cor:convergence-of-fluctuations} \\ \footnotesize Convergence of sums along a stopping line starting at $ht$};
		\node[draw,rounded corners,text width=2.2cm,text centered] (67) at (9.4,\SecSix) 
		{Lemma~\ref{lem:sums_along_stopping_line_2} Corollary~\ref{cor:convergence-of-fluctuations-2} \\ \footnotesize Convergence of sums along a stopping line for diverging $F$};
		
		\def\SecSep{2}
		\node[draw=BurntOrange,very thick,text width=3.3cm,text centered] (72) at (2,\SecSep) 
		{Proposition~\ref{prop:weaker_version_theorem_2} \footnotesize Thm~\ref{theorem-2} for regular $F$\vspace{-.25cm} \rule{3.3cm}{0.4pt} $t_0=t$, $\gamma = \frac{1}{2} \log t + \beta_t$};
		\node[draw=BurntOrange,very thick,text width=2.9cm,text centered] (73) at (7,\SecSep) 
		{Proposition~\ref{prop:regularize}  \footnotesize Regularization of $F$\vspace{-.25cm} \rule{2.9cm}{0.4pt} $t_0=(1-\varepsilon)t$, $\gamma = c \log t$, $c>1/2$};
		%
		
		\node[draw,very thick,diamond,aspect=2.5,text width=1.8cm,text centered] (71) at (4,-.3) 
		{Theorem~\ref{theorem-2} \footnotesize};
		\node[draw,very thick,diamond,aspect=2.5,text width=1.8cm,text centered] (12) at (-1,-.3) 
		{Theorem~\ref{theorem-complete} \footnotesize};
		
		\draw[->,>=latex,thick] (21) -- (51);
		\draw[->,>=latex,thick] (21) -- (52);
		\draw[->,>=latex,thick] (21) to[bend left=10] (61.33);
		\draw[->,>=latex,thick,rounded corners=8pt] (21.220) to (4.8,9.4) to (4.8,4.2) to (72);
		\draw[->,>=latex,thick,rounded corners=8pt] (22.280) to (-2.8,2) to (12);
		\draw[->,>=latex,thick] (31) -- (21);
		\draw[->,>=latex,thick] (32) -- (34);
		\draw[->,>=latex,thick] (32.300) -- (43.155);
		\draw[->,>=latex,thick] (32) -- (41);
		\draw[->,>=latex,thick] (34) -- (35);
		\draw[->,>=latex,thick] (34) -- (41);
		\draw[->,>=latex,thick] (34.260) to[out=270,in=120] (0.3,8) to[out=300,in=90] (72.135);
		\draw[->,>=latex,thick,rounded corners=8pt] (35.283) to (1.8,10.4) to (4.5,10.4) to (4.5,3.8) to (73.140);
		\draw[->,>=latex,thick] (36) -- (37);
		\draw[->,>=latex,thick] (37.310) -- (52.120);
		\draw[->,>=latex,thick] (37.295) -- (65);
		\draw[->,>=latex,thick] (41) -- (42);
		\draw[->,>=latex,thick] (42) to[bend left=25] (21);
		\draw[->,>=latex,thick] (42) -- (22);
		\draw[->,>=latex,thick] (43) -- (21);
		\draw[->,>=latex,thick] (51) -- (22);
		\draw[->,>=latex,thick] (52) -- (51);
		\draw[->,>=latex,thick] (61) -- (64);
		\draw[->,>=latex,thick] (61) to[bend left=25] (65);
		\draw[->,>=latex,thick] (61) -- (73.160);
		\draw[->,>=latex,thick] (62) -- (61);
		\draw[->,>=latex,thick] (64.240) -- (72);
		\draw[->,>=latex,thick,rounded corners=8pt] (64) to (-1,2.5) to (12);
		\draw[->,>=latex,thick] (64.280) -- (73.151);
		\draw[->,>=latex,thick] (65) -- (67);
		\draw[->,>=latex,thick] (65.225) -- (72.25);
		\draw[->,>=latex,thick] (67) -- (73);
		\draw[->,>=latex,thick] (71) -- (12);
		\draw[->,>=latex,thick] (72) -- (71);
		\draw[->,>=latex,thick] (73) -- (71);	
	\end{tikzpicture}
	\caption{\label{fig:dependency_graph}Dependencies between the intermediate results leading to the proof of Theorem~\ref{theorem-complete}. Arrows indicate that a result is used in the proof of another result. The proofs of the results in the rectangular, orange boxes implement the strategy outlined in Section~\ref{subsection:scheme_barrier}, with $t_0$ and $\gamma$ as indicated, where $\beta_t \to \infty$ with $\beta_t = o(\log t)$.
	To avoid overloading the diagram, Lemma~\ref{lem:tightness} is not included because this basic result, as well as the inequalities in its proof, are used widely throughout the paper.
	Similarly, results in Appendix \ref{section:technical-results} are not included.}
\end{figure}

We now detail the intermediate results leading to Theorem~\ref{theorem-complete}. Figure~\ref{fig:dependency_graph} provides a graphical description of the dependencies between the results.

In a first part of this paper, we work under the following different assumptions for the function $F \colon \R \to \R$: for some $\kappa \geq 0$, we assume
\begin{enumerate}[label=(A\arabic*$_\kappa$),leftmargin=*]
	\item For any $x > 0$, $\abs{F(x)} \leq \e^{\kappa x}$; \label{ass1}
	\item For any $0 < y \leq x$, $\abs{F(x)-F(y)} \leq (x-y) \e^{\kappa x}$. \label{ass2}
\end{enumerate}
Note that we do not allow divergence at 0 of the function $F$. 
Moreover, it will be convenient to have results which hold uniformly for all functions $F$ satisfying these assumptions for a given~$\kappa$.

A consequence of Assumptions~\ref{ass1} and \ref{ass2} is the following inequality: for some $C = C(\kappa)$,
\begin{equation}
	\label{eq:RF_bound}
	|\mathscr R F(r)| \le C(r\wedge 1), \quad r>0.
\end{equation}
In particular, $\mathscr RF \in \mathcal G(r^{-3/2}\diff r)$.

\paragraph{First step.}
The first result we establish is an estimate on the rate of convergence of $Z_t (F,\Delta)$ towards $\rho(F) Z_\infty$, for a shift $\Delta$ of order at most $\log t$.
It is proved in Section \ref{section:a-first-rough-approach-to-the-result}.
\begin{prop} \label{prop:rate-of-CV-Z_t(F,Delta,t)}
Let $\kappa \geq 0$ and $K>0$.
There exist $C = C(\kappa,K) >0$ and $t_0 = t_0(\kappa,K) > 0$, such that
for any function $F \colon \R \to \R$ satisfying Assumptions \ref{ass1}-\ref{ass2}, $t \geq t_0$, $\Delta \in [0,K \log t]$ and $\delta > 0$, we have 
\begin{align*}
\Pp{\abs{Z_t(F,\Delta) - \rho(F) Z_\infty} \geq \delta}
& \leq t^{-K} + \frac{C(\log t)^2}{\delta t^{1/3}}.
\end{align*}
\end{prop}
The proof relies on the scheme presented in Section \ref{subsection:scheme_barrier} with $t_0 = t^{2/3}$ and $\gamma = \frac{1}{6} \log t$.
Moreover, we work with the centered function $F - \rho(F)$, which ensures that the quantity $\widetilde{Z}_t^{t_0,\gamma} (F- \rho(F),\gamma)$ has small conditional expectation and variance given $\sF_{t_0}$.
Contributions of killed particles are bounded crudly via a first moment argument. The control of $Z_t-Z_\infty$ is contained in Proposition~\ref{prop:control-rate-of-cv-Z}, which follows from \cite{maillardpain2019}.

\paragraph{Second step.}
In Section~\ref{section:using-previous-results}, we prove the following fluctuation result.
Compared with Theorem~\ref{theorem-complete}, it is a one-dimensional convergence which does not hold conditionally on $\sF_{\ep t}$, it deals with a translated version of $Z_t(F)$, and requires different assumptions on $F$. Recall the definitions of $\mathscr R F$ and $M_{Z_\infty}$ from \eqref{eq:def R F} and \eqref{eq:def_M_Z_infty}.
\begin{prop} \label{prop:convergence-with-translation-gamma_t}
	Let $\gamma = \frac{1}{2} \log t + \beta_t$, where $\beta_t \to \infty$ and $\beta_t = o(\log t)$ as $t \to \infty$.
	Let $\kappa \geq 0$.
	For any function $F \colon \R \to \R$ satisfying \ref{ass1}-\ref{ass2}, we have 
	\begin{align*}
	\sqrt{t} 
	\left( Z_t (F,\gamma) - \rho(F) Z_t(1,\gamma)
	- \frac{\beta_t}{\sqrt{t}}  \int_0^1 \mathscr R F(r) \frac{Z_\infty}{\sqrt{2\pi}} \frac{1}{r^{3/2}}\diff r
	\right)
	\xrightarrow[t\to\infty]{\text{(law)}} \int_0^1 \mathscr R F(r) M_{Z_\infty}(\diff r).
	\end{align*}
\end{prop}
Proposition~\ref{prop:convergence-with-translation-gamma_t} is proved using again the scheme of Section \ref{subsection:scheme_barrier}, with $t_0 = t^{a}$ with $a\in (1/3,1/2)$ and $\gamma = \frac{1}{2} \log t + \beta_t$ and working with the centered function $F - \rho(F)$.
Proposition~\ref{prop:rate-of-CV-Z_t(F,Delta,t)} is used in order to improve the estimates of the fluctuations, i.e.\@ the contributions of killed particles.
Indeed, conditionally on  $\sF_{\cL^{t_0,\gamma}}$, the contributions of killed particles are independent random variables, which can be approximated using Proposition \ref{prop:rate-of-CV-Z_t(F,Delta,t)} by independent rescaled copies of $Z_\infty$.
The convergence in law of the sum of these independent variables is then handled via a calculation involving characteristic functions.

\paragraph{Third step.} Having obtained Proposition~\ref{prop:convergence-with-translation-gamma_t}, in order to obtain  Theorem~\ref{theorem-complete} it remains to:
\begin{itemize}
	\item extend the result to functions $F$ diverging at 0,
	\item remove the shift $\gamma$,
	\item allow for conditioning with respect to $\sF_{\ep t}$, with $\ep\to 0$,
	\item prove a functional convergence.
\end{itemize}
While it would have been possible to attack these points independently, we have deemed more efficient to treat them in parallel, in order to avoid having to reprove several results under different assumptions.

To this end, Section \ref{section:preliminary_results_towards_the_final_theorem} gathers some technical results which do not follow the general scheme outlined in Section~\ref{subsection:scheme_barrier} above. 
More precisely:
\begin{itemize}
	\item Proposition \ref{prop:cv_rate_for_bad_functions} in Section~\ref{subsection:bad_functions} extends Proposition \ref{prop:rate-of-CV-Z_t(F,Delta,t)} to functions $F$ diverging at 0, via a truncation argument of the function near 0.
	\item Proposition \ref{prop:from-translation-gamma_t-to-0} in Section~\ref{subsection:effect_of_a_shift} compares precisely $Z_t(F,\gamma)$ and $Z_t(F)$: their difference is of order $\gamma/\sqrt{t}$ and is responsible for the $(\log t)/\sqrt{t}$ corrective term\footnote{\label{foot:shift}This explains in particular why there is no such term for the fluctuations of the additive martingale in Corollary \ref{cor:fluctuations-for-W}: indeed $W_t$ is invariant by shifts (that is if $F(x) = 1/x$, $Z_t(F,\gamma) = Z_t(F)$, up to the contribution of particles below $\gamma$ at time $t$).} appearing in Theorem~\ref{theorem-complete}. 
	\item Section~\ref{subsection:sums_along_stopping_line_2} contains two results (Lemma~\ref{lem:sums_along_stopping_line} and Lemma~\ref{lem:sums_along_stopping_line_2}) which will be used to estimate the contribution of particles hitting the barrier after time $ht$, for some fixed $h>0$. This prepares conditioning with respect to $\sF_{\ep t}$, with $\ep\to 0$, and also enables to treat functions $F$ diverging at 0 (Lemma~\ref{lem:sums_along_stopping_line_2}).
\end{itemize}
Note that the proofs of Proposition~\ref{prop:from-translation-gamma_t-to-0} and Lemma~\ref{lem:sums_along_stopping_line} rely on Proposition~\ref{prop:cv_rate_for_bad_functions}, \emph{even for $F$ which does not diverge at 0}, because we apply Proposition~\ref{prop:cv_rate_for_bad_functions} to auxiliary functions diverging at 0.

Theorem \ref{theorem-complete} is proved in Section \ref{section:proof_theorem}. We first introduce for every function $F$ and every $h\in[0,1)$ a certain function $F^h$, which can be seen as an averaged version of $F$. We then state Theorem~\ref{theorem-2}, which provides a limit result for $Z_{at}(F) - Z_t(F^{1/a})$, $a>1$, under the same assumptions on $F$ as in Theorem~\ref{theorem-complete}. We first prove it for functions $F$ which do not diverge at $0$ (Proposition \ref{prop:weaker_version_theorem_2}). For this, we apply the scheme of Section \ref{subsection:scheme_barrier}, with $t_0 = t$ and $\gamma = \frac{1}{2} \log t + \beta_t$ where $\beta_t \to \infty$ slowly, to decompose $Z_{at}(F,\gamma)$ simultaneously for several values of $a$ and follow the proof of Proposition~\ref{prop:convergence-with-translation-gamma_t} to deal with the contributions of killed particles, relying on the auxiliary results from Section~\ref{section:preliminary_results_towards_the_final_theorem}.
To generalize the result to functions $F$ diverging as $x^{-\alpha}$ at~0, we prove in Proposition~\ref{prop:regularize} that $Z_t(F,\gamma)$ is close to $Z_{(1-\ep)t}(F^{1-\ep},\gamma)$ (up to some corrective term), applying again the scheme of Section \ref{subsection:scheme_barrier}, but with $t_0 = (1-\ep)t$ and $\gamma = c \log t$, where $c$ has to be chosen closer and closer to $3/2$ as $\alpha$ approaches 2. Here, the contribution of killed particles is negligible in the limit $t\to\infty$ and then $\ep \to 0$. As the function $F^{1-\ep}$ does not diverge at zero, we can apply Proposition \ref{prop:weaker_version_theorem_2} to this function to obtain Theorem~\ref{theorem-2}.

In order to deduce Theorem~\ref{theorem-complete} from Theorem~\ref{theorem-2}, we write the difference $Z_{at}(F) - \rho(F)Z_\infty$ as the sum of three quantities as follows:
\[
Z_{at}(F) - \rho(F)Z_\infty = (Z_{at}(F) -Z_{\ep t}(F^{\ep/a})) + (\rho(F)Z_{\ep t} - \rho(F)Z_\infty) + (Z_{\ep t}(F^{\ep/a}) - \rho(F)Z_{\ep t}). 
\]
We apply Theorem~\ref{theorem-2} to the first and second summand (with $a=\infty$ for the second). Proposition~\ref{prop:convergence-with-translation-gamma_t} is crucially used in order to show that the third summand is negligible as $t\to\infty$, then $\ep\to 0$. Theorem~\ref{theorem-complete} follows.

\section{Preliminary results}
\label{section:preliminary-results}

\subsection{Some results concerning branching Brownian motion}

In the sequel, we allow the BBM to start at an arbitrary point $x \in \R$, in which case we write $\P_x$ and $\E_x$ instead of $\P$ and $\E$.
Under our settings for the BBM, the \textit{many-to-one formula} is stated as follows (see e.g.\@ \cite[Proposition~4.1]{maillardpain2019}):
for any $x \in \R$, $t \geq 0$ and any measurable function $\Theta \colon \cC([0,t]) \to \R_+$, where $\cC([0,t])$ denotes the set of continuous function from $[0,t] \to \R$, we have
\begin{align} \label{eq:many-to-one}
\Eci{x}{\sum_{u\in\cN(t)} \e^{-X_u(t)} \Theta(X_u(s), s\in[0,t])}
& = \e^{-x} \Eci{x}{\Theta(B_s, s\in[0,t])},
\end{align}
where $(B_s)_{s \geq 0}$ denotes a standard Brownian motion, starting from $x$ under $\P_x$.
We use repetitively the following consequence, using the link \eqref{eq:link-between-R-and-B} between the Brownian motion staying positive and the 3-dimensional Bessel process: for any $x >0$, $t \geq 0$ and any measurable function $\varphi \colon \R \to \R_+$, we have 
\begin{align} \label{eq:many-to-one-bessel}
\Eci{x}{\sum_{u\in\cN(t)} \e^{-X_u(t)} 
	\varphi(X_u(t)) \1_{\forall s \in [0,t], X_u(s) > 0}}
= x \e^{-x} \Eci{x}{\frac{\varphi(R_t)}{R_t}}
= x \e^{-x} \Eci{x/\sqrt{t}}{\frac{\varphi(\sqrt{t} R_1)}{\sqrt{t} R_1}},
\end{align}
where $(R_s)_{s \geq 0}$ denotes a 3-dimensional Bessel process, starting from $x$ under $\P_x$, and we used its scaling property in the second equality.
Since we assumed that $\E[L^2] < \infty$, the \textit{many-to-two formula} holds (see e.g.\@ \cite[Lemma~10]{bramson78}). We only state it here in a special case, with a barrier at zero: for any $x >0$, $t \geq 0$ and any measurable function $\varphi \colon \R \to \R_+$, we have
\begin{align}
& \Eci{x}{\left( \sum_{u\in\cN(t)} \e^{-X_u(t)} \varphi(X_u(t)) 
	\1_{\forall s \in [0,t], X_u(s) > 0} \right)^2} \nonumber \\
\begin{split}
& = K \e^{-x} \int_0^t \diff r 
	\int_0^\infty \Eci{y}{\sum_{u\in\cN(t-r)} \e^{-X_u(t-r)} \varphi(X_u(t-r)) 
	\1_{\forall s \in [0,t-r], X_u(s) > 0}}^2 
	\e^y q_r(x,y) \diff y \\
& \relphantom{=} {} 
+ \e^{-x} \int_0^\infty \e^{-y} \varphi(y)^2 q_t(x,y) \diff y,
\end{split} \label{eq:many-to-two-formula}
\end{align}
where $q_r(x,y) \coloneqq (2 \pi r)^{-1/2} (\e^{-(x-y)^2/2r} - \e^{-(x+y)^2/2r})$ for any $x,y > 0$ and $K=\E[L(L-1)]$. 
Note that $q_r(x,\cdot)$ is the density of a Brownian motion at time $r$ starting from $x$ and killed at 0.

We now state two following useful bounds concerning the minimum of the BBM.
The first one concerns the global minimum of the BBM and can be found in \cite[Eq.\@ (D.1)]{maillardpain2019}: for any $M > 0$, we have
\begin{equation} \label{eq:global-min-of-the-BBM}
\Pp{\exists s \geq 0, \min_{u \in \cN(s)} X_u(s) \leq - M}
\leq \e^{-M}.
\end{equation} 
The second one deals with the minimum at a given time, see \cite[Proposition 3]{bramson78}: there exists $C>0$ such that, for any $s \geq 2$ and $x \in (-\infty,\sqrt{s}]$,
\begin{equation} \label{eq:local-min-of-the-BBM}
\Pp{\min_{u \in \cN(s)} X_u(s) \leq \frac{3}{2} \log s - x}
\leq C (1+x^2) \e^{-x}.
\end{equation}
We conclude this subsection with the following concentration result for the derivative martingale, which is a slight modification of \cite[Proposition 1.5]{maillardpain2019}.
\begin{prop} \label{prop:control-rate-of-cv-Z}
	For any $K>0$, there exists $C = C(K) > 0$ such that, for any $\delta>0$ and $t \geq 2$, 
	\begin{equation} \label{eq:control-rate-of-cv-Z}
	\Pp{\abs{Z_t-Z_\infty} \geq \delta}
	\leq t^{-K} + \frac{C (\log t)^2}{\delta \sqrt{t}}.
	\end{equation}
\end{prop}
\begin{proof}
	Proposition 1.5 in \cite{maillardpain2019} is restricted to the case $\delta \in (0,1)$, but we explain here how to modify its proof to establish the desired result.
	Equations (D.3)-(D.4)-(D.5) in \cite{maillardpain2019} hold for any $\delta > 0$ with an additional term $t^{-K}$ on the right-hand side and a constant $C$ depending on~$K$: this is proved in the same way as in \cite{maillardpain2019}, but with $L= K \log t$.
	Then the result follows by replacing Equation (D.6) in \cite{maillardpain2019} by the inequality
	\begin{equation} \label{eq:newD6}
	\Pp{F_{\mathrm{bad}}^{t,\gamma_t} \geq \delta}
	\leq t^{-K} + \frac{C}{\delta \sqrt{t}},
	\end{equation}
	where with notation from the present paper, 
	\[
	F_{\mathrm{bad}}^{t,\gamma_t}
	\coloneqq \sum_{u\in\cN(t)} \e^{-X_u(t)} \1_{X_u(t) \leq \log t} Z_\infty^{(u)},
	\]
	with $Z_\infty^{(u)}$ the limit of the derivative martingale of the BBM rooted at $u$ (see \eqref{eq:def_Z_infty^u} with here $T_u = t$).
	We now have to prove~\eqref{eq:newD6}.
	For $M >0$, let $A_M \coloneqq \{\min_{s\geq 0} \min_{u \in \cN(s)} X_u(s) > - M \}$.
	Then, we have
	\begin{equation} \label{eq:Fbad}
	\Ec{F_{\mathrm{bad}}^{t,\gamma_t} \1_{A_M}}
	\leq \Ec{ \sum_{u\in\cN(t)} \e^{-X_u(t)} \1_{X_u(t) \leq \log t} \1_{\forall s \in [0,t], X_u(s) > - M}}
	\cdot \Ec{ Z_\infty \1_{A_{M+\log t}}}.
	\end{equation}
	On the one hand, using the many-to-one formula \eqref{eq:many-to-one-bessel}, the first expectation on the right-hand side of \eqref{eq:Fbad} equals 
	\begin{equation}  \label{eq:Fbad2}
	M \Eci{M}{\frac{1}{R_t} \1_{R_t \leq M+\log t}}
	\leq \frac{M}{\sqrt{t}} \Ec{\frac{1}{R_1} \1_{R_1 \leq (M+\log t)/\sqrt{t}}}
	\leq \frac{C M (M+\log t)^2}{t^{3/2}},
	\end{equation}
	using, in the first inequality, the scaling property of the Bessel process and its monotonicity w.r.t.\@ the inital condition and, in the second inequality, its density given in \eqref{eq:density-of-R_1-under-P}.
	On the other hand, the following facts have been established in \cite[Theorem 9, Eq.\@ (20), Theorem 13(ii)]{kyprianou2004}: on the event $A_M$, we have a.s.\@
	\[
		Z_\infty 
		= \lim_{s\to \infty} \sum_{u\in\cN(s)} (X_u(s)+M)\e^{-X_u(s)} \1_{\forall r \in [0,s], X_u(r) > - M},
	\]
	and the quantity inside the limit on the right-hand side is a martingale with mean $M$ and converging in $L^1$. It follows that $\Ec{Z_\infty \1_{A_M}} \leq M$ for any $M > 0$.
	Combining this with \eqref{eq:Fbad} and \eqref{eq:Fbad2}, we get $\E[F_{\mathrm{bad}}^{t,\gamma_t} \1_{A_M}] \leq C (M+\log t)^4 t^{-3/2}$.
	Taking $M = K \log t$ and noting that $\P(A_M^c) \leq t^{-K}$ by \eqref{eq:global-min-of-the-BBM}, this shows \eqref{eq:newD6} and concludes the proof.
\end{proof}

\subsection{Branching Brownian motion killed at 0}
\label{subsection:BBM_killed-at-0}

In this subsection, we prove preliminary results controlling first and second moments for the following functional of BBM killed at 0, defined for $\Delta \in \R$ and $s \geq 0$ by
\begin{align}
\widetilde{Z}_s (F,\Delta)
& \coloneqq \sum_{u \in \cN(s)} (X_u(s)-\Delta)_+ \e^{-X_u(s)} 
	F \left( \frac{X_u(s)-\Delta}{\sqrt{s}} \right)
	\1_{\forall r \in [0,s], X_u(r)>0}. 
	\label{eq:def-tildeZ_s(F,Delta)}
\end{align}
Note that, with notation from Section \ref{subsection:scheme_barrier}, we have $\widetilde{Z}_s (F,\Delta) = \widetilde{Z}_s^{0,0} (F,\Delta)$.
If $\Delta = 0$, we simply write $\widetilde{Z}_s (F) \coloneqq \widetilde{Z}_s (F,0)$.

Moreover, for any function $F \colon \R \to \R$ and $h \in [0,1)$, we introduce the auxiliary function 
\begin{equation}
	\label{eq:def_F_h}
	F_h(x) \coloneqq F\left(\sqrt{1-h}\cdot x\right),\quad x\in\R.
\end{equation}
This transformed function will appear repetitively for the following reason:
the contribution of the progeny of a particle $u \in \cN(ht)$ to $\widetilde{Z}_t(F,\Delta)$, given that $X_u(ht) = x$, has the same distribution as $\widetilde{Z}_{(1-h)t}(F_h,\Delta)$ under $\P_x$.
For later reference, we also note that \eqref{eq:def R F} becomes
\begin{equation}
	\label{eq:R F with F_h}
	\mathscr R F(r) = \rho(F_r)\1_{r<1} - \rho(F),\quad r>0.
\end{equation}
\begin{lem} \label{lem:first-moment-Z_s(F,Delta))}
Let $\kappa \geq 0$.
There exists a constant $C = C(\kappa) >0$, such that
for any function $F \colon \R \to \R$ satisfying Assumptions \ref{ass1}-\ref{ass2} and any $x>0$, $h \in [0,1)$, $s > 0$ and $\Delta \geq 0$, 
\begin{align*}
\abs{\Eci{x}{\widetilde{Z}_s (F_h-\rho(F),\Delta)}}
& \leq C x \e^{-x} 
	\left(
	\left(\frac{x^2}{s} \wedge 1 \right) \e^{\kappa x\sqrt{(1-h)/s}}
	+ \frac{\Delta}{\sqrt{s}}
	+ h
	\right), \\
\Eci{x}{\widetilde{Z}_s (\abs{F_h},\Delta)}
& \leq C x \e^{-x} \e^{\kappa x\sqrt{(1-h)/s}}.
\end{align*}
\end{lem}
\begin{proof}
Using the many-to-one formula \eqref{eq:many-to-one-bessel}, we have
\begin{align}
\Eci{x}{\widetilde{Z}_s (F_h,\Delta)} 
& = x \e^{-x} 
	\Eci{x/\sqrt{s}}{
	\frac{(R_1 - \frac{\Delta}{\sqrt{s}})_+}{R_1}
	F \left( \sqrt{1-h} 
			\left( R_1 - \frac{\Delta}{\sqrt{s}} \right) \right)}. \label{ma}
\end{align}
Then, we use the following decomposition: 
$\widetilde{Z}_s (F_h-\rho(F),\Delta) 
= (\widetilde{Z}_s (F_h,\Delta) - \rho(F) x \e^{-x}) 
- \rho(F) (\widetilde{Z}_s (1,\Delta) - x \e^{-x})$.
On the one hand, applying \eqref{ma} and Lemma \ref{lem:control-eta-small} with $\alpha=\sqrt{1-h}$, $y=x/\sqrt{s}$, $\eta = \Delta/\sqrt{s}$, as well as the inequality $1-\sqrt{1-h} \leq h$, we have
\begin{align}
\abs{\Eci{x}{\widetilde{Z}_s (F_h,\Delta)} - \rho(F) x \e^{-x}} 
& \leq C x \e^{-x} 
	\left(
	\left(\frac{x^2}{s} \wedge 1 \right) \e^{\kappa x\sqrt{(1-h)/s}}
	+ \frac{\Delta}{\sqrt{s}}
	+ h
	\right), \label{mb}
\end{align}
On the other hand, noting that $\rho(F) \leq C$ by Assumption \ref{ass1} and applying \eqref{mb} with $F=1$, we get a similar bound for the term $\rho(F) (\widetilde{Z}_s (1,\Delta) - x \e^{-x})$.
This proves the first bound of the lemma.

We now deal with the second bound.
Applying \eqref{ma} with $\abs{F_h}=\abs{F}_h$ instead of $F_h$ and then Assumption \ref{ass1}, we get
\begin{align}
\Eci{x}{\widetilde{Z}_s (\abs{F_h},\Delta)}
\leq x \e^{-x} 
	\Eci{x/\sqrt{s}}{\e^{\kappa \sqrt{1-h} R_1}} 
\leq C x \e^{-x} \e^{\kappa x\sqrt{(1-h)/s}}, \label{mi}
\end{align}
by Lemma \ref{lem:control-E_y[R_1^-alpha exp(lambdaR_1)]}. 
This proves the second bound.
\end{proof}
\begin{lem} \label{lem:tightness}
	Let $\kappa \geq 0$.
	Let $F \colon \R \to \R$ be such that $\abs{F(x)} \leq (1+x^{-1}) \e^{\kappa x}$ for any $x>0$.
	Then, the family $(Z_s(F,\gamma))_{s\geq 1,\gamma\geq 0}$ is tight.
\end{lem}
\begin{proof} 
Using \eqref{ma} and then Lemma \ref{lem:control-E_y[R_1^-alpha exp(lambdaR_1)]}, we get, for any $x,s>0$ and $\Delta \geq 0$,
\begin{align*}
\Eci{x}{\abs{\widetilde{Z}_s (F,\Delta)}}
\leq \Eci{x}{\widetilde{Z}_s (\abs{F},\Delta)}
\leq x \e^{-x} 
\Eci{x/\sqrt{s}}{(1+R_1^{-1})\e^{\kappa R_1}} 
\leq C(\kappa) x \e^{-x} \e^{\kappa x/\sqrt{s}}.
\end{align*}
It follows that there exists $C = C(\kappa)> 0$ such that, for $s,M > 0$ and $\gamma \geq 0$, we have
\begin{equation} \label{eq:bound-Z_s(F,gamma,s)-with-barrier-at--M}
\Ec{\abs{Z_s(F,\gamma)} 
	\1_{\forall q \in [0,s], \min_{u \in \cN(q)} X_u(q) > - M}}
\leq \e^M \Eci{M}{\abs{\widetilde{Z}_s(F,\gamma+M)}}
\leq CM \e^{\kappa M/\sqrt{s}}.
\end{equation}
Moreover, it follows from \eqref{eq:global-min-of-the-BBM} that $\P(\exists q \in [0,s], \min_{u \in \cN(q)} X_u(q) \leq - M) \leq \e^{-M}$.
Therefore, for any $s,L,M > 0$,
\begin{equation} \label{eq:bound-Z_s(F,gamma,s)-tightness}
\Pp{\abs{Z_s(F,\gamma)} \geq L} \leq \e^{-M} + CM\e^{\kappa M/\sqrt{s}} L^{-1}.
\end{equation}
The tightness of $(Z_s(F,\gamma))_{s\geq 1,\gamma\geq 0}$ follows.
\end{proof}
\begin{lem} \label{lem:second-moment-Z_s(F,t,Delta)}
Let $\kappa \geq 0$. 	
There exists a constant $C = C(\kappa) >0$, such that
for any function $F \colon \R \to \R$ satisfying Assumptions \ref{ass1}-\ref{ass2} and any $x>0$ and $s \geq 9 \kappa^2\vee 1$,
\begin{align*}
	\Eci{x}{\widetilde{Z}_s (F-\rho(F))^2}
	& \leq C \e^{-x} \frac{x}{\sqrt{s}}, \\
	\Eci{x}{\widetilde{Z}_s (F)^2}
	& \leq C \e^{-x} \left( 1 + \frac{x}{\sqrt{s}} \right).
\end{align*}
\end{lem}
\begin{proof} We start by proving the first inequality.
For brevity, we write $\overline{F} \coloneqq F-\rho(F)$ and therefore $\overline{F}_h \coloneqq (\overline{F})_h =  F_h-\rho(F)$ for $h \in [0,1)$ (note that $(\overline{F})_h \neq \overline{(F)_h}$).
It follows from the many-to-two formula~\eqref{eq:many-to-two-formula} that, with $q_r(x,y) \coloneqq (2 \pi r)^{-1/2} (\e^{-(x-y)^2/2r} - \e^{-(x+y)^2/2r})$,
\begin{equation}
\begin{split}
\Eci{x}{\widetilde{Z}_s (\overline{F})^2}
& = K \e^{-x} \int_0^s \diff r 
	\int_0^\infty \Eci{y}{\widetilde{Z}_{s-r} (\overline{F}_{r/s})}^2 
	\e^y q_r(x,y) \diff y \\
& \relphantom{=} {} 
+ \e^{-x} \int_0^\infty \e^{-y} 
	\left( y \overline{F} \left( \frac{y}{\sqrt{s}} \right)\right)^2 
	q_s(x,y) \diff y.
\end{split} \label{md}
\end{equation}
We first deal with the second term on the right-hand side of \eqref{md}.
Using Assumption \ref{ass1} and $q_s(x,y) \leq \sqrt{2/\pi} xy/s^{3/2}$, we get
\begin{align*}
\int_0^\infty \e^{-y} 
	\left(y \overline{F} \left( \frac{y}{\sqrt{s}} \right)\right)^2 
	q_s(x,y) \diff y
\leq C \int_0^\infty \e^{-y} \left( y \e^{\kappa y/\sqrt{s}} \right)^2
	\frac{xy}{s^{3/2}} \diff y
\leq \frac{Cx}{s^{3/2}}
\leq \frac{Cx}{\sqrt{s}},
\end{align*}
where we used that $\kappa / \sqrt{s} \leq 1/3$ and that $s\ge 1$.

Now, we deal with the first term in the right-hand side of \eqref{md}: for this, we split the integral on $[0,s]$ in two pieces.
Let start with the part $r \in [0, s/2]$.
Using the first bound of Lemma \ref{lem:first-moment-Z_s(F,Delta))} and that $s-r \geq s/2$, we have
\begin{equation}
\abs{\Eci{y}{\widetilde{Z}_{s-r} (\overline{F}_{r/s})}} 
\leq C y \e^{-y} 
	\left(
	\frac{y^2}{s}  \e^{\kappa y/\sqrt{s}}
	+ \frac{r}{s}
	\right)
\leq \frac{C}{s} y \e^{-y} 
	\left( y^2 \e^{y/3} + r \right), \label{bba}
\end{equation}
using that $\kappa / \sqrt{s} \leq 1/3$.
Therefore, 
\begin{align}
& \int_0^{s/2} \diff r 
	\int_0^\infty \Eci{y}{\widetilde{Z}_{s-r} (\overline{F}_{r/s})}^2 
	\e^y q_r(x,y) \diff y \nonumber \\
& \leq \frac{C}{s^2} \int_0^{s/2} \diff r 
	\int_0^\infty y^2 \e^{-y} 
	\left( y^4 \e^{2y/3}+ r^2 \right) 
	q_r(x,y) \diff y \nonumber \\
& \leq \frac{C}{s^2}
	\int_0^\infty y^2 \e^{-y} 
	\left(
	y^4 \e^{2y/3} \int_0^{s/2} q_r(x,y) \diff r
	+ \int_0^{s/2} r^2 q_r(x,y) \diff r 
	\right)
	\diff y. \label{eq:bound_first_part_2nd}
\end{align}
Then, on the one hand, we have, using that $q_r (x,y) \leq \sqrt{2/\pi} r^{-3/2} xy$,
\begin{align*}
\int_0^{s/2} r^2 q_r(x,y) \diff r 
\leq \int_0^{s/2} r^2 \sqrt{\frac{2}{\pi}} \frac{xy}{r^{3/2}} \diff r
\leq C xy s^{3/2},
\end{align*}
and, on the other hand, we have the Green's function identity $\int_0^\infty q_r(x,y) \diff r = 2 (x\wedge y) \leq 2x$.
Thus, it follows that \eqref{eq:bound_first_part_2nd} is at most
\begin{equation} \label{eq:first-part}
\frac{Cx}{s^2} \int_0^\infty y^2 \e^{-y} 
	\left( y^4 \e^{2y/3} + y s^{3/2} \right)
	\diff y
\leq \frac{Cx}{s^2} \left( 1 + s^{3/2} \right)
\leq C \frac{x}{\sqrt{s}}, 
\end{equation}
using that $s \geq 1$. This concludes our work with the part $r \in [0,s/2]$.

Now, we deal with the part $r \in [s/2, s]$ and have, using the second bound of Lemma \ref{lem:first-moment-Z_s(F,Delta))} (note that $\overline{F}/(2 \E[\e^{\kappa R_1}])$ satisfies \ref{ass1}-\ref{ass2}),
\begin{align*}
\abs{\Eci{y}{\widetilde{Z}_{s-r} (\overline{F}_{r/s})}} 
\leq C y \e^{-y} \e^{\kappa y/\sqrt{s}}
\leq C y \e^{-2y/3},
\end{align*}
using that $\kappa / \sqrt{s} \leq 1/3$.
Noting also that $q_r(x,y) \leq C xy/s^{3/2}$ because $r \geq s/2$, we get 
\begin{align*}
\int_{s/2}^s \diff r 
	\int_0^\infty \Eci{y}{\widetilde{Z}_{s-r} (\overline{F}_{r/s})}^2 
	\e^y q_r(x,y) \diff y 
& \leq C \int_0^\infty 
	\int_{s/2}^s y^2 \e^{-y/3} \frac{xy}{s^{3/2}} \diff r \diff y
\leq C \frac{x}{\sqrt{s}}.
\end{align*}
This concludes the proof of the first inequality in the lemma.

For the second inequality, we follow the same steps, except for the part $r \in [0, s/2]$. In this case, we replace \eqref{bba} by the use of the second bound of Lemma \ref{lem:first-moment-Z_s(F,Delta))}, which gives us $\abso{\E_y[\widetilde{Z}_{s-r} (F_{r/s})]} \leq Cy\e^{-y} \e^{\kappa y/\sqrt{s}}$.
It follows that
\[
\int_0^{s/2} \diff r 
	\int_0^\infty \Eci{y}{\widetilde{Z}_{s-r} (F_{r/s})}^2 
	\e^y q_r(x,y) \diff y
\leq C \int_0^\infty y^2 \e^{-y/2} \int_0^{s/2} q_r(x,y) \diff r \diff y \leq C.
\]
using that $\int_0^\infty q_r(x,y) \diff r = 2 (x\wedge y) \leq 2y$.
This replaces the bound \eqref{eq:first-part} and concludes the proof of the second inequality in the lemma.
\end{proof}

\begin{lem}
\label{lem:f_alpha}
Let $\alpha\in(0,2)$ and $\kappa \geq 0$. 
There exists a constant $C = C(\kappa,\alpha) >0$, such that
for any $x>0$, $s \geq 9 \kappa^2 \vee 2$ and any function $F \colon \R \to \R$ satisfying $\abs{F(y)} \leq y^{-\alpha} \e^{\kappa y}$ for all $y > 0$,
\[
\Eci{x}{\widetilde Z_s(F)^2} \le C \e^{-x} \times 
\begin{cases}
1 + \frac{x}{\sqrt s},& \text{if } \alpha < 1, \\
1 + \frac{x}{\sqrt s}\log s, & \text{if } \alpha=1, \\
1 + \frac{x}{\sqrt s}s^{\alpha-1}, & \text{if } \alpha \in (1,2).
\end{cases}
\]
\end{lem}
\begin{proof}
Since $x^{-\alpha} \e^{\kappa x} \leq x^{-\alpha} \e^{\kappa} + \e^{\kappa x}$, it is enough to prove the result when $F(x) = x^{-\alpha}$ and $F(x) = \e^{\kappa x}$.
The case of the function $F(x) = \e^{\kappa x}$ is covered by the second bound of Lemma \ref{lem:second-moment-Z_s(F,t,Delta)}, since $F/(\kappa\vee 1)$ satisfies Assumptions \ref{ass1}-\ref{ass2}.
So we now focus on the function $F(x) = x^{-\alpha}$.
We start as in the proof of Lemma \ref{lem:second-moment-Z_s(F,t,Delta)}:
\begin{align}
\begin{split}
\Eci{x}{\widetilde{Z}_s (F)^2}
& = K \e^{-x} \int_0^s \diff r 
	\int_0^\infty \Eci{y}{\widetilde{Z}_{s-r} (F_{r/s})}^2 
	\e^y q_r(x,y) \diff y \\
& \relphantom{=} {} 
+ \e^{-x} \int_0^\infty \e^{-y} 
	\left(y F \left( \frac{y}{\sqrt{s}} \right)\right)^2 
	q_s(x,y) \diff y.
\end{split} \label{md2}
\end{align}
Using the bound $q_s(x,y)\le Cxy/s^{3/2}$ and that $F(x) = x^{-\alpha}$, the second term on the right-hand side of \eqref{md2} is smaller than
\begin{equation} \label{eq:T2}
T_2 \le C \e^{-x}\frac{x}{\sqrt s} s^{\alpha-1} \int_0^\infty y^{1 + 2(1-\alpha)} \e^{-y} \diff y,
\end{equation}
and the last integral is finite since $\alpha < 2$.
Hence, we now focus on the first term on the right-hand side of \eqref{md2}.
We first note that by the many-to-one formula \eqref{eq:many-to-one-bessel}, we have
\begin{align}
\Eci{y}{\widetilde Z_{s-r}(F_{r/s})}
& = y\e^{-y} \Eci{y/\sqrt{s-r}}{F\left( \sqrt{\frac{s-r}{s}} R_1\right)}
= y\e^{-y} \left( \frac{s}{s-r} \right)^{\alpha/2} \Eci{y/\sqrt{s-r}}{(R_1)^{-\alpha} } 
\nonumber \\
& \leq C y\e^{-y} \left( \frac{s}{s-r} \right)^{\alpha/2} 
\left( 1 \wedge \frac{(s-r)^{\alpha/2}}{y^\alpha} \right), 
\label{T2a}
\end{align}
applying Lemma \ref{lem:control-E_y[R_1^-alpha exp(lambdaR_1)]}.
We now split the integral over $r$ in the definition of the first term on the right-hand side of \eqref{md2} into three parts, according to whether $r\in[0,s/2]$, $r\in[s/2,s-1]$ or $r\in [s-1,s]$. 

For the first part, note that \eqref{T2a} gives $\E_y[\widetilde Z_{s-r}(F_{r/s})] \le C y\e^{-y}$ for every $r \in [0,s/2]$.
It follows that 
\begin{equation}
\label{T2d}
\int_0^{s/2} \diff r 
	\int_0^\infty \Eci{y}{\widetilde Z_{s-r}(F_{r/s})}^2 
	\e^y q_r(x,y) \diff y 
	\le 
	C\int_0^{s/2} \diff r 
	\int_0^\infty y^2 \e^{-y} q_r(x,y) \diff y
	\leq C,
\end{equation}
by integrating first over $r$ and using the Green function identity $\int_0^\infty q_r(x,y)\diff r = 2(x\wedge y) \leq 2y$.
For the second part, we use $\E_y[\widetilde Z_{s-r}(F_{r/s})] \le C y\e^{-y} (\frac{s}{s-r})^{\alpha/2}$ and $q_r(x,y)\le Cxy/r^{3/2}$. Integrating over $y$, this gives,
\begin{align}
\int_{s/2}^{s-1} \diff r 
	\int_0^\infty \Eci{y}{\widetilde Z_{s-r}(F_{r/s})}^2 
	\e^y q_r(x,y) \diff y 
&\le Cx s^{-3/2+\alpha} \int_1^{s/2} u^{-\alpha} \diff u
	\nonumber \\
&\leq C \frac{x}{\sqrt s} \times 
	\begin{cases}
1,& \text{if } \alpha < 1,\\
\log s, & \text{if } \alpha=1,\\
s^{\alpha-1}, & \text{if } \alpha \in (1,2).
\label{T2e}
\end{cases}
\end{align}
Finally, for the third part, we use $\E_y[\widetilde Z_{s-r}(F_{r/s})] \le C y^{1-\alpha}\e^{-y} s^{\alpha/2}$ and again $q_r(x,y)\le Cxy/r^{3/2}$:
\begin{align}
\int_{s-1}^{s} \diff r 
	\int_0^\infty \Eci{y}{\widetilde Z_{s-r}(F_{r/s})}^2 
	\e^y q_r(x,y) \diff y 
\le 
	Cx s^{\alpha-3/2} \int_0^\infty y^{3-2\alpha} \e^{-y}\diff y
\leq C \frac{x}{\sqrt s} s^{\alpha-1},
\label{T2g}
\end{align}
using that $\alpha < 2$. Gathering \eqref{md2}, \eqref{eq:T2}, \eqref{T2d}, \eqref{T2e} and \eqref{T2g}, the statement follows.
\end{proof}

\subsection{Particles killed by a barrier at 0}
\label{subsection:number-of-killed-particles}

In this subsection, we study BBM starting with a single particle at $x>0$ and where particles are killed by hitting 0, and focus on the stopping line $\cL$ of the killed particles defined by
\begin{align*}
\cL
& \coloneqq 
\{u \in \T : \text{there exists } s \geq 0 \text{ s.t.\@ } 
u \in \cN(s), X_u(s) \leq 0 \text{ and } \forall r \in [0,s), X_u(r) > 0 \}.
\end{align*}
Note that $\cL$ is a shorthand for $\cL^{0,0}$ introduced in Section \ref{subsection:scheme_barrier}.
recall from there that, for $u \in \cL$, we define $T_u \coloneqq \inf \{s \geq 0 : u \in \cN(s) \text{ and } X_u(s) \leq 0 \}$ to be the killing time of $u$.
Lastly, for $I \subset \R_+$ an interval, let $\cL_I \coloneqq \{ u \in \cL : T_u \in I \}$ denote the subset of particles that are killed at 0 during the time interval $I$.
Firstly, the following explicit formula follows from Lemma 4.5 of \cite{maillardpain2019}: for any measurable function $\varphi \colon \R \to \R_+$ and any $x > 0$, we have
\begin{align} \label{eq:first-moment-on-the-stopping-line}
\Eci{x}{\sum_{u \in \cL} \varphi(T_u)}  
= x \e^{-x} \int_0^\infty \varphi(r) \frac{\e^{-x^2/2r}}{\sqrt{2\pi}r^{3/2}} \diff r.
\end{align}
Now, we prove two lemmas bounding the first and second moment of the number of particles killed during the time interval $[s,\infty)$.
\begin{lem} \label{lem:first-moment-number-of-killed-particles}
There exists $C>0$ such that, for any $s,x > 0$, we have
\begin{align*}
\Eci{x}{\# \cL_{[s,\infty)}}   
\leq C \e^{-x} \left(1 \wedge \frac{x}{\sqrt{s}} \right).
\end{align*}
\end{lem}
\begin{proof}
It has already been shown in Lemma 4.5 of \cite{maillardpain2019} that $\Eci{x}{\# \cL} = \e^{-x}$.
Moreover, 
\begin{align*}
\Eci{x}{\# \cL_{[s,\infty)}}   
& = x \e^{-x} \int_s^\infty \frac{\e^{-x^2/2r}}{\sqrt{2\pi}r^{3/2}} \diff r
\leq C x \e^{-x} \int_s^\infty \frac{1}{r^{3/2}} \diff r
\leq C \frac{x}{\sqrt{s}} \e^{-x}.
\end{align*}
Combining both results, it proves the result. 
\end{proof}
\begin{lem} \label{lem:second-moment-number-of-killed-particles}
There exists $C>0$ such that, for any $s,x > 0$, we have
\[
\Eci{x}{(\# \cL_{[s,\infty)})^2}   
\leq C \e^{-x} \left( 1 \wedge \frac{x}{\sqrt{s}} \right).
\]
\end{lem}
\begin{proof}
For this proof, we use the change of measure introduced in Section~4.3 of \cite{maillardpain2019}. 
We follow the proof of Lemma 4.8 in \cite{maillardpain2019}, but note that here, since we assumed that $\Ec{L^2} < \infty$, we do not need to use a truncation of the number of offspring of the spine. 
Firstly, note that we have $\E_x[(\# \cL_{[s,\infty)})^2] \leq C \e^{-x}$ directly from Lemma 4.8 in \cite{maillardpain2019}, so we can now focus on proving the other part of the bound.
Moreover, observe that
\begin{align*}
\Eci{x}{(\# \cL_{[s,\infty)})^2}
= \Eci{x}{\# \cL_{[s,\infty)} (\# \cL_{[s,\infty)}-1)}
+ \Eci{x}{\# \cL_{[s,\infty)}}
\end{align*}
and $\E_x[\# \cL_{[s,\infty)}] \leq C x \e^{-x} / \sqrt{s}$ by Lemma \ref{lem:first-moment-number-of-killed-particles}, so we only have to deal with the first summand.

For $x> 0$, let $\P_x^0$ be the law of the BBM started from a particle at $x$ with the additional rule that particles hitting the origin stop moving and branching (but they stay in the system).
Furthermore, let $\Q_x^0$ be the probability measure obtained by changing
the measure $\P_x^0$ w.r.t.\@ the mean 1 nonnegative martingale $(\e^{-x} W_t)_{t\geq0}$. Under $\Q_x^0$, the BBM can be represented using the so-called ``spine decomposition'' using a ``spine'' $(w_t)_{t\geq0}$ moving like standard Brownian motion (without drift) stopped at 0, branching at rate $\lambda \E[L]$ on $(0,\infty)$ and at rate 0 at 0, and giving birth according to the reproduction law given by the size-biased distribution of $L$.
See \cite{maillardpain2019} for details.

For a particle $u$, let $\tau_u \coloneqq \inf \{ r \geq 0 : X_u(r) = 0 \}$ with $\inf \emptyset = \infty$.
We have, for any $t \geq s$,
\begin{align*}
\Eci{x}{\# \cL_{[s,t)} (\# \cL_{[s,t)}-1)}
&= \E_x^0 \left[ \sum_{u\in\cN(t)} \1_{\tau_u \in [s,t)} (\# \cL_{[s,t)}-1) \right]
= \e^{-x} 
\E_{\Q_x^0} \left[ \1_{\tau_{w_t} \in [s,t)} (\# \cL_{[s,t)}-1) \right],
\end{align*}
applying \cite[Proposition 4.1]{maillardpain2019} in the second equality.
Then, denoting by $\Pi_t$ the set of branching times of the spine before $t$ and by $O_{w_r}$ the number of offspring of $w_r$ at a branching time $r$, and decomposing $(\# \cL_{[s,t)}-1)$ along the spine, we get 
\begin{align*}
\Eci{x}{\# \cL_{[s,t)} (\# \cL_{[s,t)}-1)}
& = \e^{-x} \E_{\Q_x^0} \left[ \1_{\tau_{w_t} \in [s,t)}
\sum_{r \in \Pi_{\tau_{w_t}}}
	\left( O_{w_r} - 1 \right) \Eci{X_{w_r}(r)}{\# \cL_{[(s-r)\vee 0,t-r]}}
\right] \\
&\leq C \e^{-x} \E_{\Q_x^0} \left[ 
\1_{\tau_{w_t} \in [s,t)}
\sum_{r \in \Pi_{\tau_{w_t}}} \e^{-X_{w_r}(r)} \left(1 \wedge \frac{X_{w_r}(r)}{\sqrt{(s-r)\vee 0}} \right)
\right],
\end{align*}
using Lemma \ref{lem:first-moment-number-of-killed-particles} and the fact that, conditioned on the trajectory and the branching times of the spine, $O_{w_r}$ follows the size-biased distribution of $L$ and therefore has a finite first moment thanks to the standing assumption $\E[L^2] < \infty$.
Denoting by $(B_r)_{r \geq 0}$ a Brownian motion started at $x$ under
$\P_x$ and setting $\tau \coloneqq \inf \{ r \geq 0 : B_r = 0 \}$, it follows from the spinal decomposition description that
\begin{align}
\Eci{x}{\# \cL_{[s,t)} (\# \cL_{[s,t)}-1)}
& \leq C \e^{-x} \E_x \left[ \1_{\tau \in [s,t)} 
\int_0^\tau \e^{-B_r} \left(1 \wedge \frac{B_r}{\sqrt{(s-r)\vee 0}} \right) \diff r
\right] \nonumber \\
& \leq C \e^{-x} \E_x \left[ \1_{\tau \geq s}
\int_0^{s/2} \e^{-B_r} \frac{B_r}{\sqrt{s}} \diff r
\right]
+ C \e^{-x} \E_x \left[ \1_{\tau \geq s} 
\int_{s/2}^\tau \e^{-B_r} \diff r
\right] \nonumber \\
& \leq \frac{C\e^{-x}}{\sqrt{s}} \E_x \left[ 
\int_0^\tau B_r \e^{-B_r} \diff r
\right]
+ C \e^{-x} \int_{s/2}^\infty \Eci{x}{\e^{-B_r} \1_{\tau \geq r}} \diff r. \label{me}
\end{align}
Now, recall that $q_r(x,\cdot)$ is the density of a Brownian motion at time $r$ starting from $x$ and killed at 0. Then, on the one hand, using the Green's function identity $\int_0^\infty q_r(x,y) \diff r = 2 (x \wedge y) \leq 2x$, we get 
\begin{equation} \label{eq:integral-killed-BM}
\Eci{x}{\int_0^\tau B_r \e^{-B_r} \diff r} 
= \int_0^\infty \int_0^\infty y\e^{-y} q_r(x,y) \diff y \diff r
\leq \int_0^\infty 2xy\e^{-y} \diff y 
= 2 x. 
\end{equation}
On the other hand, using that $q_r (x,y) \leq \sqrt{2/\pi} r^{-3/2} xy$, it follows that
\begin{equation} \label{eq:exponential-killed-BM}
\int_{s/2}^\infty \Eci{x}{\e^{-B_r} \1_{\tau \geq r}} \diff r
= \int_{s/2}^\infty \int_0^\infty \e^{-y} q_r(x,y) \diff y \diff r
\leq C \frac{x}{\sqrt{s}}.
\end{equation}
Combining these two bounds with \eqref{me} concludes the proof.
\end{proof}

\subsection{Particles killed by a barrier at level \texorpdfstring{$\gamma$}{gamma} after time \texorpdfstring{$t_0$}{t0}}
\label{subsection:number-of-killed-particles-level-gamma}

In this subsection, we state some straightforward consequences of the results in Section \ref{subsection:number-of-killed-particles} in the more general case of a barrier killing particles below level $\gamma$ on the time interval $[t_0,\infty)$, for some arbitrary $\gamma \in \R$ and $t_0 > 0$.
We use notation from Section \ref{subsection:scheme_barrier}. 
Moreover, we define, for any interval $I \subset [t_0,\infty)$,
\[
\cL^{t_0,\gamma}_I 
\coloneqq \{ u \in \cL^{t_0,\gamma} : T_u \in I \}
\]
the subset of $\cL^{t_0,\gamma}$ consisting of particles killed during the time interval $I$.
Note that $\cL^{t_0,\gamma}_{\{t_0\}}$ plays a special role, because particles killed at time $t_0$ can be strictly below $\gamma$.
We are not considering these particles in this subsection.
Firstly, it follows from the branching property at time $t_0$ and \eqref{eq:first-moment-on-the-stopping-line} that, for any measurable $\varphi \colon (t_0,\infty) \to \R_+$, 
\begin{align} 
\E \Biggl[ \sum_{u \in \cL^{t_0,\gamma}_{(t_0,\infty)}} \varphi(T_u)
\Bigg| \sF_{t_0} \Biggr]
& = \sum_{v \in \cN(t_0)} 
	(X_v(t_0) - \gamma)_+ \e^{-(X_v(t_0) - \gamma)} 
	\int_{t_0}^\infty \varphi(s) 
	\frac{\e^{-(X_v(t_0)-\gamma)^2/2(s-t_0)}}{\sqrt{2\pi}(s-t_0)^{3/2}} 
	\diff s
	\nonumber \\
& \leq \e^{\gamma} Z_{t_0}(1,\gamma)
	\int_{t_0}^\infty \frac{\varphi(s)}{(s-t_0)^{3/2}} \diff s.
	\label{eq:first-moment-on-the-stopping-line-gamma}
\end{align}
Secondly, it follows from the branching property at time $t_0$ and Lemma \ref{lem:first-moment-number-of-killed-particles} that
\begin{equation} \label{eq:number-of-killed-particles-given-sF_s^a}
\Ecsq{\# \cL_{(t_0,\infty)}^{t_0,\gamma}}{\sF_{t_0}}   
\leq C \e^{\gamma} W_{t_0}.
\end{equation}

\section{A first control on the rate of convergence: proof of Proposition~\ref{prop:rate-of-CV-Z_t(F,Delta,t)}}
\label{section:a-first-rough-approach-to-the-result}

The goal of this section is to prove Proposition \ref{prop:rate-of-CV-Z_t(F,Delta,t)}, which gives a first bound on the rate of convergence of $Z_t (F,\Delta)$ towards $\rho(F) Z_\infty$, when $\Delta$ is at most of order $\log t$.
To this aim, we use the strategy presented in Section \ref{subsection:scheme_barrier}, by introducing a barrier at level $\gamma$ from time $t_0 = t^a$, for some $a \in (0,1)$.
We eventually choose $\gamma = \frac{1}{6} \log t$ and $a = 2/3$ in Section~\ref{subsection:rate_of_convergence}, but we keep some generality for $\gamma$ and $a$ throughout the section to show where this choice comes from.

We first prove a concentration result for 
$Z_t(F,\gamma) - \rho(F) Z_t(1,\gamma)$, and only afterwards prove that we can replace $Z_t(1,\gamma)$ by $Z_\infty$ and $\gamma$ by any other shift $\Delta$.
With the notation of Section~\ref{subsection:scheme_barrier}, we have the following decomposition
\begin{align} \label{eq:decompo-Z_s(F,t,Delta_s)}
Z_t(F,\gamma) - \rho(F) Z_t(1,\gamma)
& = \widetilde{Z}_t^{t^a,\gamma} (F-\rho(F),\gamma) 
+ \sum_{u \in \cL^{t^a,\gamma}} \overline{\Omega}^{(u)}_t,
\end{align}
where $\overline{\Omega}^{(u)}_t$ is the contribution of the progeny of the killed particle $u$ to $Z_t(F-\rho(F),\gamma)$ and is defined by
\begin{align} \label{eq:def_Omega_barre}
\overline{\Omega}^{(u)}_t \coloneqq 
\sum_{v \in \cN(t) \text{ s.t.\@ } u \leq v} 
(X_v(t)-\gamma)_+ \e^{-X_v(t)} 
	\left( F \left( \frac{X_v(t)-\gamma}{\sqrt{t}} \right)
	- \rho(F) \right)
	\quad \text{if } T_u \leq t,
\end{align}
and $\overline{\Omega}^{(u)}_t \coloneqq 0$ if $T_u>t$.
Recall from \eqref{eq:def_F_h} that $F_h(x) \coloneqq F(\sqrt{1-h}\cdot x)$ for $h \in [0,1)$ and $x \in \R$.
Note that conditionally on $\sF_{\cL^{t^a,\gamma}}$, the $\overline{\Omega}^{(u)}_t$ for $u \in \cL^{t^a,\gamma}$ with $T_u \leq t$ are independent with respectively the same law as $Z_{t-T_u} (F_{T_u/t}-\rho(F),\gamma)$ under $\P_{X_u(T_u)}$.

\subsection{The branching Brownian motion with barrier between times \texorpdfstring{$t^a$}{ta} and \texorpdfstring{$t$}{t}}
\label{subsection:BBM-with-barrier-s^alpha-to-s}

In this section, we deal with the first part in decomposition \eqref{eq:decompo-Z_s(F,t,Delta_s)}: we prove $\widetilde{Z}_t^{t^a,\gamma} (F-\rho(F),\gamma)$ is small, by controlling its first moment and its variance conditionally on $\sF_{t^a}$.
\begin{lem} \label{lem:first-second-moment-Z_s^(gamma_s)(F,Delta_s)}
Let $\kappa \geq 0$ and $a\in(0,1)$.
There exist $C = C(\kappa) >0$ and $t_0 = t_0(\kappa,a) > 0$, such that
for any function $F \colon \R \to \R$ satisfying Assumptions \ref{ass1}-\ref{ass2}, $\gamma \in \R$ and $t \geq t_0$,
\begin{align*}
\abs{ \Ecsq{\widetilde{Z}_t^{t^a,\gamma} (F-\rho(F),\gamma)}{\sF_{t^a}} }
& \leq \frac{C}{t^{1-a}} Z_{t^a} (x \mapsto x^2\e^{\kappa x}+1,\gamma), \\
\Varsq{\widetilde{Z}_t^{t^a,\gamma} (F-\rho(F),\gamma)}{\sF_{t^a}} 
& \leq \frac{C \e^{-\gamma}}{\sqrt{t}} 
	Z_{t^a} \left( x \mapsto 1 + x^{-1},\gamma \right).
\end{align*}
\end{lem}
\begin{proof} For brevity, we write $\overline{F} \coloneqq F-\rho(F)$.
Let start with the first moment.
Using the branching property at time $t^a$, we get 
\[
\abs{ \Ecsq{\widetilde{Z}_t^{t^a,\gamma} (\overline{F},\gamma)}{\sF_{t^a}} }
\leq \sum_{v \in \cN(t^a)} 
	\e^{-\gamma} \1_{X_v(t^a) > \gamma} 
	\abs{\Eci{X_v(t^a)-\gamma}{\widetilde{Z}_{t-t^a}(\overline{F}_{t^a/t}) }}.
\]
Then, applying the first bound of Lemma \ref{lem:first-moment-Z_s(F,Delta))} and noting that $t-t^a \geq t/2$ for $t$ large enough, we have, for any $y>0$,
\[
\abs{\Eci{y}{\widetilde{Z}_{t-t^a}(\overline{F}_{t^a/t}) }} 
\leq C y \e^{-y} 
	\left(
	\frac{y^2}{t} 
	\e^{\kappa y / \sqrt{t}}
	+ \frac{t^a}{t}
	\right)
\leq C y \e^{-y} \frac{t^a}{t} \left(\frac{y^2}{t^a}\e^{\kappa y / \sqrt{t^a}}+1\right),
\]
for $t$ large enough. The first inequality follows.

We can now deal with the second moment.
Using the branching property at time $t^a$, we get
\begin{align} \label{eq:conditional_variance_1}
\Varsq{\widetilde{Z}_t^{t^a,\gamma} (\overline{F},\gamma)}{\sF_{t^a}}
\leq \sum_{v \in \cN(t^a)} 
	\e^{-2\gamma} \1_{X_v(t^a) > \gamma} 
	\Eci{X_v(t^a) - \gamma}{\widetilde{Z}_{t-t^a} (\overline{F}_{t^a/t})^2}.
\end{align}
For any $y >0$ and $h \in [0,1)$, we bound
\begin{align*}
\Eci{y}{\widetilde{Z}_{t-t^a} (\overline{F}_h)^2}
\leq 2 \Eci{y}{\widetilde{Z}_{t-t^a} (F_h - \rho(F_h))^2}
+ 2 (\rho(F_h) - \rho(F))^2 \Eci{y}{\widetilde{Z}_{t-t^a}(1)^2}.
\end{align*}
Then, since $F_h$ satisfies Assumptions \ref{ass1}-\ref{ass2}, we can apply the first bound of Lemma \ref{lem:second-moment-Z_s(F,t,Delta)} to the first term on the right-hand side of the last equation.
For the second term, we use that $\abso{\rho(F_h) - \rho(F)} \leq Ch$ by \eqref{eq:R F with F_h} and \eqref{eq:RF_bound} and we apply the second bound of Lemma~\ref{lem:second-moment-Z_s(F,t,Delta)} to $\E_y[\widetilde{Z}_{t-t^a}(1)^2]$.
This yields
\begin{align} \label{eq:conditional_variance_2}
\Eci{y}{\widetilde{Z}_{t-t^a} (\overline{F}_{t^a/t})^2}
\leq C \e^{-y} \frac{y}{\sqrt{t}}
+ C \left( \frac{t^a}{t} \right)^2 \e^{-y} \left( 1	+ \frac{y}{\sqrt{t}} \right) 
\leq \frac{C}{\sqrt{t}} y \e^{-y} 
	\left( 1 + \frac{\sqrt{t^a}}{y} \right).
\end{align}
Collecting the previous estimates, we finally get
\begin{align*}
	\Pp{\abs{Z_t(F,\Delta) - \rho(F) Z_\infty} \geq \delta}
	& \leq 5t^{-K} + \frac{C(\log t)^2}{\delta t^{1/3}},
\end{align*}
and applying this bound with $K+1$ instead of $K$ yields the result.
\end{proof}
In the following lemma, we combine the first and second moments calculation to get a concentration result on $\widetilde{Z}_t^{t^a,\gamma} (F-\rho(F),\gamma)$ around 0.
\begin{lem} \label{lem:control-rate-of-CV-killed-Z_s^(s,gamma_s)(F,Delta_s)}
Let $\kappa \geq 0$, $K>0$ and $a\in(0,1)$.
There exist $C = C(\kappa,K) >0$ and $t_0 = t_0(\kappa,K,a) >0$, such that
for any function $F \colon \R \to \R$ satisfying Assumptions \ref{ass1}-\ref{ass2}, $t \geq t_0$, $\abs{\gamma} \leq K \log t$, $M \in [1,K \log t]$ and $\delta > 0$, 
\begin{align*}
\Pp{\abs{\widetilde{Z}_t^{t^a,\gamma} (F-\rho(F),\gamma)} \geq \delta}
& \leq \e^{-M} 
+ \frac{C M}{\delta t^{1-a}} 
+ \frac{C \e^{-\gamma}}{\delta^2 \sqrt{t}} M.
\end{align*}
\end{lem}
\begin{proof}
For brevity, we write $\overline{F} \coloneqq F-\rho(F)$.
We introduce a barrier at $-M$ between times $0$ and $t^a$ by considering the event $A \coloneqq \{\forall r \in [0,t^a], \min_{u \in \cN(r)} X_u(r) > -M \}$.
Using \eqref{eq:global-min-of-the-BBM}, we have $\P(A^c) \leq \e^{-M}$ and, therefore,
\begin{align}
\begin{split}
\Pp{\abs{\widetilde{Z}_t^{t^a,\gamma} (\overline{F},\gamma)} \geq \delta} 
& \leq \e^{-M}
	+ \Pp{A,
	\abs{ \Ecsq{\widetilde{Z}_t^{t^a,\gamma} (\overline{F},\gamma)}{\sF_{t^a}} }
		\geq \frac{\delta}{2}} \\
& \relphantom{\leq} {}
	+ \Pp{ A,
	\abs{\widetilde{Z}_t^{t^a,\gamma} (\overline{F},\gamma) 
	- \Ecsq{\widetilde{Z}_t^{t^a,\gamma} (\overline{F},\gamma)}{\sF_{t^a}}} 
	\geq \frac{\delta}{2} }.
\end{split}	\label{iz}
\end{align}
Then, using Markov's inequality and the first part of Lemma \ref{lem:first-second-moment-Z_s^(gamma_s)(F,Delta_s)}, the second term in the right-hand side of \eqref{iz} is smaller than
\begin{align} \label{iy}
\frac{2}{\delta} 
\Ec{\1_A \abs{\Ecsq{\widetilde{Z}_t^{t^a,\gamma} (\overline{F},\gamma)}{\sF_{t^a}}}}
& \leq 
\frac{C}{\delta t^{1-a}} 
\Ec{\1_A Z_{t^a} (x \mapsto x^2e^{\kappa x}+1,\gamma)}
\leq \frac{C M}{\delta t^{1-a}},
\end{align}
where the second inequality follows from \eqref{eq:bound-Z_s(F,gamma,s)-with-barrier-at--M} for $t$ large enough such that $M\leq \sqrt{t}$.
On the other hand, using Chebyshev's inequality and the second part of Lemma \ref{lem:first-second-moment-Z_s^(gamma_s)(F,Delta_s)}, the third term in the right-hand side of \eqref{iz} is smaller than
\begin{align} \label{iw}
\frac{4}{\delta^2}
\Ec{\1_A \Varsq{\widetilde{Z}_t^{t^a,\gamma} (\overline{F},\gamma)}{\sF_{t^a}}} 
\leq \frac{C \e^{-\gamma}}{\delta^2 \sqrt{t}}
	\Ec{\1_A Z_{t^a}\left( x \mapsto 1 + x^{-1},\gamma \right)}
\leq \frac{C \e^{-\gamma}}{\delta^2 \sqrt{t}} M,
\end{align}
using again \eqref{eq:bound-Z_s(F,gamma,s)-with-barrier-at--M}. 
The result follows from \eqref{iz}, \eqref{iy} and \eqref{iw}.
\end{proof}

\subsection{Contributions of killed particles}

In this subsection, we deal with the contributions of the particles killed by the barrier at level $\gamma$ between times $t^a$ and $t$.
This relies on the following lemma, proved by a crude triangle inequality and first moment bound on the contributions $\overline{\Omega}^{(u)}_t$ defined in \eqref{eq:def_Omega_barre}.
\begin{lem} \label{lem:control-fluctuations-Z_s(F,t,Delta_s)}
Let $\kappa \geq 0$, $K>0$ and $a\in(0,1)$.
There exist $C = C(\kappa,K) >0$ and $t_0 = t_0(\kappa,K,a) >0$, such that
for any function $F \colon \R \to \R$ satisfying Assumptions \ref{ass1}-\ref{ass2}, $t \geq t_0$, $\gamma \in [0,K \log t]$ and $\delta > 0$, 
\begin{align*}
\P \Biggl( \sum_{u \in \cL^{t^a,\gamma}} \abs{\overline{\Omega}^{(u)}_t} \geq \delta \Biggr) 
& \leq t^{-K} + \frac{C}{\delta} \frac{(\log t)^2}{t^{a/2}}.
\end{align*}
\end{lem}
\begin{proof} 
For brevity, we write $\overline{F} \coloneqq F-\rho(F)$ (and $\overline{F}_h = (\overline{F})_h$).
Let $M \coloneqq K \log t$.
We introduce a barrier at $-M$: we set $A_s \coloneqq \{\forall r \in [0,s], \min_{u \in \cN(r)} X_u(r) > - M \}$ for any $s \geq 0$.
We will bound the first moment of the quantity of interest on the event $A_t$.
Conditionally on $\sF_{\cL^{t^a,\gamma}}$, for each $u \in \cL^{t^a,\gamma}$, 
the random variable $\1_{A_t}\bigl\lvert \overline{\Omega}^{(u)}_t \bigr\rvert$ is stochastically dominated by $\e^M \widetilde{Z}_{t-T_u} (\lvert \overline{F}_{T_u/t} \rvert,\gamma+M)$ under $\P_{X_u(T_u)+M}$. 
Thus, by linearity of expectation and using that the event $A_{t^a}$ is $\sF_{\cL^{t^a,\gamma}}$-measurable and contains $A_t$, we get
\begin{align*}
\E \Biggl[ \1_{A_t}
	\sum_{u \in \cL^{t^a,\gamma}} \abs{\overline{\Omega}^{(u)}_t}
	\,\Bigg|\, \sF_{\cL^{t^a,\gamma}} \Biggr]
& \leq \1_{A_{t^a}} 
	\sum_{u \in \cL^{t^a,\gamma}} 
	\e^M \Eci{X_u(T_u)+M}{\widetilde{Z}_{t-T_u} (\lvert \overline{F}_{T_u/t} \rvert,\gamma+M)} \\
& \leq C \1_{A_{t^a}} 
	\sum_{u \in \cL^{t^a,\gamma}} 
	(X_u(T_u)+M) \e^{-X_u(T_u)} 
	\e^{\kappa (X_u(T_u)+M)/\sqrt{t}},
\end{align*}
using the second bound of Lemma \ref{lem:first-moment-Z_s(F,Delta))}.
Note that for $u \in \cL^{t^a,\gamma}$, $X_u(T_u) \leq \gamma$ so, for $t$ large enough, $\e^{\kappa (X_u(T_u)+M)/\sqrt{t}} \leq \e^{\kappa}$.
Therefore, distinguishing between particles killed at time $t^a$ or later (for which $X_u(T_u) = \gamma$), we get
\begin{align}
\E \Biggl[ \1_{A_t}	\sum_{u \in \cL^{t^a,\gamma}} \abs{\overline{\Omega}^{(u)}_t} \Biggr]
& \leq C (\gamma+M) \cdot \E \Biggl[ \1_{A_{t^a}}  \Biggl(
	\sum_{u \in \cL^{t^a,\gamma}_{\{t^a\}}} 
	\e^{-X_u(t^a)}
	+ \e^{-\gamma} \# \cL^{t^a,\gamma}_{(t^a,t]}
	\Biggr) \Biggr], \label{fff}
\end{align}
On the one hand, we have
\[
	\E \Biggl[ \1_{A_{t^a}} \sum_{u \in \cL^{t^a,\gamma}_{\{t^a\}}} \e^{-X_u(t^a)} \Biggr]
	\leq \Ec{ \sum_{u \in \cN(t^a)} 
	 \e^{-X_u(t^a)} \1_{X_u(t^a) \leq \gamma} \1_{\forall s \in [0,t^a], X_u(s) > - M}}
	\leq \frac{CM(M+\gamma)^2}{t^{3a/2}},
\]
by proceeding as in \eqref{eq:Fbad2}.
On the other hand, applying \eqref{eq:number-of-killed-particles-given-sF_s^a}, we get
\[
	\e^{-\gamma} \Ec{ \1_{A_{t^a}} \# \cL^{t^a,\gamma}_{(t^a,t]} }
	= \e^{-\gamma} \Ec{ \1_{A_{t^a}} \Ecsq{\# \cL^{t^a,\gamma}_{(t^a,t]}}{\sF_{t^a}} }
	\leq C \Ec{ \1_{A_{t^a}} W_{t^a} }
	\leq \frac{CM}{\sqrt{t^a}},
\]
where the last inequality follows from \eqref{eq:bound-Z_s(F,gamma,s)-with-barrier-at--M}.
Therefore the right-hand side of \eqref{fff} is smaller than $C(\log t)^2/\sqrt{t^a}$.
Combining this with $\P(A^c_t) \leq \e^{-M} = t^{-K}$ by \eqref{eq:global-min-of-the-BBM} and Markov's inequality, this concludes the proof.
\end{proof}

\subsection{Rate of convergence of \texorpdfstring{$Z_t(F,\Delta)$}{Zt(F,Delta)}}
\label{subsection:rate_of_convergence}
\begin{proof}[Proof of Proposition \ref{prop:rate-of-CV-Z_t(F,Delta,t)}]
Let $a \in (0,1)$, $K>0$ $\gamma \in [0,K \log t]$.
Using \eqref{eq:decompo-Z_s(F,t,Delta_s)}, Lemma \ref{lem:control-rate-of-CV-killed-Z_s^(s,gamma_s)(F,Delta_s)} (with $M = K \log t$) and Lemma \ref{lem:control-fluctuations-Z_s(F,t,Delta_s)}, we get
\[
\Pp{\abs{Z_t(F,\gamma) - \rho(F) Z_t(1,\gamma)} \geq \delta} 
\leq t^{-K} + 
C (\log t)^2
\left( 
\frac{t^{a-1}}{\delta}
+ \frac{\e^{-\gamma}}{\delta^2 \sqrt{t}}
+ \frac{t^{-a/2}}{\delta}
\right).
\]
The optimal choice of $a$ is $a=2/3$ which we set for the remainder of the proof.
We also choose\footnote{We remark that we could have chosen $\gamma = b\log t$ for any $b\ge 1/6$ as well.} $\gamma = \frac{1}{6} \log t$ and therefore assume that $K\ge 1/6$, which we can do without loss of generality, as the statement of the proposition becomes stronger upon increasing $K$. 
Then the previous bound becomes 
$t^{-K} + C (\log t)^2 (t^{-1/3} \delta^{-1} + (t^{-1/3} \delta^{-1})^2)$. Furthermore, noting that a probability is always bounded by 1 and that $\min((t^{-1/3} \delta^{-1})^2,1) \le t^{-1/3} \delta^{-1}$, we have
\[
\Pp{\abs{Z_t(F,\gamma) - \rho(F) Z_t(1,\gamma)} \geq \delta} 
\leq t^{-K} + C (\log t)^2 t^{-1/3} \delta^{-1}.
\]
We now compare $Z_t(F,\gamma)$ and $Z_t(F,\Delta)$.
It follows from Assumptions~\ref{ass1} and \ref{ass2} (see \eqref{eq:bound_Lipschitz}) that
\begin{align*}
\abs{Z_t(F,\gamma) - Z_t(F,\Delta)} 
& \leq C \frac{\abs{\Delta-\gamma}}{\sqrt{t}} 
\sum_{u \in \cN(t)} (X_u(t)_+ +1) \e^{-X_u(t)} 
	\e^{\kappa X_u(t)/\sqrt{t}} \\
& \leq C \frac{\log t}{\sqrt{t}} Z_t(x \mapsto (1+x^{-1})\e^{\kappa x}).
\end{align*}
Then, using \eqref{eq:bound-Z_s(F,gamma,s)-tightness} with $M = K \log t$, we get
\begin{align} \label{ix} 
\Pp{\abs{Z_t(F,\gamma) - Z_t(F,\Delta)} \geq \delta} 
\leq t^{-K} + C (\log t)^2 t^{-1/2} \delta^{-1}.
\end{align}
Finally, we have to bound $\abs{Z_t(1,\gamma) - Z_\infty}$.
Applying \eqref{ix} to $F=1$ and $\Delta=0$, we can bound $\abs{Z_t(1,\gamma) - Z_t(1,0)}$
and recall from Proposition~\ref{prop:control-rate-of-cv-Z} that we have a bound for $\abs{Z_t - Z_\infty}$.
Now, note that $Z_t(1,0)$ corresponds to $Z_t$ without the particles below 0 at time $t$. We bound the difference using a first moment on the event $A \coloneqq \{\forall s \in [0,t], \min_{u \in \cN(s)} X_u(s) > - M \}$ with $M = K \log t$: using the many-to-one formula \eqref{eq:many-to-one-bessel}, we get 
\begin{align*}
\Ec{ \abs{Z_t(1,0) - Z_t} \1_A }
\leq \Ec{\sum_{u\in\cN(t)} M \e^{-X_u(t)} \1_{X_u(t) \leq 0}
	\1_{\forall s \in [0,t], X_u(s) > -M} }
\leq C \frac{M^4}{t^{3/2}},
\end{align*}
by proceeding as in \eqref{eq:Fbad2}.
Recalling from \eqref{eq:global-min-of-the-BBM} that $\P(A^c) \leq \e^{-M} = t^{-K}$, we get the bound
$\Pp{\abs{Z_t(1,0) - Z_t} \geq \delta} 
\leq t^{-K} + C (\log t)^4 t^{-3/2} \delta^{-1}$.
The result follows.
\end{proof}
The following corollary of Proposition \ref{prop:rate-of-CV-Z_t(F,Delta,t)} will also be useful.
\begin{cor} \label{cor:rate_of_convergence_as_expectation}
Let $\kappa \geq 0$ and $K > 0$. 
There exist $C = C(\kappa,K) >0$ and $t_0 = t_0(\kappa,K) > 0$, such that
for any function $F \colon \R \to \R$ satisfying Assumptions \ref{ass1}-\ref{ass2} and any $t \geq t_0$, $\Delta \in [0,K \log t]$ and $\ep \geq t^{-K}$, 
\[
\E\left[ \left( \ep \abs{ Z_t(F,\Delta) - \rho(F)Z_\infty} \right) \wedge 1 \right] 
\le C (\log t)^3 t^{-1/3} \ep.
\]
\end{cor}
\begin{proof} The expectation on the left-hand side equals 
$\int_0^1 \P(\abs{ Z_t(F,\Delta) - \rho(F)Z_\infty} \geq x/\ep) \diff x$.
For $x \leq t^{-1/3}\ep$, we bound the probability by 1.
Otherwise, we bound it using Proposition \ref{prop:rate-of-CV-Z_t(F,Delta,t)}.
Integrating over $x$, the result follows.
\end{proof}

\section{A first limit result: proof of Proposition~\ref{prop:convergence-with-translation-gamma_t}}
\label{section:using-previous-results}

In this section, we prove Proposition \ref{prop:convergence-with-translation-gamma_t}, a limit theorem for the fluctuations of $Z_t(F,\gamma) - \rho(F) Z_t(1,\gamma)$.
For this we use the scheme presented in Section \ref{subsection:scheme_barrier} with $t_0 = t^a$ for some $a \in (1/3,1/2)$ and $\gamma = \frac{1}{2} \log t + \beta_t$ with $\beta_t \to \infty$ and $\beta_t = o(\log t)$ as $t \to \infty$.
More precisely, we use the same decomposition of $Z_t(F,\gamma) - \rho(F) Z_t(1,\gamma)$ as in the previous section, see \eqref{eq:decompo-Z_s(F,t,Delta_s)} and \eqref{eq:def_Omega_barre}.
It follows from Lemma \ref{lem:control-rate-of-CV-killed-Z_s^(s,gamma_s)(F,Delta_s)} that when particles killed by the barrier are not taken into account, there are no fluctuations at scale $1/\sqrt{t}$ (see the proof of Proposition \ref{prop:convergence-with-translation-gamma_t} below).
Therefore, the main task in this section is to prove the convergence in distribution for the contribution of killed particles, which is stated in the following lemma and proved in Section \ref{subsection:proof-of-proposition-convergence-of-fluctuations}. 
Recall that $\overline{\Omega}^{(u)}_t$ denotes the contribution of the progeny of the killed particle $u$ in $Z_t(F-\rho(F),\gamma)$, see \eqref{eq:def_Omega_barre}.
\begin{lem} \label{lem:convergence-of-fluctuations}
	Let $\gamma = \frac{1}{2} \log t + \beta_t$, where $\beta_t \to \infty$ and $\beta_t = o(\log t)$ as $t \to \infty$.
	Let $\kappa \geq 0$.
	For any function $F \colon \R \to \R$ satisfying \ref{ass1}-\ref{ass2} and any $a \in (1/3,1)$, we have
	\begin{align*}
		\sqrt{t}
		\Biggl(
		\sum_{u \in \cL^{t^a,\gamma}} \overline{\Omega}^{(u)}_t
		- \frac{\beta_t}{\sqrt{t}} \int_0^1 \mathscr R F(r) \frac{Z_\infty}{\sqrt{2\pi}} \frac{\diff r}{r^{3/2}}
		\Biggr)
		\xrightarrow[t\to\infty]{\text{(law)}} \int_0^1 \mathscr R F(r) M_{Z_\infty}(\diff r).
	\end{align*}
\end{lem}
With this result in hand, we can prove Proposition \ref{prop:convergence-with-translation-gamma_t}.
\begin{proof}[Proof of Proposition \ref{prop:convergence-with-translation-gamma_t}] Fix $a\in(1/3,1/2)$ arbitrarily.
Combining decomposition \eqref{eq:decompo-Z_s(F,t,Delta_s)} with Lemma \ref{lem:convergence-of-fluctuations} (which uses $a>1/3$), it is enough to prove 
\[
	\sqrt{t} \widetilde{Z}_t^{t^a,\gamma} (F-\rho(F),\gamma) 
	\xrightarrow[t\to\infty]{(\P)} 0.
\]
But,using now $a<1/2$, this convergence follows from Lemma \ref{lem:control-rate-of-CV-killed-Z_s^(s,gamma_s)(F,Delta_s)} applied to $\delta = \varepsilon/\sqrt{t}$ and $M = \beta_t$, for any fixed $\ep > 0$.
\end{proof}

\subsection{Convergence of sums along a stopping line starting at time \texorpdfstring{$t^a$}{ta}}
\label{subsection:sums_along_stopping_line}

In this section, we prove preliminary results that will be useful in the proof of Lemma \ref{lem:convergence-of-fluctuations}.
The first one shows the convergence of sums along the stopping line $\cL^{t^a,\gamma}$ of a function of the hitting time, with a rate of convergence at least $1/\beta_t$.
Recall $\cL^{t^a,\gamma}_I$ denotes the set of particles killed during the time interval $I$.
\begin{lem} \label{lem:convergence-of-fluctuations-part-3-bis}
Let $\gamma = \frac{1}{2} \log t + \beta_t$, where $\beta_t \to \infty$ and $\beta_t = o(\log t)$ as $t \to \infty$.
Let $\Upsilon \colon [0,1] \to \C$ be a measurable function such that $\abs{\Upsilon(r)} \leq C r^{3/4}$, $r\in[0,1]$.
Then, for any $a \in (0,1)$, we have
\[
\beta_t \Biggl(
\e^{-\beta_t} \sum_{u \in \cL^{t^a,\gamma}_{(t^a,t]}} 
	\Upsilon(T_u/t)
- \int_0^1 \Upsilon(r) \frac{Z_\infty}{\sqrt{2\pi}} \frac{\diff r}{r^{3/2}}
\Biggr)
\xrightarrow[t\to\infty]{(\P)} 0.
\]
\end{lem}
\begin{proof}
By linearity of sum and integral, we can and will suppose that $\abs{\Upsilon(r)} \leq r^{3/4}$, $r\in[0,1]$.
The idea is to use a first and second moment calculation given $\sF_{t^a}$ to prove the sum converges to the integral faster than $1/\beta_t$.
But we apply this strategy only to particles killed after time $t_1 \coloneqq t \beta_t^{-6}$: we first consider
\begin{align*}
\Xi_t
\coloneqq 
\e^{-\beta_t} \sum_{u \in \cL^{t^a,\gamma}_{[t_1,t]}} 	
	\Upsilon(T_u/t)
- Z_{t^a}(1,\gamma) \int_{t_1/t}^1 \Upsilon(r) \frac{\diff r}{\sqrt{2\pi} r^{3/2}}.
\end{align*}
and show $\beta_t \Xi_t \to 0$ in probability by the means of a first and second moment calculation.
We start with the first moment: by the first part of \eqref{eq:first-moment-on-the-stopping-line-gamma}, we have
\begin{align*}
\E \Biggl[
\sum_{u \in \cL^{t^a,\gamma}_{[t_1,t]}}
\Upsilon(T_u/t)
\Bigg| \sF_{t^a} \Biggr]
& = \sum_{v \in \cN(t^a)}
	(X_v(t^a) - \gamma)_+ \e^{-(X_v(t^a) - \gamma)}  
	\int_{t_1}^t
	\frac{\e^{-(X_v(t^a) - \gamma)^2/2(s-t^a)}}{\sqrt{2\pi}(s-t^a)^{3/2}} 
	\Upsilon(s/t)
	\diff s.
\end{align*}
Changing variable with $r = s/t$ and using $\e^\gamma = \e^{\beta_t} \sqrt{t}$, we get
\begin{align*}
\abs{\Ecsq{\Xi_t}{\sF_{t^a}}} 
& \leq 
\sum_{v \in \cN(t^a)}
	(X_v(t^a) - \gamma)_+ \e^{-X_v(t^a)} 
	\int_{t_1/t}^1
	\abs{
	\frac{\e^{-(X_v(t^a) - \gamma)^2/2(rt-t^a)}}{(1-t^{a-1}/r)^{3/2}}
	- 1 } 
	\frac{\abs{\Upsilon(r)}}{\sqrt{2\pi}r^{3/2}}
	\diff r. 
\end{align*}
Now, for $y > 0$ and $r \in [t_1/t,1]$, we have 
\begin{align*}
\abs{\frac{\e^{-y^2/2(rt-t^a)}}{(1-t^{a-1}/r)^{3/2}} - 1}  
& \leq \e^{-y^2/2(rt-t^a)} \abs{\frac{1}{(1-t^{a-1}/r)^{3/2}}- 1} 
+ \abs{\e^{-y^2/2(rt-t^a)} - 1} \\
& \leq \abs{\frac{1}{(1-t^a/t_1)^{3/2}} - 1} 
+ \frac{y^2}{2(rt-t^a)}
\leq 2 \left( \frac{t^a}{t_1} + \frac{y^2}{t_1} \right)
= 2 \beta_t^6 \frac{t^a}{t} \left( 1 + \frac{y^2}{t^a} \right),
\end{align*}
for $t$ large enough, recalling that $t_1 = t \beta_t^{-6}$,
and this yields
\begin{align}
\abs{\Ecsq{\Xi_t}{\sF_{t^a}}} 
& \leq \sum_{v \in \cN(t^a)}
	(X_v(t^a) - \gamma)_+ \e^{-X_v(t^a)} 
	\cdot 2 \beta_t^6 \frac{t^a}{t}
	\left( 1 + \frac{(X_v(t^a) - \gamma)^2}{t^a} \right)
	\int_0^1 
	\frac{\abs{\Upsilon(r)}}{r^{3/2}}
	\diff r \nonumber \\
& \leq C \beta_t^6 t^{a-1}
Z_{t^a}( x \mapsto 1+x^2,\gamma), \label{pk}
\end{align}
because the integral is finite by assumption on $\Upsilon$.
We now deal with the second moment: by the branching property at time $t^a$ and since $\abs{\Upsilon} \leq 1$,
\begin{align*}
\Var \Biggl(
\sum_{u \in \cL^{t^a,\gamma}_{[t_1,t]}} \e^{-\beta_t}
	\Upsilon(T_u/t)
\Bigg| \sF_{t^a} \Biggr) 
\leq \e^{-2\beta_t}
\sum_{v \in \cN(t^a)}
	\1_{X_v(t^a) > \gamma}
	\Eci{X_v(t^a) - \gamma}{(\# \cL_{[t_1-t^a,t-t^a]})^2},
\end{align*}
using the notation of Section~\ref{subsection:number-of-killed-particles}.
Applying Lemma \ref{lem:second-moment-number-of-killed-particles}, it proves
\begin{align}
\Varsq{\Xi_t}{\sF_{t^a}}
& \leq C \e^{-2\beta_t}
\e^{\gamma} \frac{Z_{t^a}(1,\gamma)}{\sqrt{t_1}} 
= C \e^{-\beta_t} \beta_t^3
Z_{t^a}(1,\gamma). \label{pj}
\end{align}
For $\varepsilon >0$, we can now apply Markov's and Chebyshev's inequalities given $\sF_{t^a}$ to get
\begin{align*}
\Ppsq{\abs{\Xi_t} \geq \frac{\varepsilon}{\beta_t}}{\sF_{t^a}}
& \leq \frac{2 \beta_t}{\varepsilon} \abs{\Ecsq{\Xi_t}{\sF_{t^a}}}
+ \left( \frac{2\beta_t}{\varepsilon} \right)^2 \Varsq{\Xi_t}{\sF_{t^a}} \\
& \leq \left( \frac{C}{\varepsilon} \beta_t^7 t^{a-1}
+ \frac{C}{\varepsilon^2} \beta_t^5 \e^{-\beta_t}
\right)
Z_{t^a} (x \mapsto 1+x^2,\gamma),
\end{align*}
applying \eqref{pk} and \eqref{pj}.
But, it follows from Lemma~\ref{lem:tightness} that $(Z_{t^a} (x \mapsto 1+x^2,\gamma))_{t\geq1}$ is tight. Therefore, using that $\beta_t\to\infty$ and $\beta_t = o(\log t)$, this proves that
$\beta_t \Xi_t \to 0$ in probability
and concludes the first step.

The second step consists of dealing with the remaining terms, which are
\begin{align*}
\Xi_t^{(1)}
& \coloneqq 
\e^{-\beta_t} \sum_{u \in \cL^{t^a,\gamma}_{(t^a,t_1)}} 
\Upsilon(T_u/t), \\
\Xi_t^{(2)}
& \coloneqq 
Z_{t^a}(1,\gamma)
\int_0^{t_1/t} \Upsilon(r) \frac{\diff r}{\sqrt{2\pi} r^{3/2}}, \\
\Xi_t^{(3)}
& \coloneqq 
\left( Z_\infty-Z_{t^a}(1,\gamma) \right)
\int_0^1 \Upsilon(r) \frac{\diff r}{\sqrt{2\pi} r^{3/2}},
\end{align*}
so that the convergence of Lemma \ref{lem:convergence-of-fluctuations-part-3-bis} can be written $\beta_t(\Xi_t + \Xi_t^{(1)} - \Xi_t^{(2)} -\Xi_t^{(3)}) \to 0$ in probability.
For $\Xi_t^{(3)}$, we have $\beta_t (Z_{t^a}(1,\gamma)-Z_\infty) \to 0$ in probability by Proposition \ref{prop:rate-of-CV-Z_t(F,Delta,t)} and so $\beta_t\Xi_t^{(3)} \to 0$ in probability.
We now deal with $\Xi_t^{(2)}$: using $\abs{\Upsilon(r)} \leq r^{3/4}$, we have
\begin{align}
\abs{\int_0^{t_1/t} \Upsilon(r) \frac{\diff r}{r^{3/2}}}
& \leq \int_0^{t_1/t} \frac{\diff r}{r^{3/4}}
\leq C \left( \frac{t_1}{t} \right)^{1/4}
= C \beta_t^{-3/2}. \label{pn}
\end{align}
Hence, $\beta_t \lvert \Xi_t^{(2)} \rvert 
\leq C \beta_t^{-1/2} Z_{t^a}(1,\gamma) \to 0$,
because $Z_{t^a}(1,\gamma) \to Z_\infty$.
Finally, we deal with $\Xi_t^{(1)}$, by bounding its first moment given $\sF_{t^a}$.
For particles $u$ with $T_u \in (t^a,2t^a]$, we use $\abs{\Upsilon(T_u/t)} \leq C t^{3(a-1)/4}$ and \eqref{eq:number-of-killed-particles-given-sF_s^a} to bound the number of such particles.
For particles $u$ with $T_u \in (2t^a,t_1]$, we apply \eqref{eq:first-moment-on-the-stopping-line-gamma}.
This yields
\begin{align*}
\Ecsq{\bigl\lvert \Xi_t^{(1)} \bigr\rvert}{\sF_{t^a}} 
& \leq C \e^{-\beta_t} 
	\left( 
	\e^{\gamma} W_{t^a} t^{3(a-1)/4}
	+ \e^{\gamma} Z_{t^a}(1,\gamma) 
	\int_{2t^a}^{t_1} \frac{\abs{\Upsilon(s/t)}}{s^{3/2}} \diff s 
	\right) \\
& \leq C Z_{t^a}(1+x^{-1},\gamma) 
	\left( 
	t^{(a-1)/4}
	+ \int_0^{t_1/t} \frac{\abs{\Upsilon(r)}}{r^{3/2}} \diff r
	\right) \\
& \leq C Z_{t^a}(1+x^{-1},\gamma) \beta_t^{-3/2},
\end{align*}
applying again \eqref{pn}. It follows that $\beta_t \Xi_t^{(1)} \to 0$ in probability (using Lemma~\ref{lem:tightness}).
Therefore, we showed $\beta_t \Xi_t^{(i)} \to 0$ in probability for each $i \in \{1,2,3\}$ and this concludes the proof.
\end{proof}
The second result of this section is a consequence of the first one. For any $u \in \cL^{t^a,\gamma}$, we define $Z_\infty^{(u)}$ as the limit of the derivative martingale of the BBM rooted at $u$ at time $T_u$, that is
\begin{align} \label{eq:def_Z_infty^u}
Z_\infty^{(u)} \coloneqq 
\lim_{s\to\infty}
\sum_{v \in \cN(s) \text{ s.t.\@ } u \leq v} 
(X_v(s)-X_u(T_u)) \e^{-(X_v(s)-X_u(T_u))}.
\end{align}
By the branching property, given $\sF_{\cL^{t^a,\gamma}}$, the $Z_\infty^{(u)}$ for $u \in \cL^{t^a,\gamma}$ are independent copies of $Z_\infty$.
The following result shows the convergence in distribution of sums of rescaled versions of the $Z_\infty^{(u)}$ along $\cL^{t^a,\gamma}$, after an appropriate centering.
\begin{cor} \label{cor:convergence-of-fluctuations-part-2}
Let $\gamma = \frac{1}{2} \log t + \beta_t$, where $\beta_t \to \infty$ and $\beta_t = o(\log t)$ as $t \to \infty$.
Let $g \colon [0,1] \to \R$ be a measurable function such that $\abs{g(r)} \leq Cr$, $r\in[0,1]$.
For any $a \in (0,1)$, we have 
\begin{align} \label{eq:convergence-of-fluctuations-part-2}
\sum_{u \in \cL^{t^a,\gamma}_{(t^a,t]}} \e^{-\beta_t}
	g(T_u/t) \left( Z_\infty^{(u)} - \beta_t \right)
\xrightarrow[t\to\infty]{\text{(law)}} \int_0^1 g(r) M_{Z_\infty}(\diff r).
\end{align}
\end{cor}
\begin{proof}
	By linearity of sum and stable integral w.r.t. $g$, we can and will suppose that $\abs{g(r)} \leq r$, $r\in[0,1]$.
For brevity, we denote by $\Gamma_t$ the sum on the left-hand side of \eqref{eq:convergence-of-fluctuations-part-2}.
Recalling the definition of $M_{Z_\infty}$ from \eqref{eq:def_M_Z_infty}, it is sufficient to prove that
\begin{align} \label{pr}
\Ecsq{\e^{i \Gamma_t}}{\sF_{\cL^{t^a,\gamma}}}
\xrightarrow[t\to\infty]{(\P)} 
\exp \left( - \int_0^\infty 
\left[ \abs{g(r)} + i \frac{2}{\pi} g(r) \left( \log \abs{g(r)} - \mu_Z \right) \right] 
\frac{\sqrt{\pi}}{2\sqrt{2}} \frac{Z_\infty}{r^{3/2}} \diff r \right).
\end{align}
Indeed, since we can replace $g$ by $\xi g$ for any $\xi \in \R$, this implies the convergence of the characteristic function of $\Gamma_t$.
By the branching property, we have
\begin{align}
\Ecsq{\e^{i \Gamma_t}}{\sF_{\cL^{t^a,\gamma}}}
= \prod_{u \in \cL^{t^a,\gamma}_{(t^a,t]}} 
	\Psi_{Z_\infty} \left( \e^{-\beta_t} g(T_u/t) \right) 
	\exp \left(- i \beta_t\e^{-\beta_t} g(T_u/t) \right), \label{pa}
\end{align}
where $\Psi_{Z_\infty}$ denotes the characteristic function of $Z_\infty$.
It has been shown in \cite[Equation (1.12)]{maillardpain2019} (see also Footnote~\ref{foot:Zinfty}) that there exists a continuous function $\psi \colon \R \to \C$, with $\psi(0) = 0$, such that for any $\xi \in \R$,
\begin{align} \label{eq:characteristic-function-Z_infty}
\Psi_{Z_\infty}(\xi)
= \exp \left( 
- \frac{\pi}{2} \abs{\xi} - i \xi (\log \abs{\xi} - \mu_Z) + \xi \psi(\xi)
\right).
\end{align}
Using this expansion of $\Psi_{Z_\infty}$ around 0, we get, for any $r \in [0,1]$,
\begin{align*}
\Psi_{Z_\infty} \left( \e^{-\beta_t} g(r) \right) 
	\exp \left(- i \beta_t \e^{-\beta_t} g(r) \right) 
& = \exp \left( \e^{-\beta_t} G (r)
+ \e^{-\beta_t} g(r) 
	\psi \left( \e^{-\beta_t} g(r) \right) 
\right),
\end{align*}
with $$G (r) \coloneqq - \frac{\pi}{2} \abs{g(r)} - i g(r) (\log \abs{g(r)} - \mu_Z).$$
Therefore, \eqref{pa} becomes
\begin{align*}
\Ecsq{\e^{i \Gamma_t}}{\sF_{\cL^{t^a,\gamma}}}
= \exp \Biggl( \e^{-\beta_t}
\sum_{u \in \cL^{t^a,\gamma}_{(t^a,t]}}
\left[
G (T_u/t)
+ g(T_u/t) \psi \left( \e^{-\beta_t} g(T_u/t) \right)
\right]
\Biggr).
\end{align*}
Then, using that, for any $z_1,z_2 \in \{ z \in \C : \mathrm{Re}(z) \leq 0 \}$, $\lvert \e^{z_1} - \e^{z_2} \rvert \leq \abs{z_1 - z_2}$, we get
\begin{align}
\Bigg\lvert
\Ecsq{\e^{i \Gamma_t}}
	{\sF_{\cL^{t^a,\gamma}}}
- \exp \Biggl( \e^{-\beta_t}
	\sum_{u \in \cL^{t^a,\gamma}_{(t^a,t]}}
	G(T_u/t)
	\Biggr)
\Bigg\rvert 
& \leq \e^{-\beta_t} 
\sum_{u \in \cL^{t^a,\gamma}_{(t^a,t]}}
\abs{g(T_u/t) 
\psi \left( \e^{-\beta_t} g(T_u/t) \right)} \nonumber \\
& \leq \varepsilon(t) \e^{-\beta_t} 
\sum_{u \in \cL^{t^a,\gamma}_{(t^a,t]}}
\abs{g(T_u/t)}, \label{pc}
\end{align}
with $\varepsilon(t) \to 0$ as $t \to \infty$, using that $\abs{g(T_u/t)} \leq 1$ and $\psi(\xi) \to 0$ as $\xi \to 0$.
Then it follows from Lemma \ref{lem:convergence-of-fluctuations-part-3-bis} applied to $\Upsilon = \abs{g}$ that the right-hand side of \eqref{pc} tends to 0 in probability.
On the other hand, since $\abs{g(r)} \leq r$, we have $\abs{G(r)} \leq Cr^{3/4}$.
Therefore, we can apply Lemma \ref{lem:convergence-of-fluctuations-part-3-bis} again but to the function $\Upsilon = G$ to get
\begin{align} \label{px}
\exp \Biggl( \e^{-\beta_t}
	\sum_{u \in \cL^{t^a,\gamma}_{(t^a,t]}}
	G(T_u/t)
	\Biggr)
\xrightarrow[t\to\infty]{(\P)} 
\exp \left( \int_0^1 G(r) \frac{Z_\infty}{\sqrt{2\pi}}\frac{\diff r}{r^{3/2}} \right).
\end{align}
The right-hand side of \eqref{px} coincides with the right-hand side of \eqref{pr} so this concludes the proof.
\end{proof}

\subsection{Convergence of the contributions of the killed particles}
\label{subsection:proof-of-proposition-convergence-of-fluctuations}

In this section, we prove Lemma \ref{lem:convergence-of-fluctuations}, that is the convergence in law of the sum of contributions~$\overline{\Omega}^{(u)}_t$ of the particles killed by the barrier.

\begin{proof}[Proof of Lemma \ref{lem:convergence-of-fluctuations}]
First note that, with high probability, no particles are killed exactly at time $t^a$: as a consequence of \eqref{eq:local-min-of-the-BBM}, we have 
\begin{align} \label{eq:killed_at_t^a}
\P(\cL^{t^a,\gamma}_{\{t^a\}} \neq \emptyset ) 
= \Pp{\min_{u \in \cN(t^a)} X_u(t^a) \leq \gamma}
\leq C (\log t)^2 \e^{\gamma} t^{-3a/2} 
\xrightarrow[t \to \infty]{} 0,
\end{align}
since $a > 1/3$.
Moreover, setting $t_2 = t - t \e^{-2\beta_t}$, with high probability, no particles are killed between times $t_2$ and $t$: indeed, it follows from \eqref{eq:first-moment-on-the-stopping-line-gamma} that for $t$ large enough,
\begin{align} \label{eq:killed_between_t_2_and_t}
\Ecsq{\# \cL^{t^a,\gamma}_{(t_2,t]}}{\sF_{t^a}}
& \leq C \e^{\gamma} Z_{t^a}(1,\gamma) \int_{t_2}^{t} \frac{\diff s}{(s-t^a)^{3/2}} \nonumber \\
&\leq C \e^{\gamma} Z_{t^a}(1,\gamma) \frac{t-t_2}{t^{3/2}}
= C Z_{t^a}(1,\gamma)  \e^{-\beta_t}
\xrightarrow[t\to\infty]{(\P)} 0,
\end{align}
since $(Z_{t^a}(1,\gamma))_{t\geq 1}$ is tight by Lemma \ref{lem:tightness}.
Then, recall from \eqref{eq:def R F} that $\mathscr R F(r) = \rho(F_r) \1_{r<1} -\rho(F)$ with $F_r(x) = F(\sqrt{1-r}\cdot x)$, and from \eqref{eq:def_Z_infty^u} the definition of $Z_\infty^{(u)}$.
We now aim at proving
\begin{align} \label{eq:simplifying-the-contribution}
\sqrt{t} 
\sum_{u \in \cL^{t^a,\gamma}_{(t^a,t_2]}} 
\abs{ \overline{\Omega}^{(u)}_t - \e^{-\gamma} \mathscr R F(T_u/t) Z_\infty^{(u)} }
\xrightarrow[t\to\infty]{(\P)} 0.
\end{align}
Given $\sF_{\cL^{t^a,\gamma}}$, the $\overline{\Omega}^{(u)}_t- \e^{-\gamma} \mathscr R F(T_u/t) Z_\infty^{(u)}$ for $u \in \cL^{t^a,\gamma}_{(t^a,t]}$ are independent with the same law as $\e^{-\gamma} (Z_{t-T_u} (F_{T_u/t}-\rho(F),0) - \mathscr R F(T_u/t) Z_\infty)$ under $\P$.
Hence, we have
\begin{align} \label{ja}
\E \Biggl[ \Biggl( \sqrt{t} 
\sum_{u \in \cL^{t^a,\gamma}_{(t^a,t_2]}} 
\abs{ \overline{\Omega}^{(u)}_t - \e^{-\gamma} \mathscr R F(T_u/t) Z_\infty^{(u)} }
\Biggr) \wedge 1 
\Bigg| \sF_{\cL^{t^a,\gamma}} \Biggr]
\leq \sum_{u \in \cL^{t^a,\gamma}_{(t^a,t_2]}} \chi(T_u),
\end{align}
where we set 
$$
\chi(s) \coloneqq \Ec{\left(\e^{-\beta_t} \left\lvert Z_{t-s} (F_{s/t}-\rho(F),0) - \mathscr R F(s/t) Z_\infty \right\rvert\right) \wedge 1}.
$$
Note that, for any $r \in [0,1)$, the function $F_r$ satisfies Assumptions \ref{ass1}-\ref{ass2} and $\mathscr R F(r) = \rho(F_r-\rho(F))$, so we can apply Corollary \ref{cor:rate_of_convergence_as_expectation} to get that the right-hand side of \eqref{ja} is smaller than
\begin{align*} 
C (\log t)^3 (t-t_2)^{-1/3} \e^{-\beta_t} \# \cL^{t^a,\gamma}_{(t^a,t_2]}
= C (\log t)^3 t^{-1/3} \e^{-\beta_t/3} \# \cL^{t^a,\gamma}_{(t^a,t_2]}.
\end{align*}
Then, we take the conditional expectation given $\sF_{t^a}$ and apply \eqref{eq:number-of-killed-particles-given-sF_s^a} to get
\begin{align*}
\E \Biggl[ \Biggl( \sqrt{t} 
\sum_{u \in \cL^{t^a,\gamma}_{(t^a,t_2]}} 
\abs{ \overline{\Omega}^{(u)}_t - \e^{-\gamma} \mathscr R F(T_u/t) Z_\infty^{(u)} }
\Biggr) \wedge 1 
\Bigg| \sF_{t^a} \Biggr]
\leq C (\log t)^3 t^{-1/3} \e^{-\beta_t/3} \e^\gamma W_{t^a}
\xrightarrow[t\to\infty]{(\P)} 0,
\end{align*}
using \eqref{eq:cv-of-W_t}, $\gamma = \frac{1}{2} \log t  + \beta_t$ with $\beta_t = o(\log t)$ and $a > 1/3$. This proves \eqref{eq:simplifying-the-contribution}.
Combining \eqref{eq:killed_at_t^a}, \eqref{eq:killed_between_t_2_and_t} and \eqref{eq:simplifying-the-contribution}, the convergence we have to prove is equivalent to 
\begin{align*}
\sum_{u \in \cL^{t^a,\gamma}_{(t^a,t]}} \e^{-\beta_t} 
	\mathscr R F(T_u/t) Z_\infty^{(u)}
- \beta_t \int_0^1 \mathscr R F(r) \frac{Z_\infty}{\sqrt{2\pi}} \frac{\diff r}{r^{3/2}}
\xrightarrow[t\to\infty]{\text{(law)}} 
\int_0^1 \mathscr R F(r) M_{Z_\infty}(\diff r).
\end{align*}
This follows from Lemma \ref{lem:convergence-of-fluctuations-part-3-bis} and Corollary \ref{cor:convergence-of-fluctuations-part-2} applied respectively with $\Upsilon = g = \mathscr R F$, using that $\abs{\mathscr R F(u)} \leq C(\kappa) u$ for all $u \in [0,1]$, by \eqref{eq:RF_bound}.
\end{proof}

\section{Further results towards the final theorem}
\label{section:preliminary_results_towards_the_final_theorem}

The goal of this section is to prove several results that will be useful for the proofs of Theorems~\ref{theorem-complete} and \ref{theorem-2} in Section~\ref{section:proof_theorem}.
In particular, we start working with functions $F$ that can diverge at 0 as $x^{-\alpha}$ for $\alpha \in [0,2)$.
To keep track of uniformity of results, we introduce this new set of assumptions, for some $\alpha,\kappa \geq 0$:
\begin{enumerate}[label=(H\arabic*$_{\alpha,\kappa}$),leftmargin=*]
\item For any $x > 0$, $\abs{F(x)} \leq x^{-\alpha} \e^{\kappa x}$; \label{ass1_alpha}
\item $F$ is differentiable on $(0,\infty)$ and, for any $x>0$, $\abs{F'(x)} \leq x^{-\alpha-1} \e^{\kappa x}$; \label{ass2_alpha}
\item $F$ is twice differentiable on $(0,\infty)$ and, for any $x>0$, $\abs{F''(x)} \leq x^{-\alpha-2} \e^{\kappa x}$. \label{ass3_alpha}
\end{enumerate}

\subsection{Rate of convergence for functions with a divergence at 0}
\label{subsection:bad_functions}

In this subsection, we prove a first estimate, that generalizes Proposition \ref{prop:rate-of-CV-Z_t(F,Delta,t)} to functions $F$ which can be diverging at 0 as $x^{-\alpha}$ for $\alpha \in [0,3)$ (note that we have to go up to 3 here in order to prove Proposition \ref{prop:from-translation-gamma_t-to-0} later), but with a slower rate of convergence.
\begin{prop} \label{prop:cv_rate_for_bad_functions}
Let $\alpha \in [0,3)$, $\kappa \geq 0$ and $K > 0$. 
There exist $C = C(\kappa,\alpha,K) >0$ and $t_0 = t_0(\kappa,\alpha,K) > 0$ such that
for any $F \colon \R \to \R$ satisfying Assumptions \ref{ass1_alpha}-\ref{ass2_alpha} and any $t \geq t_0$, $\Delta \in [0,K \log t]$ and $\delta > 0$, 
\begin{align*}
\Pp{\bigl\lvert Z_t(F,\Delta) - \rho(F) Z_\infty \bigr\rvert \geq \delta}
& \leq t^{-K} + \frac{C(\log t)^2}{\delta t^{(3-\alpha)/12}}
+ \frac{C (\log t)^2 \e^\Delta}{t^{3/2-K^{-1}}} \1_{\alpha \geq 2}.
\end{align*}
\end{prop}
The main idea of the proof is to apply Proposition \ref{prop:rate-of-CV-Z_t(F,Delta,t)} to a modified version of $F$ that satisfies Assumptions \ref{ass1}-\ref{ass2}.
We introduce the function 
\[
\widetilde{F}_t(x) \coloneqq F(x \vee x_t) 
\]
where $x_t \to 0$ as $t\to\infty$ and will be chosen explicitly later but we assume for now that $(\log t)^2/\sqrt{t} \leq x_t \leq 1$.
We first prove in the following lemma that we can replace $F$ by $\widetilde{F}_t$.
\begin{lem} \label{lem:cv_rate_for_bad_functions_1}
Let $\alpha \in [0,3)$, $\kappa \geq 0$ and $K > 0$. 
There exist $C = C(\kappa,\alpha,K) >0$ and $t_0 = t_0(\kappa,\alpha,K) > 0$ such that
for any $F \colon \R \to \R$ satisfying Assumption \ref{ass1_alpha} and any $t \geq t_0$, $\Delta \in [0,K \log t]$ and $\delta > 0$, we have 
\begin{align}
\Pp{\abs{Z_t(F-\widetilde{F}_t,\Delta)} 
\geq \delta} 
& \leq t^{-K} + \frac{C\log t}{\delta} x_t^{3-\alpha}
+ \frac{C (\log t)^2 \e^\Delta}{t^{3/2-K^{-1}}} \1_{\alpha \geq 2}, \label{eq:Ftilde_part1}\\
\Pp{\abs{\rho(F-\widetilde{F}_t)} Z_\infty
	\geq \delta} 
& \leq t^{-K} + \frac{C}{\delta} x_t^{3-\alpha}.
\label{eq:Ftilde_part2}
\end{align}
\end{lem}
\begin{proof} 
We start with \eqref{eq:Ftilde_part1}.
We first work in the case $\alpha < 2$ and bound the first moment of $Z_t(F-\widetilde{F}_t,\Delta)$ on the event $A \coloneqq \{\forall r \in [0,t], \min_{u \in \cN(r)} X_u(r) > - M \}$, where we set $M = K \log t$. 
By Assumption \ref{ass1_alpha}, we have $\lvert F-\widetilde{F}_t \rvert(x) \leq C x^{-\alpha} \1_{x \leq x_t}$.
By the many-to-one lemma \eqref{eq:many-to-one-bessel}, we have, setting $\eta = (\Delta+M)/\sqrt{t}$,
\begin{align}
\Ec{\abs{Z_t(F-\widetilde{F}_t,\Delta)} \1_A}
& \leq \e^M \Eci{M}{\widetilde{Z}_t(\lvert F-\widetilde{F}_t \rvert,\Delta+M)} 
\nonumber \\
& = M \Eci{M}{ \frac{(R_t-\Delta-M)_+}{R_t}  \cdot 
\lvert F-\widetilde{F}_t \rvert \left( \frac{R_t-\Delta-M}{\sqrt{t}} \right)} 
\nonumber \\
& \leq CM \Eci{M/\sqrt{t}}{ \frac{1}{R_1 (R_1-\eta)^{\alpha-1}}
\1_{0 < R_1-\eta \leq x_t}}. \label{hhf}
\end{align}
Using the density of $R_1$ in \eqref{eq:density-of-R_1-under-P_x} and bounding $\e^{-(z-y)^2/2} (1-\e^{-2zy}) \leq 2zy$, we get that the last expectation is smaller than 
\begin{equation}
\int_\eta^{x_t+\eta} \frac{C z \diff z}{(z-\eta)^{\alpha-1}} 
\leq C (x_t+\eta) \int_0^{x_t} \frac{\diff u}{u^{\alpha-1}} 
\leq C (x_t+\eta) x_t^{2-\alpha}, \label{hha}
\end{equation}
where we used that $\alpha < 2$.
Finally, noting that $\eta \leq x_t$ for $t$ large enough, we get 
\begin{equation}
\Ec{\abs{Z_t(F-\widetilde{F}_t,\Delta)} \1_A}
\leq C(\log t) x_t^{3-\alpha}. \label{hhb}
\end{equation}
Using that $\P(A^c) \leq \e^{-M}$ by \eqref{eq:global-min-of-the-BBM}, we obtain \eqref{eq:Ftilde_part1} when $\alpha<2$.

In the case $\alpha \in [2,3)$, we have to remove some particles before taking the first moment, because the integral on the left-hand side of \eqref{hha} is infinite when $\alpha \geq 2$.
For this, let $B \coloneqq \{ \min_{u\in\cN(t)} X_u(t) - \Delta > (\log t)/K \}$, and the same calculation shows
\begin{align*}
\Ec{\abs{Z_t(F-\widetilde{F}_t,\Delta)} \1_{A\cap B}}
& \leq CM \Eci{M/\sqrt{t}}{ \frac{1}{R_1 (R_1-\eta)^{\alpha-1}}
\1_{(\log t)/(K\sqrt{t}) < R_1-\eta \leq x_t}} \\
& \leq CM \int_{\eta+(\log t)/(K\sqrt{t})}^{x_t+\eta} \frac{z \diff z}{(z-\eta)^{\alpha-1}}
\leq C(\log t) x_t^{3-\alpha},
\end{align*}
where we use that for $z \geq \eta+(\log t)/(K\sqrt{t})$, we have $z-\eta \geq c z$, where the constant $c$ depends on $K$.
Recalling that $\P(A^c) \leq \e^{-M}$ and using also that
$\P(B^c) \leq C(\log t)^2 \e^\Delta t^{-3/2+K^{-1}}$ by \eqref{eq:local-min-of-the-BBM}, we obtain \eqref{eq:Ftilde_part1}.

We now prove \eqref{eq:Ftilde_part2}. 
Using again $\lvert F-\widetilde{F}_t \rvert(x) \leq C x^{-\alpha} \1_{x \leq x_t}$ and the density of $R_1$ in \eqref{eq:density-of-R_1-under-P}, we get
\begin{equation} \label{eq:F-Ftilde-R_1}
\abs{ \rho(F-\widetilde{F}_t) }
\leq \int_0^{x_t} \frac{C}{z^\alpha} z^2 \diff z 
\leq C x_t^{3-\alpha}.
\end{equation}
Together with the fact that $\P(Z_\infty > x) \leq C/x$ (see \cite[Chap.\@ 2 Prop.\@ 4.1]{maillard2012thesis}), this proves~\eqref{eq:Ftilde_part2}.
\end{proof}
We can now prove Proposition \ref{prop:cv_rate_for_bad_functions} by applying Proposition \ref{prop:rate-of-CV-Z_t(F,Delta,t)} to a rescaled version of $\widetilde{F}_t$
\begin{proof}[Proof of Proposition \ref{prop:cv_rate_for_bad_functions}]
	By Assumption \ref{ass1_alpha}, we have 
	$\lvert \widetilde{F}_t(x) \rvert 
	\leq x_t^{-\alpha} e^{\kappa} \cdot e^{\kappa x}$ for any $x > 0$.
	Moreover, for any $0 < y \leq x$, by Assumption \ref{ass2_alpha} and distinguishing three cases ($x \leq x_t$, $y \leq x_t < x$ and $x_t < y$), we get that
	$\lvert \widetilde{F}_t(x) - \widetilde{F}_t(y) \rvert
	\leq x_t^{-\alpha-1} (x-y) \e^{\kappa x}$.
	This proves that the function $x_t^{\alpha+1} \widetilde{F}_t$ satisfies Assumptions \ref{ass1}-\ref{ass2}, hence we can apply Proposition~\ref{prop:rate-of-CV-Z_t(F,Delta,t)} to this function to get 
	\begin{align*}
	\Pp{\abs{Z_t(\widetilde{F}_t,\Delta)- \rho(\widetilde{F}_t) Z_\infty} 
	\geq \delta}
	\leq t^{-K} + C(\log t)^2 t^{-1/3} x_t^{-1-\alpha} \delta^{-1}.
	\end{align*}
	We combine this with Lemma \ref{lem:cv_rate_for_bad_functions_1}, choosing $x_t = t^{-1/12}$, which is the choice given by the equation $x_t^{3-\alpha} = t^{-1/3} x_t^{-1-\alpha}$. 
	This proves the result.
\end{proof}
Finally, we state a corollary of Proposition \ref{prop:cv_rate_for_bad_functions}. We only need it in the case $\alpha \in (0,2)$ and do not try to optimize the exponent of $t$ in the bound.
\begin{cor} \label{cor:rate_of_convergence_as_expectation_for bad_functions}
Let $\alpha \in [0,2)$, $\kappa \geq 0$ and $K > 0$. 
There exists $t_0 = t_0(\kappa,\alpha,K) > 0$ such that
for any $F \colon \R \to \R$ satisfying Assumptions \ref{ass1_alpha}-\ref{ass2_alpha} and any $t \geq t_0$, $\Delta \in [0,K \log t]$ and $\ep \geq t^{-K}$,
\[
\E\left[ \left( \ep \abs{ Z_t(F,\Delta) - \rho(F)Z_\infty} \right) \wedge 1 \right] 
\le t^{-1/12} \ep.
\]
\end{cor}
\begin{proof}
We get $\E[( \ep \abs{ Z_t(F,\Delta) - \rho(F)Z_\infty}) \wedge 1] \le C (\log t)^3 t^{(\alpha-3)/12} \ep$ for some $C = C(\kappa,\alpha,K)$ by following the proof of Corollary \ref{cor:rate_of_convergence_as_expectation} (using Proposition \ref{prop:cv_rate_for_bad_functions} instead of Proposition \ref{prop:rate-of-CV-Z_t(F,Delta,t)}). The result follows from the fact that $\alpha < 2$.
\end{proof}

\subsection{Precise effect of a shift}
\label{subsection:effect_of_a_shift}

\begin{prop} \label{prop:from-translation-gamma_t-to-0}
Let $\alpha \in [0,2)$ and $\kappa \geq 0$.
Let $F$ be a function satisfying Assumptions \ref{ass1_alpha}-\ref{ass2_alpha}-\ref{ass3_alpha}.
Let $\varepsilon > 0$.
For all $\Delta = \Delta(t) \in [0,(\frac{3}{2} - \ep) \log t]$, we have the following convergence in probability
\begin{align*}
\sqrt{t} \cdot \left(
	Z_t(F) - Z_t (F,\Delta) 
	+ \frac{\Delta}{\sqrt{t}} 
	\int_0^\infty \mathscr R F(r) \frac{Z_\infty}{\sqrt{2\pi}} \frac{\diff r}{r^{3/2}}
	\right)
\xrightarrow[t\to\infty]{} 0.
\end{align*}
\end{prop}

\begin{proof}
We first note that
\begin{align} \label{eq:two-constants-are-identical}
\frac{1}{\sqrt{2\pi}} \int_0^\infty \mathscr R F(r) \frac{\diff r}{r^{3/2}}
= - \Ec{\frac{F(R_1)}{R_1} + F'(R_1)},
\end{align}
which is proved by direct calculation in Appendix \ref{section:two-constants-are-identical}.
Let $E_t \coloneqq \{ \min_{u \in \cN(t)} X_u(t) > \Delta+1 \}$. 
Using \eqref{eq:local-min-of-the-BBM}, we have $\P((E_t)^c) \leq C (\log t)^2 t^{-\ep} \to 0$ as $t \to \infty$ and, therefore it is now sufficient to prove that
\begin{align} \label{na}
\1_{E_t} \sqrt{t} \cdot 
\left( Z_t(F) - Z_t (F,\Delta) - \Ec{\frac{F(R_1)}{R_1} + F'(R_1)} Z_\infty \frac{\Delta}{\sqrt{t}} \right)
& \xrightarrow[t\to\infty]{(\P)} 0.
\end{align}
Firstly, on event $E_t$, we have $Z_t(F) - Z_t (F,\Delta) = Z_t(\Phi_t,\Delta)$, where we set, for $x \in \R$, 
\[
\Phi_t (x) \coloneqq \frac{1}{x} [(x+\delta) F(x+\delta) - x F(x)] \1_{x>0}, 
\qquad \text{with} \qquad \delta \coloneqq \frac{\Delta}{\sqrt{t}}.
\]
Moreover, let $\Phi (x) \coloneqq \frac{1}{x} (F(x) + x F'(x)) \1_{x>0}$.
Note that, to first order, $\Phi_t$ behaves as $\delta \Phi$.
With this in mind, we split the proof of \eqref{na} into the two following convergences in probability: 
\begin{align} 
\1_{E_t} \sqrt{t} \cdot Z_t(\Phi_t-\delta \Phi,\Delta)
& \xrightarrow[t\to\infty]{} 0, \label{nc} \\
\Delta \cdot 
\left( Z_t(\Phi,\Delta) - \rho(\Phi) Z_\infty \right)
& \xrightarrow[t\to\infty]{} 0. \label{nd}
\end{align}
For \eqref{nd}, we simply note that it follows from Assumptions \ref{ass1_alpha}-\ref{ass2_alpha}-\ref{ass3_alpha} for $F$, that $\Phi/3$ satisfies Assumptions \hyperref[ass1_alpha]{(H1$_{\alpha+1,\kappa}$)}-\hyperref[ass2_alpha]{(H2$_{\alpha+1,\kappa}$)} and therefore we can apply 
Proposition \ref{prop:cv_rate_for_bad_functions} to $\Phi/3$ (since $\alpha+1 < 3$) and this proves \eqref{nd}.
We now prove \eqref{nc}.
For $x > 0$, by Taylor--Lagrange inequality, we have
\begin{align*}
\abs{x(\Phi_t(x) - \delta \Phi(x))}
\leq \frac{\delta^2}{2} 
\sup_{y \in [x,x+\delta]} \abs{\frac{\diff^2}{\diff y^2} (yF(y))}
\leq \frac{\delta^2}{2} \cdot 3 x^{-\alpha-1} \e^{\kappa(x+\delta)},
\end{align*}
using Assumptions \ref{ass2_alpha}-\ref{ass3_alpha} for $F$.
Therefore we get
$\abs{\Phi_t(x) - \delta \Phi(x)} \leq C \delta^2 x^{-\alpha-2} \e^{\kappa x}$ for any $x>0$.
Note that this is a strong divergence at 0, which is not integrable against $\rho$ if $\alpha \geq 1$, so that we don't expect $Z_t(x \mapsto x^{-\alpha-2} \e^{\kappa x},\Delta)$ to be tight.
However, we will bound its first moment, on the event $A \cap E_t$, where $A \coloneqq \{ \forall r \in [0,t], \min_{u \in \cN(r)} X_u(r) > - M \}$ with $M  \coloneqq \log t$.
Using the many-to-one formula \eqref{eq:many-to-one-bessel}, we get
\begin{align*} 
\Ec{\abs{Z_t(\Phi_t - \delta \Phi, \Delta)} \1_{A\cap E_t}}
& \leq M \Eci{M}{\frac{(R_t-\Delta-M)_+}{R_t} 
\cdot \abs{\Phi_t - \delta \Phi}\left( \frac{R_t-\Delta-M}{\sqrt{t}} \right)
\1_{R_t > \Delta+1+M}} \\
& \leq CM \delta^2 
\Eci{M/\sqrt{t}}{ \frac{\e^{\kappa(R_1-\eta)}}{R_1 (R_1-\eta)^{\alpha+1}} 
\1_{R_1 > \eta + t^{-1/2}}},
\end{align*}
using $\abs{\Phi_t(x) - \delta \Phi(x)} \leq C \delta^2 x^{-\alpha-2} \e^{\kappa x}$
and setting $\eta \coloneqq (\Delta+M)/\sqrt{t}$.
We now consider $t$ large enough so that $t^{-1/2} < \eta < 1/2$.
Using the density of $R_1$ in \eqref{eq:density-of-R_1-under-P_x} and $1-e^{-2zy} \leq 2zy$, the last expectation is smaller than
\begin{align*} 
\int_{\eta + t^{-1/2}}^1 \frac{C z \diff z}{(z-\eta)^{\alpha+1}} 
+ \int_1^\infty C \e^{\kappa z} \e^{-(z-M/\sqrt{t})^2/2} \diff z
& \leq \int_{\eta + t^{-1/2}}^{2\eta} \frac{C \eta \diff z}{(z-\eta)^{\alpha+1}} 
+ \int_{2\eta}^1 \frac{C \diff z}{z^\alpha} 
+ C \\
& \leq C \eta t^{\alpha/2} + \eta^{1-\alpha} + C
\leq C (\log t) t^{(\alpha-1)/2},
\end{align*}
using $\eta \leq C(\log t)/\sqrt{t}$.
Hence, since $\delta \leq C(\log t)/\sqrt{t}$, we get
\begin{align*} 
\sqrt{t} \cdot \Ec{\abs{Z_t(\Phi_t - \delta \Phi, \Delta)} \1_{A\cap E_t}}
& \leq C (\log t)^4 t^{\alpha/2 - 1}
\xrightarrow[t\to\infty]{} 0,
\end{align*}
which proves \eqref{nc} and concludes the proof, because $\P(A^c) \leq \e^{-M} \to 0$ by \eqref{eq:global-min-of-the-BBM}.
\end{proof}

\subsection{Convergence of sums along a stopping line starting at time \texorpdfstring{$ht$}{ht}}
\label{subsection:sums_along_stopping_line_2}

In this section, we prove two results estimating the rate of convergence of sums along the stopping line $\cL^{ht,\gamma}$ of a function of the hitting time. 
They are slightly different versions of Lemma \ref{lem:convergence-of-fluctuations-part-3-bis} because the stopping line starts here at time $ht$ instead of time $t^a$ and, in the second result, we allow the test function to diverge close to 1.
We also present corollaries for the convergence in distribution of sums of rescaled independent copies of $Z_\infty$ along $\cL^{ht,\gamma}$.
\begin{lem} \label{lem:sums_along_stopping_line}
Let $\gamma = \frac{1}{2} \log t + \beta_t$ with $\beta_t \in [0,10 \log t]$.
Let $\Upsilon \colon [0,\infty) \to \C$ be a measurable bounded function.
Let $h > 0$.
There exist $t_0 > 0$ a universal constant and $C = C(h) > 0$, such that, for any $\delta > 0$ and $t \geq t_0$,
\begin{align*}
\P \Biggl( \Biggl\lvert
\e^{-\beta_t} \sum_{u \in \cL^{ht,\gamma}_{(ht,\infty)}} 
	\Upsilon(T_u/t)
-  \int_h^\infty \Upsilon(r) \frac{Z_\infty}{\sqrt{2\pi}} \frac{\diff r}{r^{3/2}}
\Biggr\rvert
\geq \delta
\Biggr)
\leq C \left( \e^{-\beta_t} 
+ \frac{(\log t)^2 \norme{\Upsilon}_\infty}{t^{1/6} \delta} 
+ \frac{\beta_t \norme{\Upsilon}_\infty^2}{\e^{\beta_t} \delta^2}  \right).
\end{align*}
\end{lem}
\begin{proof}
The main difference with the proof of Lemma \ref{lem:convergence-of-fluctuations-part-3-bis} is that we handle the conditional first moment differently: we express it in the form $Z_{ht} (\chi_h,\gamma)$ for some function $\chi_h$ and then estimate it using Proposition \ref{prop:cv_rate_for_bad_functions}.

Firstly, using \eqref{eq:first-moment-on-the-stopping-line-gamma} and a change of variable in the integral, we have
\begin{align}
\E \Biggl[
\sum_{u \in \cL^{ht,\gamma}_{(ht,\infty)}}
\Upsilon(T_u/t)
\Bigg| \sF_{ht} \Biggr]
& = \e^{\beta_t} 
	\sum_{v \in \cN(ht)}
	(X_v(ht) - \gamma)_+ \e^{-X_v(ht)}  
	\int_h^\infty
	\frac{\e^{-(X_v(ht) - \gamma)^2/[2(r-h)t]}}{\sqrt{2\pi}(r-h)^{3/2}} 
	\Upsilon(r)
	\diff r \nonumber \\
& = \e^{\beta_t} Z_{ht} (\chi_h,\gamma), \label{qa}
\end{align}
where we set, for $y > 0$, 
\begin{align*}
	\chi_h(y)
	\coloneqq 
	\int_h^\infty
	\frac{\e^{-y^2h/[2(r-h)]}}{\sqrt{2\pi}(r-h)^{3/2}} 
	\Upsilon(r)
	\diff r.
\end{align*}
For the second moment, since $\Upsilon$ is bounded, we can proceed as in \eqref{pj}, but using the first bound in Lemma \ref{lem:second-moment-number-of-killed-particles} instead of the second one, and we get
\begin{align} \label{qb}
\Var \Biggl(
\sum_{u \in \cL^{ht,\gamma}_{(ht,\infty)}}
\Upsilon(T_u/t)
\Bigg| \sF_{ht} \Biggr) 
\leq C \norme{\Upsilon}_\infty^2 
\e^{\gamma} W_{ht}.
\end{align}
Combining \eqref{qa} and \eqref{qb}, it follows from Chebyshev's inequality given $\sF_{ht}$ that
\begin{equation} \label{eq:bound_cond}
\P \Biggl( \Biggl\lvert
\e^{-\beta_t} \sum_{u \in \cL^{ht,\gamma}_{(ht,\infty)}} 
	\Upsilon(T_u/t)
- Z_{ht} (\chi_h,\gamma)
\Biggr\rvert
\geq \delta
\Bigg| \sF_{ht} \Biggr) 
\leq C \delta^{-2} \e^{-2 \beta_t} \norme{\Upsilon}_\infty^2 
\e^{\gamma} W_{ht}.
\end{equation}
To bound the non-conditional probability, we first restrict ourselves to the event $A \coloneqq \{\forall r \in [0,ht], \min_{u \in \cN(r)} X_u(r) > -\beta_t \}$ by bounding $\P(A^c) \leq \e^{-\beta_t}$ by \eqref{eq:global-min-of-the-BBM}, and then we take the expectation of the right-hand side of \eqref{eq:bound_cond} on the event $A$, noting that
$\E[\sqrt{ht} W_{ht} \1_A] \leq C \beta_t$ by \eqref{eq:bound-Z_s(F,gamma,s)-with-barrier-at--M}.
This yields
\begin{align} \label{qc}
\P \Biggl( \Biggl\lvert
\e^{-\beta_t} \sum_{u \in \cL^{ht,\gamma}_{(ht,\infty)}} 
	\Upsilon(T_u/t)
- Z_{ht} (\chi_h,\gamma)
\Biggr\rvert
\geq \delta \Biggr) 
\leq \e^{-\beta_t} + \frac{C \beta_t}{\sqrt{h}} \delta^{-2} \e^{-\beta_t} \norme{\Upsilon}_\infty^2.
\end{align}
Now we have to estimate $Z_{ht} (\chi_h,\gamma)$ and for this we want to apply Proposition \ref{prop:cv_rate_for_bad_functions}.
Firstly, note that $\chi_h$ is differentiable on $(0,\infty)$ and, for any $y>0$, using the variable change $u = y\sqrt{h}/\sqrt{r-h}$, 
\[
\abs{\chi_h'(y)}
\leq \norme{\Upsilon}_\infty
	\int_h^\infty
	\frac{yh\e^{-y^2h/[2(r-h)]}}{\sqrt{2\pi}(r-h)^{5/2}} 
	\diff r
= \frac{\norme{\Upsilon}_\infty}{y^2 \sqrt{h}} \sqrt{\frac{2}{\pi}} 
\int_0^\infty u^2 \e^{-u^2/2} \diff u
= \frac{\norme{\Upsilon}_\infty}{y^2 \sqrt{h}}.
\]
Similarly, we have $\chi_h(y) \leq \frac{\norme{\Upsilon}_\infty}{y\sqrt{h}}$.
Hence, the function $\frac{\sqrt{h}}{\norme{\Upsilon}_\infty} \chi_h$ satisfies Assumptions \ref{ass1_alpha}-\ref{ass2_alpha} with $\alpha = 1$ and $\kappa = 0$.
Therefore, it follows from Proposition \ref{prop:cv_rate_for_bad_functions} that
\begin{equation} \label{qd}
\Pp{\abs{Z_{ht}(\chi_h ,\gamma)- \rho(\chi_h) Z_\infty} \geq \delta}
\leq (ht)^{-10} + C(\log t)^2 (ht)^{-1/6} \frac{\norme{\Upsilon}_\infty}{\delta \sqrt{h}}.
\end{equation}
Moreover, we have $\E[\e^{-a R_1^2/2}] = (a+1)^{-3/2}$ by a direct calculation using the density in \eqref{eq:density-of-R_1-under-P}, and so we get $\rho(\chi_h) = \int_h^\infty \frac{\Upsilon(r) \diff r}{\sqrt{2\pi} r^{3/2}}$.
Hence, the result follows from \eqref{qc} and \eqref{qd}.
\end{proof}
\begin{cor} \label{cor:convergence-of-fluctuations}
Let $\gamma = \frac{1}{2} \log t + \beta_t$ with $\beta_t \to \infty$ and $\beta_t/\log t \to 0$.
Let $g \colon [0,\infty) \to \R$ be a measurable bounded function. 
For any $h > 0$, we have
\[
\E \Biggl[ \exp \Biggl( i \sum_{u \in \cL^{ht,\gamma}_{(ht,\infty)}} \e^{-\beta_t}
	g(T_u/t) \left( Z_\infty^{(u)} - \beta_t \right) \Biggr) \Bigg| \sF_{ht} \Biggr]
\xrightarrow[t\to\infty]{(\P)} 
\Ecsq{\exp \left( i\int_h^\infty g(r) M_{Z_\infty}(\diff r) \right)}{Z_\infty}.
\]
\end{cor}
\begin{proof} 
It is enough to prove the result with $\sF_{\cL^{ht,\gamma}}$ instead of $\sF_{ht}$.
The proof follows exactly the lines of the proof of Corollary \ref{cor:convergence-of-fluctuations-part-2}, by replacing the interval $(t^a,t]$ by $(ht,\infty)$ and the use of Lemma \ref{lem:convergence-of-fluctuations-part-3-bis} by the use of Lemma \ref{lem:sums_along_stopping_line}.
\end{proof}
In the second result, we allow the test function $\Upsilon$ to diverge close to 1. This result, as well as its corollary below, is tailored for its future use in the proof of Proposition \ref{prop:regularize}.
\begin{lem} \label{lem:sums_along_stopping_line_2}
Let $\gamma = \frac{1}{2} \log t + \beta_t$ with $\beta_t \in [0,10 \log t]$ such that $\beta_t \to \infty$ as $t \to \infty$. Let $C > 0$ and $\theta \in (0,1)$. 
For every $t>0$, let $\Upsilon = \Upsilon^{(t)} \colon [0,1) \to \C$ be a measurable function such that, $\abs{\Upsilon(r)} \leq C (1-r)^{-\theta}$, $r \in [0,1)$.
Then, for any $h \in (0,1)$, we have 
\begin{align} \label{qe}
\beta_t \Biggl( 
\e^{-\beta_t} \sum_{u \in \cL^{ht,\gamma}_{(ht,t)}} 
	\Upsilon(T_u/t)
- \int_h^1 \Upsilon(r) \frac{Z_\infty}{\sqrt{2\pi}} \frac{\diff r}{r^{3/2}}
\Biggr)
\xrightarrow[t\to\infty]{(\P)} 0.
\end{align}
\end{lem}
\begin{proof}
	We can without loss of generality that the constant $C$ from the statement equals $1$.
Fix some $k > 1/(1-\theta)$ and let $\Upsilon_1(r) \coloneqq \Upsilon(r) \1_{r < 1-\beta_t^{-k}}$.
Noting that $\norme{\Upsilon_1}_\infty \leq\beta_t^{k\theta}$, it follows from Lemma \ref{lem:sums_along_stopping_line} that \eqref{qe} holds with $\Upsilon_1$ instead of $\Upsilon$.
Therefore, it is now sufficient to prove 
\begin{align} \label{qf}
\beta_t \e^{-\beta_t} \sum_{u \in \cL^{ht,\gamma}_{(t-t\beta_t^{-k},t)}} \abs{\Upsilon(T_u/t)}
\xrightarrow[t\to\infty]{(\P)} 0
\qquad \text{and} \qquad 
\beta_t \int_{1-\beta_t^{-k}}^1 \abs{\Upsilon(r)} \frac{\diff r}{r^{3/2}}
\xrightarrow[t\to\infty]{} 0.
\end{align}
Using \eqref{eq:first-moment-on-the-stopping-line-gamma} and a change of variable in the integral,
\begin{align*}
\E \Biggl[ \beta_t \e^{-\beta_t} 
\sum_{u \in \cL^{ht,\gamma}_{(t-t\beta_t^{-k},t)}}
\abs{\Upsilon(T_u/t)}
\Bigg| \sF_{ht} \Biggr]
& \leq \beta_t
	Z_{ht}(1,\gamma)
	\int_{1-\beta_t^{-k}}^1
	\frac{\abs{\Upsilon(r)} \diff r}{\sqrt{2\pi}(r-h)^{3/2}}
\leq C \beta_t^{1-k(1-\theta)} Z_{ht}(1,\gamma),
\end{align*}
where $C$ depends on $h$. Since $(Z_{ht}(1,\gamma))_{t\geq1}$ is tight by Lemma~\ref{lem:tightness} and using that $k > 1/(1-\theta)$, the first part of \eqref{qf} follows. The second part of \eqref{qf} is straightforward.
\end{proof}
\begin{cor} \label{cor:convergence-of-fluctuations-2}
Let $\gamma = \frac{1}{2} \log t + \beta_t$ with $\beta_t \in [0,10 \log t]$ such that $\beta_t \to \infty$ as $t \to \infty$.
Let $C > 0$ and $\theta \in (0,1)$. 
Let $g \colon [0,1) \to \C$ be a measurable function such that, $\abs{g(r)} \leq C (1-r)^{-\theta}$, $r \in [0,1)$.
Then, for any $h \in (0,1)$ and $p \in (0,1/\theta)$, we have
\[
	\sum_{u \in \cL^{ht,\gamma}_{(ht,t-t \e^{-p\beta_t})}} \e^{-\beta_t}
	g(T_u/t) \left( Z_\infty^{(u)} - \beta_t \right)
	\xrightarrow[t\to\infty]{\text{(law)}} 
	\int_h^1 g(r) M_{Z_\infty}(\diff r).
\]
\end{cor}
\begin{proof} 
The proof follows the lines of the proof of Corollary \ref{cor:convergence-of-fluctuations-part-2}
with $\cL^{ht,\gamma}_{(ht,t-t \e^{-p\beta_t})}$ instead of~$\cL^{t^a,\gamma}_{(t^a,t]}$ and applying Lemma \ref{lem:sums_along_stopping_line_2} instead of Lemma \ref{lem:convergence-of-fluctuations-part-3-bis}.
The single step which should be justified differently is \eqref{pc} because the function $g$ is not bounded anymore here. However, for $s \leq t-t \e^{-p\beta_t}$, we have $g(s/t) \leq \e^{p\beta_t\theta}$ and therefore it is still true that $\psi(\e^{-\beta_t} g(s/t)) \to 0$ uniformly in $s \leq t-t \e^{-p\beta_t}$, using that $p < 1/\theta$.
\end{proof}

\section{Proof of the main results}
\label{section:proof_theorem}

In this section, we state a second theorem, Theorem~\ref{theorem-2}, which complements Theorem~\ref{theorem-complete}. 
Instead of comparing $Z_{at}(F)$ with its limit $\rho(F)Z_\infty$, it compares it with $Z_t(G)$ for some appropriate function $G$. This allows to prove a convergence conditionally on $\sF_t$ instead of $\sF_{\ep t}$ with $\ep \to 0$. 
Therefore, Theorem~\ref{theorem-2} contains a bit more information than Theorem~\ref{theorem-complete} on the effect of the conditioning.
However, Theorem~\ref{theorem-complete} is not a direct consequence of Theorem~\ref{theorem-2} because taking the limit $\varepsilon \to 0$ requires an additional argument which is where we make use of the non-conditional convergence proved in Proposition~\ref{prop:convergence-with-translation-gamma_t}.

In Section~\ref{sec:another-theorem}, we introduce further notation and state Theorem~\ref{theorem-2}.
In Section~\ref{section:weaker_version_theorem_2}, we prove a weaker version of Theorem~\ref{theorem-2} for functions $F$ that do not diverge at 0.
Section~\ref{subsection:regularize} contains a tool to compare $Z_t(F)$ for $F$ diverging at $0$ with $Z_{(1-\varepsilon)t}(G)$ for some non-diverging function $G$.
This is used in Section~\ref{subsection:proof_theorem2} to prove Theorem~\ref{theorem-2} as a consequence of the weaker version mentioned before.
Finally, Theorem~\ref{theorem-complete} is established in Section~\ref{subsection:proof_theorem1}.


\subsection{Another theorem}
\label{sec:another-theorem}

We first introduce some additional notation.
Let $F \colon \R \to \R$ and $h \in [0,1]$. Recall from \eqref{eq:def_F_h} that we introduced $F_h(x) = F(\sqrt{1-h}\cdot x)$. We define a new auxiliary function: for $x\geq 0$,
\begin{equation}
	\label{eq:def F^h}
F^h(x) \coloneqq \Eci{x\sqrt{h}}{F(R_{1-h})} = \begin{cases}
	\Eci{x\sqrt{h/(1-h)}}{F_h(R_1)}, & h\in [0,1)\\
	F(x), & h = 1.
\end{cases}
\end{equation}
Note that $F^0 = \rho(F)$.
This function appears for the following reason: 
recalling the definition of $\widetilde{Z}_t^{ht,\gamma}(F,\gamma)$ in \eqref{eq:def_Z_t^t_0,gamma(F,gamma)}, applying the branching property at time $ht$ and then the many-to-one formula \eqref{eq:many-to-one-bessel}, we have for $h\in [0,1)$,
\begin{align}
\Ecsq{\widetilde{Z}_t^{ht,\gamma} (F,\gamma)}{\sF_{ht}}
& = \sum_{v \in \cN(ht)} 
	\e^{-\gamma} \1_{X_v(ht) > \gamma} 
	\Eci{X_v(ht)-\gamma}{\widetilde{Z}_{(1-h)t}(F_h)} \nonumber \\
& = \sum_{v \in \cN(ht)} 
	(X_v(ht)-\gamma)_+ \e^{-X_v(ht)} 
	\Eci{(X_v(ht)-\gamma)/\sqrt{t(1-h)}}{F_h(R_1)} \nonumber \\
& = Z_{ht}(F^h,\gamma), \label{eq:property_F^h}
\end{align}
and \eqref{eq:property_F^h} also trivially holds for $h=1$.
This means that if we want to compare $Z_t(F)$ with an $\sF_{ht}$-measurable quantity, a natural candidate is $Z_{ht}(F^h)$. This motivates the statement of Theorem~\ref{theorem-2} below. We set $Z_\infty(F) \coloneqq \rho(F)Z_\infty$.

\begin{thm} \label{theorem-2}
	Let $F \colon \R \to \R$.
	Assume $F$ is twice differentiable on $(0,\infty)$ and, for any $x>0$, $\abs{F''(x)} \leq C x^{-\alpha-2} \e^{C x}$ for some $\alpha \in [0,2)$ and $C>0$.
	Then, $\rho(F_{\cdot/a})\in \mathcal G(r^{-3/2} \1_{1\leq r < a} \diff r)$ for all $a\in [1,\infty]$ and conditionally on $\sF_t$, as $t\to\infty$, the finite-dimensional distributions of
	\begin{equation}\label{eq:process_at_t_2}
	\sqrt t \cdot \left( 
	Z_{at}(F) - Z_t(F^{1/a})
	+ \frac{\log t}{2 \sqrt{t}} 
	\int_1^a \rho(F_{r/a})\frac{Z_\infty}{\sqrt{2\pi}}  \frac{\diff r}{r^{3/2}}
	\right)_{a\in[1,\infty]}
	\end{equation}
	converge weakly in probability to the finite-dimensional distributions, conditionally on $Z_\infty$, of
	\begin{equation} \label{eq:process_limit_2}
	\left( \int_1^a \rho(F_{r/a}) M_{Z_\infty}(\diff r)\right)_{a\in[1,\infty]}.
	\end{equation}		
	More generally, if $F_1,\ldots,F_n$ are functions as above, and we consider the processes defined as in \eqref{eq:process_at_t} with $F = F_i$ for each $i=1,\ldots,n$, then the above convergence holds jointly in $i=1,\ldots,n$.
\end{thm}

We conclude this subsection by stating basic properties of the new auxiliary function $F^h$.
Let $h,h' \in [0,1]$.
Using successively twice the definition of $F^h$, the scaling and Markov properties of the Bessel process, for any $x\geq0$, we have
\begin{align}
(F^h)^{h'}(x) & = \Eci{x\sqrt{h'}}{F^h(R_{1-h'})}
= \Eci{x\sqrt{h'}}{\Eci{R_{1-h'}\sqrt{h}}{F(R_{1-h})}} \nonumber \\
& = \Eci{x\sqrt{hh'}}{\Eci{R_{h(1-h')}}{F(R_{1-h})}}
= \Eci{x\sqrt{hh'}}{F(R_{1-hh'})}
=F^{hh'}(x). \label{eq:Fh_semi_group}
\end{align}
Similarly, if $r \in [0,1)$, we have $\E[F^h(R_{1-r})] = \E[\E_{R_{h(1-r)}}[F(R_{1-h})]] = \E[F(R_{1-rh})]$, in other words, 
\begin{equation} \label{eq:Fh_and_rho}
\rho((F^h)_r) = \rho(F_{rh}). 
\end{equation}
With $r=0$, this gives $\rho(F^h) = \rho(F)$. Hence, we get, for any $r\geq0$, using \eqref{eq:R F with F_h},
\begin{equation} \label{eq:formula R F^h}
\mathscr R F^h(r) =  \rho(F_{rh})\1_{r<1} - \rho(F) = \mathscr R F(rh) - \rho(F_{rh}) \1_{1 \leq r< 1/h}.
\end{equation}
In particular, we have
\begin{equation}
	\label{eq:formula R F^h r<1}
	\mathscr R F^h(r) = \mathscr R F(rh),\quad r<1.
\end{equation}
Furthermore, \eqref{eq:R F with F_h} and \eqref{eq:formula R F^h} yield,
\begin{equation}
	\label{eq:more_RFh}
	\mathscr R F^h(r) - \rho(F)\mathscr R 1(r) = \mathscr R F(rh)\1_{r<1}.
\end{equation}

\subsection{A weaker form of Theorem~\ref{theorem-2}}
\label{section:weaker_version_theorem_2}

In this section, we prove the following result, close to Theorem \ref{theorem-2}: the only difference is that the function $F$ is not allowed to diverge at 0.
%
\begin{prop} \label{prop:weaker_version_theorem_2}
	Let $F \colon \R \to \R$.
	Assume $F$ is twice differentiable on $(0,\infty)$ and, for any $x>0$, $\abs{F''(x)} \leq C \e^{C x}$ for some $C>0$.
	Then, conditionally on $\sF_t$, as $t\to\infty$, the finite-dimensional distributions of \eqref{eq:process_at_t_2} converge weakly in probability to the finite-dimensional distributions, conditionally on $Z_\infty$, of \eqref{eq:process_limit_2}.
\end{prop}

\begin{proof}
	Note that \eqref{eq:process_at_t_2} and \eqref{eq:process_limit_2} are simply zero if $a=1$, so it is enough to work with $a \in (1,\infty]$ in the proof.
	We repetitively use in the proof that, up to dividing it by a constant, the function $F$ satisfies Assumptions \ref{ass1}-\ref{ass2} and \hyperref[ass1_alpha]{(H1$_{0,\kappa}$)}-\hyperref[ass2_alpha]{(H2$_{0,\kappa}$)}-\hyperref[ass3_alpha]{(H3$_{0,\kappa}$)} for $\kappa$ large enough. The same is true for the function $F^{1/a}$ for every $a\in(1,\infty]$, by Lemma~\ref{lem:Fh}.
	We consider $\gamma = \frac{1}{2} \log t + \beta_t$, with $\beta_t \to \infty$ and $\beta_t = o(\log t)$ as $t \to \infty$.
	
	We first claim that we can replace \eqref{eq:process_at_t_2} by%
	\begin{equation} \label{eq:process2_aux}
		\sqrt t \cdot \left( 
		Z_{at}(F,\gamma) - Z_t(F^{1/a},\gamma)
		- \frac{\beta_t}{\sqrt{t}} 
		\int_1^a \rho(F_{r/a})\frac{Z_\infty}{\sqrt{2\pi}}  \frac{\diff r}{r^{3/2}}
		\right)_{a\in(1,\infty]},
	\end{equation}
	with the convention $Z_\infty(F,\gamma) \coloneqq \rho(F) Z_\infty$.
	Indeed, by \eqref{eq:formula R F^h}, we have for every $a\in (1,\infty]$,
	\begin{align*}
	\int_1^a \rho(F_{r/a})\frac{\diff r}{r^{3/2}} 
	&= \int_0^\infty \mathscr R F^{1/a}(r) \frac{\diff r}{r^{3/2}} - \int_0^\infty \mathscr R F(r/a) \frac{\diff r}{r^{3/2}} \\
	&= \int_0^\infty \mathscr R F^{1/a}(r) \frac{\diff r}{r^{3/2}} - \frac 1 {\sqrt{a}} \int_0^\infty \mathscr R F(r) \frac{\diff r}{r^{3/2}}.
	\end{align*}
	It follows from Proposition \ref{prop:from-translation-gamma_t-to-0} that for every $a\in(1,\infty]$, we have the following convergence in probability as $t\to\infty$:
	\begin{equation}
		\sqrt t \cdot \left(Z_{at}(F) - Z_t(F^{1/a}) - (Z_{at}(F,\gamma) - Z_t(F^{1/a},\gamma)) + \frac{\gamma}{\sqrt{t}} \int_1^a \rho(F_{r/a})\frac{Z_\infty}{\sqrt{2\pi}}  \frac{\diff r}{r^{3/2}}\right) \to 0.
	\end{equation}
	This implies\footnote{We repetitively use in this proof that we can modify the process we are considering as long as the difference tends to 0 in probability. The fact that this operation is allowed when proving a weak convergence in probability of conditional finite-dimensional distributions is stated in \cite[Remark A.3]{maillardpain2019}.} that the finite-dimensional distributions of \eqref{eq:process2_aux} and of \eqref{eq:process_at_t_2}, conditioned on $\mathscr F_t$, converge to the same limit, as claimed.
	
	We now follow the strategy presented in Section \ref{subsection:scheme_barrier} with $t_0=t$ and $\gamma$ as above. 
	For any fixed $a \in (1,\infty]$, we decompose as in \eqref{eq:decompo-Z_t(F,gamma)}:
	\begin{equation} \label{eq:decompo1}
	Z_{at}(F,\gamma) = 
	\widetilde{Z}_{at}^{t,\gamma}(F,\gamma)
	+ \sum_{u\in\cL^{t,\gamma}} \Omega^{(u)}_{at},
	\end{equation}
	where $\Omega^{(u)}_{at}$ is the contribution of the progeny of particle $u$ in $Z_{at}(F,\gamma)$, see \eqref{eq:def_Omega} when $a<\infty$ and recall $\Omega^{(u)}_{at}=0$ if $T_u > at$. 
	In the case $a=\infty$, we have $\Omega^{(u)}_{at} = \rho(F) \e^{-X_u(T_u)} Z_\infty^{(u)}$, with $Z_\infty^{(u)}$ defined in \eqref{eq:def_Z_infty^u}, and $\widetilde{Z}_\infty^{t,\gamma}(F,\gamma) = \rho(F) \widetilde{Z}_\infty^{t,\gamma}$, where $\widetilde{Z}_\infty^{t,\gamma}$ is the a.s.\@ limit of the nonnegative martingale $(\widetilde{Z}_s^{t,\gamma}(1,\gamma))_{s \geq t}$, so that the decomposition \eqref{eq:decompo1} corresponds exactly to the one in \cite[Eq.\@ (3.2)]{maillardpain2019}.
	
	We first study $\widetilde{Z}_{at}^{t,\gamma}(F,\gamma)$, and assume first that $a \in (1,\infty)$.
	By \eqref{eq:property_F^h}, we have
	\begin{align}
	\Ecsq{\widetilde{Z}_{at}^{t,\gamma}(F,\gamma) }{\sF_{t}}
	= Z_{t}(F^{1/a},\gamma). \label{ggb2}
	\end{align}
	Moreover, proceeding as in \eqref{eq:conditional_variance_1} and then using the second bound of Lemma~\ref{lem:second-moment-Z_s(F,t,Delta)} with $s=(a-1)t$, we get
	\begin{align}
	\Varsq{\widetilde{Z}_{at}^{t,\gamma}(F,\gamma)}{\sF_{t}}
	& \leq C \e^{-\gamma} \sum_{v \in \cN(t)} 
		\1_{X_v(t) > \gamma} \e^{-X_v(t)}
		\left(	
		1 + \frac{X_v(t) - \gamma}{\sqrt{(a-1)t}}
		\right) \nonumber \\
	& \leq C \frac{\e^{-\beta_t}}{t} 
	Z_{t} (x\mapsto x^{-1} + 1, \gamma), \label{ggc2}
	\end{align}
	where $C$ depends on $a$.
	Combining \eqref{ggb2}, \eqref{ggc2} and Lemma~\ref{lem:tightness} proves that, for any $a \in (1,\infty)$, 
	\begin{equation*}
	\sqrt{t} ( \widetilde{Z}_{at}^{t,\gamma}(F,\gamma) - Z_{t}(F^{1/a},\gamma)) \xrightarrow[t\to\infty]{(\P)} 0.
	\end{equation*}
	When $a=\infty$, this convergence has been proved in \cite[Lemma~5.1]{maillardpain2019}.
	Hence, we can replace \eqref{eq:process2_aux} by
	\begin{equation}
		\label{ggd2}
		\sqrt t \cdot \left(\sum_{u\in\cL^{t,\gamma}} \Omega^{(u)}_{at} - \frac{\beta_t}{\sqrt{t}} 
		\int_1^a \rho(F_{r/a})\frac{Z_\infty}{\sqrt{2\pi}}  \frac{\diff r}{r^{3/2}}
		\right)_{a\in(1,\infty]}.
	\end{equation}
	
	It therefore remains to show that the finite-dimensional distributions of \eqref{ggd2}, conditioned on $\sF_t$, converge to those, conditioned on $Z_\infty$, of \eqref{eq:process_limit_2}.
	We first note that, by \eqref{eq:local-min-of-the-BBM},
	\begin{equation} \label{eq:part_at_time_t}
		\Pp{\cL^{t,\gamma}_{\{t\}} \neq \emptyset}
		\leq \Pp{\min_{u\in\cN(t)} X_u(t) \leq \gamma}
		\xrightarrow[t \to \infty]{} 0,
	\end{equation}
	so that it is enough in \eqref{ggd2} to sum over particles $u\in\cL^{t,\gamma}_{(t,\infty)}$ instead of $u\in\cL^{t,\gamma}$. We now proceed in a similar way as Lemma~\ref{lem:convergence-of-fluctuations}, so we skip some details here.
	We first fix $a\in (1,\infty)$ and aim at replacing the contributions $\Omega^{(u)}_{at}$ by their limiting behavior (as detailed below this step is immediate if $a=\infty$).
	First note that, setting $t_2=t-t\e^{-2\beta_t}$ and proceeding as in \eqref{eq:killed_between_t_2_and_t}, we have
	\begin{equation} \label{eq:empty_part_1}
	\Pp{\cL^{t,\gamma}_{(a t_2,at]} \neq \emptyset} \xrightarrow[t \to \infty]{} 0.
	\end{equation}
	For the contributions of particles $u$ killed in the time interval $(t,at_2]$, we proceed as for the proof of \eqref{eq:simplifying-the-contribution}: we note that, given $\sF_{\cL^{t,\gamma}}$, the $\Omega^{(u)}_{at} - \e^{-\gamma} \rho(F_{T_u/(at)}) Z_\infty^{(u)}$ for $u \in \cL^{\ep t,\gamma}_{(t,a t_2]}$ are independent with the same law as $\e^{-\gamma} (Z_{at-T_u} (F_{T_u/(at)},0) - \rho(F_{T_u/(at)}) Z_\infty)$ under $\P$, 
	then apply Corollary \ref{cor:rate_of_convergence_as_expectation} to get
	\begin{equation*} 
	\E \Biggl[ \Biggl( \sqrt{t} 
	\sum_{u \in \cL^{t,\gamma}_{(t,at_2]}} 
	\abs{ \Omega^{(u)}_{at} - \e^{-\gamma} \rho(F_{T_u/(at)}) Z_\infty^{(u)} }
	\Biggr) \wedge 1 
	\Bigg| \sF_{\cL^{t,\gamma}} \Biggr]
	\leq C (\log t)^3 t^{-1/3} \e^{-\beta_t/3} \# \cL^{t,\gamma}_{(t,at_2]}.
	\end{equation*}
	We now note that the conditional expectation of the right-hand side given $\sF_{t}$ tends to 0 in probability by \eqref{eq:number-of-killed-particles-given-sF_s^a}, \eqref{eq:cv-of-W_t} and $\beta_t = o(\log t)$.
	Combining this with \eqref{eq:empty_part_1} and the fact that $\Omega_{at}^{(u)} = 0$ when $T_u > at$, we obtain
	\begin{align} \label{ggf}
		\sqrt{t} \sum_{u \in \cL^{t,\gamma}_{(t,\infty)}}
		\abs{ \Omega^{(u)}_{at} - \e^{-\gamma}\rho(F_{T_u/(at)})\1_{T_u\leq at} Z_\infty^{(u)} }
		\xrightarrow[t\to\infty]{(\P)} 0.
	\end{align}
	This is also true for $a=\infty$, because in that case the left-hand side equals zero.
	Combining this with \eqref{eq:part_at_time_t}, it follows that we can replace \eqref{ggd2} by
	\begin{align}
		\label{ggf2}
		\Biggl(\e^{-\beta_t}\sum_{u\in\cL^{t,\gamma}_{(t,at]}} \rho(F_{T_u/(at)}) Z_\infty^{(u)} - \beta_t 
		\int_1^a \rho(F_{r/a})\frac{Z_\infty}{\sqrt{2\pi}}  \frac{\diff r}{r^{3/2}}
		\Biggr)_{a\in(1,\infty]}.
	\end{align}
	Moreover, by Lemma \ref{lem:sums_along_stopping_line}, we have for every $a\in(1,\infty]$,
	\begin{align} \label{ggg1}
		\e^{-\beta_t} \sum_{u\in\cL^{\ep t,\gamma}_{(t,at]}} \rho(F_{T_u/(at)}) \beta_t
		- \beta_t \int_1^a \rho(F_{r/a})\frac{Z_\infty}{\sqrt{2\pi}}  \frac{\diff r}{r^{3/2}}
		\xrightarrow[t\to\infty]{(\P)} 0.
	\end{align}
	Therefore, we can now replace \eqref{ggf2} by 
	\begin{equation} \label{gge1}
	\Biggl( 
	\sum_{u\in\cL^{t,\gamma}_{(t,at]}} \e^{-\beta_t} \rho(F_{T_u/(at)}) (Z_\infty^{(u)} - \beta_t) \Biggr)_{a\in(1,\infty]}.
	\end{equation}
	Hence, in order to prove the result, by \cite[Proposition A.1]{maillardpain2019}, it is enough to prove that the conditional characteristic function of finite-dimensional marginals of \eqref{gge1} given $\sF_{t}$, as $t\to\infty$, converges in probability to the characteristic function of finite-dimensional marginals of \eqref{eq:process_limit_2} given $Z_\infty$.
	But this follows from Corollary \ref{cor:convergence-of-fluctuations} applied to the function $g(r) = \sum_{k=1}^n \xi_k \rho(F(r/a_k))\1_{r\le a_k}$, for some fixed $n \geq 1$, $\xi_1,\dots,\xi_n \in \R$ and $a_1,\dots,a_n \in(1,\infty]$.
\end{proof}

\subsection{Regularizing the function at 0}
\label{subsection:regularize}

We now deal with functions diverging as $x^{-\alpha}$ at 0, for some $\alpha \in [0,2)$.
The following result will allow us to regularize such a function at 0: note that the function $F^{1-\ep}$ defined in \eqref{eq:def F^h} is smooth and locally bounded at 0 by Lemma~\ref{lem:Fh}. 
The proof relies again on the scheme of Section \ref{subsection:scheme_barrier} but note that, exceptionally, we have to work here with $\beta_t$ of order $\log t$.
\begin{prop} \label{prop:regularize}
	Let $F \colon \R \to \R$. Assume $F$ is twice differentiable on $(0,\infty)$ and, for any $x>0$, $\abs{F''(x)} \leq C x^{-\alpha-2} \e^{C x}$ for some $\alpha \in [0,2)$ and $C>0$.
	As $t\to\infty$ and then $\ep \to 0$, we have
	\begin{equation*} 
	\sqrt{t} \left( 
	Z_t(F) - Z_{(1-\ep)t}(F^{1-\ep}) 
	+ \frac{\log t}{2\sqrt{t}} \int_{1-\ep}^1 \rho(F_r) \frac{Z_\infty}{\sqrt{2\pi}} \frac{\diff r}{r^{3/2}}
	\right) 
	\xrightarrow{(\P)} 0.
	\end{equation*}
\end{prop}

\begin{proof}
	First note that, without loss of generality, we can assume that $\alpha >0$ and that the function $F$ satisfies Assumptions \ref{ass1_alpha}-\ref{ass2_alpha}-\ref{ass3_alpha} for some $\kappa$ large enough (note that if $\alpha=0$ then the function $F$ could explode logarithmically at 0).
	
	Then, consider $\gamma = \frac{1}{2} \log t + \beta_t$, with $\beta_t = b \log t$ for some $b \in ((\alpha - 1)\vee 0, 1)$.
	Observe that the convergence in the statement is equivalent to
	\begin{equation} \label{eq:previous_statement}
	\sqrt{t} \left( 
	Z_t(F,\gamma) - Z_{(1-\ep)t}(F^{1-\ep},\gamma) 
	- \frac{\beta_t}{\sqrt{t}} \int_{1-\ep}^1 \rho(F_r) \frac{Z_\infty}{\sqrt{2\pi}} \frac{\diff r}{r^{3/2}}
	\right) 
	\xrightarrow[t\to\infty,\, \ep\to 0]{(\P)} 0.
	\end{equation}
	This follows from the same argument as in the proof of Proposition~\ref{prop:weaker_version_theorem_2}, following \eqref{eq:process2_aux}: we apply Proposition \ref{prop:from-translation-gamma_t-to-0} to $Z_t(F,\gamma)$ and $Z_{(1-\ep)t}(F^{1-\ep},\gamma)$ (by Lemma~\ref{lem:Fh}, up to dividing it by a constant which can depend on $\varepsilon$, the function $F^{1-\varepsilon}$ satisfies Assumptions \hyperref[ass1_alpha]{(H1$_{0,\kappa}$)}-\hyperref[ass2_alpha]{(H2$_{0,\kappa}$)}-\hyperref[ass3_alpha]{(H3$_{0,\kappa}$)} for $\kappa$ large enough) and use the fact that
	\[
		\frac{\gamma}{\sqrt{t}} \int_0^\infty \mathscr{R}F(r) \frac{Z_\infty}{\sqrt{2\pi}} \frac{\diff r}{r^{3/2}}
		- \frac{\gamma}{\sqrt{(1-\varepsilon)t}} \int_0^\infty \mathscr{R}F^{1-\varepsilon}(r) \frac{Z_\infty}{\sqrt{2\pi}} \frac{\diff r}{r^{3/2}}
		= \frac{\gamma}{\sqrt{t}} \int_{1-\ep}^1 \rho(F_r) \frac{Z_\infty}{\sqrt{2\pi}} \frac{\diff r}{r^{3/2}},
	\]
	by applying \eqref{eq:formula R F^h} and a change of variable in the second integral on the left-hand side.
	
	We now aim at proving \eqref{eq:previous_statement}.
	Let $\ep > 0$. We apply the scheme of Section \ref{subsection:scheme_barrier} with $t_0 = (1-\ep)t$ and $\gamma$ as above.
	We decompose as in \eqref{eq:decompo-Z_t(F,gamma)}:
	\begin{equation} \label{ob}
	Z_t(F,\gamma) = 
	\widetilde{Z}_t^{(1-\ep) t,\gamma}(F,\gamma)
	+ \sum_{u\in\cL^{(1-\ep) t,\gamma}} \Omega^{(u)}_{t},
	\end{equation}
	where $\Omega^{(u)}_{t}$ is the contribution of the progeny of $u$ in $Z_t(F,\gamma)$ defined in \eqref{eq:def_Omega}.
	By \eqref{eq:property_F^h}, we have 
	\[
	\Ecsq{\widetilde{Z}_{t}^{(1-\ep) t,\gamma}(F,\gamma)}{\sF_{(1-\ep) t}} = Z_{(1-\ep)t}(F^{1-\ep},\gamma).
	\]
	Moreover, using branching property at time $(1-\ep)t$ as in \eqref{eq:conditional_variance_1} and applying Lemma \ref{lem:f_alpha}, we get
	\begin{align*}
	\Varsq{\widetilde{Z}_{t}^{(1-\ep) t,\gamma}(F,\gamma)}{\sF_{(1-\ep) t}}
	& \leq C \e^{-\gamma} \sum_{v \in \cN((1-\ep) t)} \e^{-X_v((1-\ep)t)}
		\left( 1 + \frac{X_v((1-\ep)t)}{\sqrt{\ep t}} g_\alpha(\ep t)
		\right),
	\end{align*}
	where $g_\alpha(s)$ equals 1 if $\alpha < 1$, $\log s$ if $\alpha =1$ and $s^{\alpha-1}$ if $\alpha > 1$.
	For fixed $\ep > 0$, this is a $O(\e^{-\gamma} g_\alpha(t) t^{-1/2})$ in probability as $t\to\infty$, using Lemma~\ref{lem:tightness}. 
	Since $b > (\alpha - 1) \vee 0$, this shows that the conditional variance is a $o(t^{-1})$ in probability and we conclude that
	\begin{equation} \label{oc}
	\sqrt{t}
	\left( \widetilde{Z}_{t}^{(1-\ep) t,\gamma}(F,\gamma) 
	- Z_{(1-\ep)t}(F^{1-\ep},\gamma) \right)
	\xrightarrow[t\to\infty]{(\P)} 0.
	\end{equation}

	We now have to estimate the contributions of killed particles and for this we proceed in a way similar to the proof of Lemma \ref{lem:convergence-of-fluctuations} to show the convergence in distribution as $t \to \infty$ and then we argue that this limit vanishes when $\ep \to 0$.
	First note, using \eqref{eq:local-min-of-the-BBM} and $b < 1$, that no particle is killed exactly at time $(1-\ep)t$ with high probability.
	Moreover, for any $p > 1$, it follows from the same argument as in \eqref{eq:killed_between_t_2_and_t} that, with high probability, no particle is killed during the time interval $[(1- \e^{-p\beta_t})t,t]$.
	Therefore, we can deal only with particles killed in the time interval $I \coloneqq ((1-\ep)t, (1- \e^{-p\beta_t})t)$.
	We fix $p \in (1,\frac{2}{\alpha} \wedge \frac{1}{b})$ for reasons that will be made clear later.

	Given $\sF_{\cL^{(1-\ep) t,\gamma}}$, the $(\Omega^{(u)}_t,Z^{(u)}_\infty)$ for $u \in \cL^{(1-\ep) t,\gamma}_{((1-\ep) t,t]}$ are independent with respectively the same distribution as $(\e^{-\gamma} Z_{t-T_u}(F_{T_u/t},0), Z_\infty)$ under $\P$, recalling from \eqref{eq:def_F_h} that $F_h(x) = F(\sqrt{1-h}\cdot x)$.
	Hence, we have, with $G(r) \coloneqq \rho(F_r)$ and $\chi(s) \coloneqq \E[(t^{-b} \lvert Z_{t-s}(F_{s/t},0) - G(s/t) Z_\infty \rvert) \wedge 1]$,
	\begin{equation} \label{oa}
	\E\Biggl[\Biggl( \sqrt{t}
	\sum_{u\in\cL^{(1-\ep) t,\gamma}_{I}} 
	\abs{\Omega^{(u)}_{t} - \e^{-\gamma} G(T_u/t) Z^{(u)}_\infty }
	\Biggr) \wedge 1
	\,\Bigg|\, \sF_{\cL^{(1-\ep) t,\gamma}} \Biggr] 
	\leq \sum_{u\in\cL^{(1-\ep) t,\gamma}_{I}} \chi(T_u).
	\end{equation}
	Since $F$ satisfies Assumptions \ref{ass1_alpha}-\ref{ass2_alpha}, the function $(1-h)^{\alpha/2} F_h$ satisfies Assumptions \ref{ass1_alpha}-\ref{ass2_alpha}.
	Thus, for any $s \leq (1- \e^{-p\beta_t})t$, noting that $t-s \geq t^{1-bp}$ and recalling that $1-bp > 0$ by our choice of $p$, we can apply Corollary~\ref{cor:rate_of_convergence_as_expectation_for bad_functions} with $\ep = t^{-b} (1-\frac{s}{t})^{-\alpha/2}$ (which satisfies $\ep \geq t^{-b} \geq (t-s)^{-K}$ with $K = b/(1-bp)$) to get $\chi(s) \leq (t-s)^{-1/12} \ep$.
	Hence, \eqref{oa} is smaller than
	\[
	t^{-(1-bp)/12} \cdot t^{-b}  
	\sum_{u\in\cL^{(1-\ep) t,\gamma}_{I}} 
	\left( 1-\frac{T_u}{t} \right)^{-\alpha/2}
	\xrightarrow[t\to\infty]{(\P)} 0,
	\]
	using Lemma \ref{lem:sums_along_stopping_line_2} and that $\alpha<2$ and $t^{-(1-bp)/12} \to 0$.
	Therefore, the left-hand side of \eqref{oa} tends to 0 in probability and, combining this with \eqref{ob} and \eqref{oc}, it is now sufficient to prove that
	\begin{equation}
	\sum_{u\in\cL^{(1-\ep) t,\gamma}_{I}} 
	\e^{-\beta_t} G(T_u/t) Z^{(u)}_\infty 
	- \beta_t \int_{1-\ep}^1 G(r) \frac{Z_\infty}{\sqrt{2\pi}} \frac{\diff r}{r^{3/2}} 
	\xrightarrow[t\to\infty,\, \ep\to 0]{(\P)} 0. \label{jja}
	\end{equation}
	Note that, using Assumption \ref{ass1_alpha} for $F$, we have, for any $r \in [0,1)$,
	\begin{equation} \label{eq:bound_rho(F_r)}
	\abs{G(r)} = \abs{\rho(F_r)}
	\leq (1-r)^{-\alpha/2} \int_0^\infty x^{-\alpha} \e^{\kappa \sqrt{1-r} x} \sqrt{\frac 2 \pi} x^2 e^{-x^2/2} \diff x
	\leq C(\alpha,\kappa) (1-r)^{-\alpha/2}.
	\end{equation}
	Therefore, it follows from Lemma \ref{lem:sums_along_stopping_line_2} applied to $\Upsilon = G \1_I$ that
	\[
	\beta_t \Biggl( 
	\sum_{u\in\cL^{(1-\ep) t,\gamma}_{I}} 
	\e^{-\beta_t} G(T_u/t) 
	- \int_{1-\ep}^1 G(r) \frac{Z_\infty}{\sqrt{2\pi}} \frac{\diff r}{r^{3/2}} \Biggr)
	\xrightarrow[t\to\infty]{(\P)} 0.
	\]
	Moreover, it follows from Corollary \ref{cor:convergence-of-fluctuations-2} (using here that $p< \alpha/2$) that
	\begin{equation} \label{jjb}
	\sum_{u\in\cL^{(1-\ep) t,\gamma}_{I}} 
	\e^{-\beta_t} G(T_u/t) 
	\left( Z^{(u)}_\infty - \beta_t \right)
	\xrightarrow[t\to\infty]{\text{(law)}} \int_{1-\ep}^1 G(r) M_{Z_\infty}(\diff r).
	\end{equation}
	Using that the function $\abs{G \log \abs{G}}$ is integrable at 1, a characteristic function calculation shows that the right-hand side of \eqref{jjb} tends to 0 in probability as $\ep \to 0$.
	Therefore, \eqref{jja} is proved and this concludes the proof.
\end{proof}

\subsection{Proof of Theorem~\ref{theorem-2}}
\label{subsection:proof_theorem2}

We prove here Theorem~\ref{theorem-2}: we use Proposition~\ref{prop:regularize} to replace the (possibly) diverging function~$F$ by the non-diverging function $F^{1-\ep}$ and then apply Proposition~\ref{prop:weaker_version_theorem_2} to $F^{1-\ep}$.

\begin{proof}[Proof of Theorem \ref{theorem-2}]
	The claim that $\rho(F_{\cdot/a})\in \mathcal G(r^{-3/2} \1_{1\leq r < a} \diff r)$ for all $a\in [1,\infty]$ follows from \eqref{eq:bound_rho(F_r)}. That bound also yields that $r \in [0,1) \mapsto \rho(F_r)$ is integrable.
	
	For the remainder of the proof, it is enough to prove the convergence with a single function $F$, for, if $F_1,\dots,F_n$ are functions as in the statement of the theorem, it is enough to apply the theorem to the function $F = \sum_{i=1}^n \xi_i F_i$ for any $\xi_1,\dots,\xi_n \in \R$ to deduce the joint convergence.	
	Also the quantities in \eqref{eq:process_at_t_2} and \eqref{eq:process_limit_2} are zero for $a=1$ so it is enough to consider $a \in (1,\infty]$.
	
	For any $a \in (1,\infty]$, for $\ep \in (0,1)$ small enough such that $(1-\ep)a>1$, we decompose the process into three parts:
	\begin{align}
	& \sqrt t \cdot \left( 
	Z_{at}(F) - Z_t(F^{1/a})
	+ \frac{\log t}{2 \sqrt{t}} 
	\int_1^a \rho(F_{r/a})\frac{Z_\infty}{\sqrt{2\pi}}  \frac{\diff r}{r^{3/2}}
	\right) \nonumber \\
	\begin{split}
	& = \sqrt t \cdot \biggl( 
	Z_{at}(F) - Z_{(1-\ep)at}(F^{1-\ep}) 
	+ \frac{\log (at)}{2\sqrt{at}} \int_{1-\ep}^1 \rho(F_r) \frac{Z_\infty}{\sqrt{2\pi}} \frac{\diff r}{r^{3/2}} \\
	& \qquad \qquad 
	+ Z_{(1-\ep)at}(F^{1-\ep}) - Z_t(F^{1/a}) 
	+ \frac{\log t}{2 \sqrt{t}} 
	\int_1^{(1-\ep)a} \rho(F_{r/a})\frac{Z_\infty}{\sqrt{2\pi}}  \frac{\diff r}{r^{3/2}} \\
	& \qquad \qquad 
	- \frac{\log a}{2\sqrt{at}} \int_{1-\ep}^1 \rho(F_r) \frac{Z_\infty}{\sqrt{2\pi}} \frac{\diff r}{r^{3/2}}
	\biggr).
	\end{split}
	 \label{eq:th2_decompo}
	\end{align}
	In the case $a=\infty$, we use the convention $\log a/\sqrt a = 0$ and so the first and third parts are simply~0, because $Z_{(1-\ep)at}(F^{1-\ep}) = \rho(F^{1-\ep}) Z_\infty = \rho(F) Z_\infty = Z_{at}(F)$ using \eqref{eq:Fh_and_rho}.
	For $a \in (1,\infty)$, the first part converges to 0 in probability as $t\to\infty$ and then $\ep\to0$ by Proposition~\ref{prop:regularize}.
	The third part also converges to 0 (even almost surely) as $t\to\infty$ and then $\ep\to0$ because $r \in [0,1) \mapsto \rho(F_r)$ is integrable.
	Therefore, it remains to prove that, conditionally on $\sF_t$, as $t\to\infty$ and then $\ep \to 0$, the finite-dimensional distributions of the second part seen as a process indexed by $a \in (1,\infty]$ converge weakly in probability to the finite-dimensional distributions, conditionally on $Z_\infty$, of \eqref{eq:process_limit_2}.
	
	Fix $\varepsilon \in (0,1)$. For $a\in[1/(1-\ep),\infty]$, note that $F^{1/a} = (F^{1-\ep})^{1/((1-\ep)a)}$ by \eqref{eq:Fh_semi_group} and 
	$\rho(F_{r/a}) = \rho((F^{1-\ep})_{r/((1-\ep)a)})$ by \eqref{eq:Fh_and_rho}.
	Therefore, by Proposition~\ref{prop:weaker_version_theorem_2} applied to the function $F^{1-\ep}$ (it satisfies the assumptions by Lemma~\ref{lem:Fh}), conditionally on $\sF_t$, as $t\to\infty$,	
	the finite-dimensional distributions of
	\begin{equation*}
	\sqrt t \cdot \left( 
	Z_{(1-\ep)at}(F^{1-\ep}) - Z_t(F^{1/a}) 
	+ \frac{\log t}{2 \sqrt{t}} 
	\int_1^{(1-\ep)a} \rho(F_{r/a})\frac{Z_\infty}{\sqrt{2\pi}}  \frac{\diff r}{r^{3/2}}
	\right)_{a\in[1/(1-\ep),\infty]}
	\end{equation*}
	converge weakly in probability to the finite-dimensional distributions, conditionally on $Z_\infty$, of
	\begin{equation} \label{eq:process_limit_4}
	\left( \int_1^{(1-\ep)a} \rho(F_{r/a}) M_{Z_\infty}(\diff r)\right)_{a\in[1/(1-\ep),\infty]}.
	\end{equation}		
	Then, an explicit characteristic function calculation, together with the fact that $\rho(F_{\cdot/a})\in \mathcal G(r^{-3/2} \1_{1\leq r < a} \diff r)$, shows that the finite-dimensional distributions, conditionally on $Z_\infty$, of \eqref{eq:process_limit_4} converge weakly in probability to the finite-dimensional distributions, conditionally on $Z_\infty$, of \eqref{eq:process_limit_2}. This concludes the proof.
\end{proof}

\subsection{Proof of Theorem~\ref{theorem-complete}}
\label{subsection:proof_theorem1}

In this section, we prove Theorem \ref{theorem-complete} as a byproduct of Theorem~\ref{theorem-2} and Proposition~\ref{prop:convergence-with-translation-gamma_t}. 

\begin{proof}[Proof of Theorem~\ref{theorem-complete}]
	As observed in the proof of Theorem \ref{theorem-2}, it is enough to prove the convergence with a single function $F$ and, as observed in the proof of Proposition~\ref{prop:regularize}, we can assume that $\alpha>0$ and that $F$ satisfies Assumptions \ref{ass1_alpha}-\ref{ass2_alpha}-\ref{ass3_alpha} for some $\kappa$ large enough.
	Recall $\mathscr{R}F(r) = \rho(F_r) \1_{r<1} - \rho(F)$, which is therefore a finite constant for $r\geq 1$.
	By \eqref{eq:bound_rho(F_r)}, we have $\mathscr{R}F(r) \leq C(1-r)^{-\alpha/2}$ for $r \in [0,1)$.
	Moreover, for $r \in [0,1/2]$, using Assumption~\ref{ass2_alpha}, for any $x >0$, 
	\[
	\abs{F(x \sqrt{1-r}) - F(x)}
	\leq \int_{x \sqrt{1-r}}^x y^{-\alpha-1} \e^{\kappa x} \diff y
	\leq C(\alpha) r x^{-\alpha} \e^{\kappa x},
	\]
	which then implies that $\mathscr{R}F(r) \leq C r$.
	Combining all these previous bounds shows that $\mathscr R F(\cdot/a)\in \mathcal G(r^{-3/2}\diff r)$ for all $a\in(0,\infty)$. 
	We now aim at proving the claimed convergence.
	
	First fix $\ep>0$.
	Theorem~\ref{theorem-2} applied with $\ep t$ instead of $t$, together with the scaling property of the 1-stable noise $M_{Z_\infty}$ (see Remark~\ref{rem:scaling}) yields after a few lines of calculation the following result: conditioned on $\sF_{\ep t}$, the finite-dimensional distributions of the process
	\begin{align}
	\label{eq:process_3}
	\sqrt t \cdot \left(Z_{at}(F)-Z_{\ep t}(F^{\ep/a}) + \frac{\log t}{2\sqrt t}\int_\ep^a\rho(F_{r/a})\frac{Z_\infty}{\sqrt{2\pi}}\frac{\diff r}{r^{3/2}}\right)_{a\in[\ep,\infty]}
	\end{align}
	converge weakly in probability as $t\to\infty$ to the finite-dimensional distributions, conditioned on $Z_\infty$, of
	\begin{align}
	\label{eq:process_limit_3}
	\left(\int_\ep^a\rho(F_{r/a})M_{Z_\infty}(\diff r)\right)_{a\in[\ep,\infty]}.
	\end{align}
	We now decompose the process \eqref{eq:process_at_t} for $a\in[\ep,\infty)$ into three parts as follows, using \eqref{eq:R F with F_h}:
	\begin{align}
	\sqrt t \cdot \Bigg( 
	Z_{at}(F) - \rho(F) Z_\infty 
	&+ \frac{\log t}{2 \sqrt{t}} 
	\int_0^\infty \mathscr R F\left( \frac{r}{a} \right) \frac{Z_\infty}{\sqrt{2\pi}} \frac{\diff r}{r^{3/2}}\Bigg)\nonumber\\
	\begin{split}
	= \sqrt t\cdot \Bigg(&Z_{at}(F)-Z_{\ep t}(F^{\ep/a}) + \frac{\log t}{2\sqrt t} \int_\ep^a \rho(F_{r/a}) \frac{Z_\infty}{\sqrt{2\pi}} \frac{\diff r}{r^{3/2}} \\
	&+ \rho(F) Z_{\ep t} - \rho(F)Z_\infty - \frac{\log t}{2\sqrt t}\int_\ep^\infty \rho(F) \frac{Z_\infty}{\sqrt{2\pi}} \frac{\diff r}{r^{3/2}} \\
	&+ Z_{\ep t}(F^{\ep/a})-\rho(F)Z_{\ep t} + \frac{\log t}{2\sqrt t} \int_0^\ep \mathscr R F\left(\frac r a\right) \frac{Z_\infty}{\sqrt{2\pi}} \frac{\diff r}{r^{3/2}}\Bigg).	
	\end{split}
	\label{eq:process_decomposition_3_parts}
	\end{align}
	The first part is exactly \eqref{eq:process_3}. The second part is also \eqref{eq:process_3}, with $a = \infty$ and a minus sign. It follows that the finite-dimensional distributions of the sum of the first two parts of \eqref{eq:process_decomposition_3_parts}, conditioned on $\sF_{\ep t}$, converge weakly in probability to the finite-dimensional distributions, conditioned on $Z_\infty$, of
	\begin{align*}
	\left(\int_\ep^\infty \mathscr R F\left(\frac r a\right) M_{Z_\infty}(\diff r)\right)_{a\in[\ep,\infty)},
	\end{align*}
	using \eqref{eq:R F with F_h} again. As $\ep\to 0$, these converge (conditioned on $Z_\infty$) to the finite-dimensional distributions of \eqref{eq:process_limit}, using that $\mathscr R F(\cdot/a)\in \mathcal G(r^{-3/2}\diff r)$ to justify the convergence. Hence, it suffices to show that the third part in \eqref{eq:process_decomposition_3_parts} converges to 0 in probability as $t\to\infty$, then $\ep\to 0$. 
	
	Without loss of generality, we may assume $a=1$. We wish to use Proposition \ref{prop:convergence-with-translation-gamma_t}, but first have to introduce a shift. Define $\gamma = \frac 1 2 \log t + \beta_t$, with $\beta_t\to\infty$ and $\beta_t = o(\log t)$ as $t\to\infty$. 
	Note that, by Lemma~\ref{lem:Fh}, up to dividing it by a constant possibly depending on $\varepsilon$, the function $F^\varepsilon$ satisfies Assumptions \ref{ass1}-\ref{ass2} and \hyperref[ass1_alpha]{(H1$_{0,\kappa}$)}-\hyperref[ass2_alpha]{(H2$_{0,\kappa}$)}-\hyperref[ass3_alpha]{(H3$_{0,\kappa}$)} for $\kappa$ large enough.
	By Proposition~\ref{prop:from-translation-gamma_t-to-0} and \eqref{eq:more_RFh}, in order to show that the third part of \eqref{eq:process_decomposition_3_parts} converges to 0 in probability, it is enough to show that, as $t\to\infty$ and then $\ep\to0$,
	\begin{equation}
	\label{Fepstoshow}
	\sqrt t\cdot \left(Z_{\ep t}(F^{\ep},\gamma) - \rho(F)Z_{\ep t}(1,\gamma) - \frac{\beta_t}{\sqrt t} \int_0^\ep \mathscr R F(r) \frac{Z_\infty}{\sqrt{2\pi}} \frac{\diff r}{r^{3/2}} \right)\xrightarrow{(\P)} 0.
	\end{equation}
	By Proposition \ref{prop:convergence-with-translation-gamma_t} applied at time $\varepsilon t$ with $\gamma = \frac{1}{2} \log (\ep t) + (\beta_t - \frac{1}{2} \log \ep)$, we have for fixed $\ep\in(0,1)$,
	\begin{align}
	\label{eq:convergence from prop 2.2}
	& \sqrt {\ep t}\cdot \left(Z_{\ep t}(F^{\ep},\gamma) - \rho(F)Z_{\ep t}(1,\gamma) - \frac{\beta_t - \frac{1}{2} \log \varepsilon}{\sqrt {\ep t}} \int_0^1 \mathscr R F^\ep(r) \frac{Z_\infty}{\sqrt{2\pi}} \frac{\diff r}{r^{3/2}} \right) \nonumber \\
	& \xrightarrow[t\to\infty]{\text{(law)}} 
	\int_0^1 \mathscr R F^\ep(r)M_{Z_\infty}(\diff r).
	\end{align}
	Using \eqref{eq:formula R F^h r<1} and the scaling properties of the 1-stable noise (Remark~\ref{rem:scaling}), we can rewrite \eqref{eq:convergence from prop 2.2} as follows
	\begin{equation} \label{eq:convergence from prop 2.2 bis}
	\sqrt {t}\cdot \left(Z_{\ep t}(F^{\ep},\gamma) - \rho(F)Z_{\ep t}(1,\gamma) - \frac{\beta_t}{\sqrt {t}} \int_0^\ep \mathscr R F(r) \frac{Z_\infty}{\sqrt{2\pi}} \frac{\diff r}{r^{3/2}} \right)
	\xrightarrow[t\to\infty]{\text{(law)}} 
	\int_0^\ep \mathscr R F(r)M_{Z_\infty}(\diff r).
	\end{equation}
	As $\ep\to0$, the right-hand side of \eqref{eq:convergence from prop 2.2 bis} converges to 0 in probability,
	using again that $\abs{\mathscr R F(u)} \leq Cu$ for $u\in[0,1]$ by \eqref{eq:RF_bound}. This proves \eqref{Fepstoshow} and finishes the proof of the theorem.
\end{proof}

\appendix

\section{Some calculations regarding the 3-dimensional Bessel process}
\label{section:technical-results}

Let $(R_s)_{s\geq 0}$ denote a 3-dimensional Bessel process starting from $x$ under $\P_x$. 
Recall the following link between the 3-dimensional Bessel process and the Brownian motion (see Imhof \cite{imhof84}): for any $t \geq 0$, any $x > 0$ and any measurable function $\Theta \colon \cC([0,t]) \to \R_+$,
\begin{equation} \label{eq:link-between-R-and-B}
	\Eci{x}{\Theta(B_s,s\in [0,t]) \1_{\forall s \in [0,t], B_s > 0}}
	=
	\Eci{x}{\frac{x}{R_s} \Theta(R_s,s\in [0,t])}.
\end{equation}
In this subsection, we bound the expectations of some functions of $R_1$.
First recall that the density of $R_1$ under $\P = \P_0$ is
\begin{equation} \label{eq:density-of-R_1-under-P}
	z \mapsto \sqrt{\frac{2}{\pi}} z^2 \e^{-z^2/2} \1_{z>0},
\end{equation}
and the density of $R_1$ under $\P_y$, for $y>0$, is
\begin{equation} \label{eq:density-of-R_1-under-P_x}
	z \mapsto 
	\frac{z}{y} \frac{\e^{-(z-y)^2/2}}{\sqrt{2\pi}} 
	\left( 1 - \e^{-2zy} \right) \1_{z>0}
	= \sqrt{\frac{2}{\pi}} \frac{z}{y} \e^{-z^2/2}
	\e^{-y^2/2} \sinh(zy) \1_{z>0}.
\end{equation}
Note that, for any $\lambda \ge 0$ and $\alpha \in [0,3)$, 
$\E[(R_1)^{-\alpha} \e^{\lambda R_1}]$ is finite.
\begin{lem} \label{lem:control-E_y[R_1^-alpha exp(lambdaR_1)]}
	Let $\lambda \ge 0$ and $\alpha \in [0,3)$. There exists $C=C(\lambda,\alpha) >0$ such that, for any $y \ge 0$, we have
		\begin{equation*} 
		\Eci{y}{ (R_1)^{-\alpha} \e^{\lambda R_1}}
		\leq C (y^{-\alpha}\wedge 1)\e^{\lambda y},
	\end{equation*}
	and
	\begin{equation*} 
		\Eci{y}{ \left((R_1)^{-\alpha}+1 \right) \e^{\lambda R_1}}
		\leq C \e^{\lambda y}.
	\end{equation*}
\end{lem}
\begin{proof}
	The second inequality follows by applying the first one twice, once with the same $\alpha$ and once with $\alpha = 0$. It therefore suffices to prove the first inequality.
	
	We first treat the case $y\le 1$. 
	Distinguishing whether $R_1 < 1$ or $R_1 \geq 1$, 
	\[
	(R_1)^{-\alpha} \e^{\lambda R_1}
	\leq (R_1)^{-\alpha} \e^{\lambda} + \e^{\lambda R_1}.
	\]
	By monotonicity of the Bessel process w.r.t.\@ the initial condition, we have $$\Eci{y}{(R_1)^{-\alpha}} \leq \Ec{(R_1)^{-\alpha}} < \infty,$$ and $$\E_y[\e^{\lambda R_1}] \leq \E_1[\e^{\lambda R_1}] < \infty,$$ which yields the result. 
	
	Suppose now that $y\ge 1$.
	Using \eqref{eq:density-of-R_1-under-P_x} and $1 - \e^{-2zy} \leq (2zy) \wedge 1$, we have (bounding the constants by 1),
		\begin{align}
			\label{eq:R_1_integral_decomp_y}
		\Eci{y}{(R_1)^{-\alpha}\e^{\lambda R_1}}
		& \leq \int_0^{y/2} z^{2-\alpha} e^{\lambda z}
		\e^{-(z-y)^2/2}  \diff z
		+ \int_{y/2}^\infty \frac{z^{1-\alpha}}{y} e^{\lambda z} \e^{-(z-y)^2/2} \diff z \eqqcolon I_1+I_2.
		\end{align}
		In the remainder of the proof, let $C$ denote a constant that may depend on $\lambda$ and $\alpha$ and whose value may change from line to line.
		The first term in \eqref{eq:R_1_integral_decomp_y} is bounded by
		\begin{align}
			\label{eq:I_1_bound}
			I_1 \le e^{\lambda y/2-y^2/8}\int_0^{y/2} z^{2-\alpha}\diff z \le C y^{-\alpha}.
		\end{align}
		For the second term, we have
		\begin{align*}
		I_2 & \leq 2^{\alpha}y^{-\alpha} \frac 1 y \int_{y/2}^\infty z e^{\lambda z} \e^{-(z-y)^2/2} \diff z \\
		& \le 2^\alpha  \e^{\lambda^2/2} y^{-\alpha} \e^{\lambda y} \frac 1 y \int_{-\infty}^\infty (y+t+\lambda)_+ \e^{-t^2/2} \diff t.
	\end{align*}
	Bounding $(y+t+\lambda)_+ \le y + (t+\lambda)_+$, we obtain,
	\begin{align}
				\label{eq:I_2_bound}
		I_2 \le C y^{-\alpha}e^{\lambda y} \frac{1+y}{y} \le 2 Cy^{-\alpha}e^{\lambda y},
	\end{align}
	where we finally used that $y\ge 1$ in the last inequality. Collecting \eqref{eq:R_1_integral_decomp_y}, \eqref{eq:I_1_bound} and \eqref{eq:I_2_bound}, yields the result in the case $y\ge 1$ and finishes the proof of the lemma.
\end{proof}
\begin{lem} \label{lem:control-eta-small}
	Let $\kappa \geq 0$. There is a constant $C = C(\kappa) > 0$ such that for any function $F \colon \R \to \R$ satisfying \ref{ass1}-\ref{ass2} and any $y \geq 0$, $\alpha \in (0,1]$ and $\eta \geq 0$, we have
	\[
	\abs{\Eci{y}{\frac{(R_1-\eta)_+}{R_1} F(\alpha(R_1-\eta))}
		- \rho(F)} 
	\leq C \left((y^2\wedge 1) \e^{\kappa \alpha y} + \eta + (1-\alpha) \right).
	\]
\end{lem}
\begin{proof}
	We prove the inequality in three steps.
		The first step is to prove that
	\begin{align}
		\abs{
			\Eci{y}{\frac{(R_1-\eta)_+}{R_1} F(\alpha(R_1-\eta))}
			- \Ec{\frac{(R_1-\eta)_+}{R_1} F(\alpha(R_1-\eta))}
		} 
		& \leq C (y^2\wedge 1) \e^{\kappa\alpha y}. \label{le}
	\end{align}
	If $y\ge 1$, this follows from Lemma~\ref{lem:control-E_y[R_1^-alpha exp(lambdaR_1)]}, using Assumption~\ref{ass1}. It remains to consider $y < 1$.
	By \eqref{eq:density-of-R_1-under-P} and \eqref{eq:density-of-R_1-under-P_x}, the left-hand side of \eqref{le} equals
	\begin{align}
		\abs{\int_0^\infty 
			\frac{(z-\eta)_+}{z} F(\alpha(z-\eta))
			\sqrt{\frac{2}{\pi}} \frac{z}{y} \e^{-z^2/2}
			\left( \e^{-y^2/2} \sinh(zy) - zy \right) \diff z}. \label{ld}
	\end{align}
	But, for $y\in[0,1)$ and $z\ge 0$, we have
	\begin{align*}
		\abs{\e^{-y^2/2} \sinh(zy) - zy}
		& \leq \e^{-y^2/2}\abs{\sinh(zy) - zy} + \abs{\e^{-y^2/2}-1} zy\\
		&\leq \sum_{k \geq 1} \frac{(zy)^{2k+1}}{(2k+1)!} 
		+ \frac{z y^3}{2} \\
		& \leq (zy)^3 \sum_{k \geq 0} \frac{z^{2k}}{(2k+3)!} + z y^3\\
		&\leq (z^3 \e^{z} + z)y^3 
	\end{align*}
	Then, using Assumption \ref{ass1} (because of the factor $(z-\eta)_+$, $F$ is only evaluated at positive points) and that $\eta \geq 0$, \eqref{ld} is smaller than
	\begin{align*}
		C\left(
		\int_0^\infty 
		\e^{\kappa z} 
		\e^{-z^2/2}
		\left( z^2+ z^4 \e^{z} \right)\diff z\right) y^2,
	\end{align*}
	and this proves \eqref{le}.
	The second step consists in proving that
	\begin{equation} \label{lf}
		\Ec{\abs{
				\frac{(R_1-\eta)_+}{R_1} F(\alpha (R_1-\eta))
				- F(\alpha R_1) 
		}} 
		\leq C \eta. 
	\end{equation}
	This follows from the fact that for any $x > 0$, by Assumptions \ref{ass1}-\ref{ass2}, we have
	\begin{align}
		\abs{(x-\eta)_+ F(\alpha (x-\eta)) - x F(\alpha x)}
		& \leq (x-\eta)_+  \abs{F(\alpha (x-\eta)) - F(\alpha x)} 
		+ \abs{F(\alpha x)} \cdot \abs{(x-\eta)_+ - x} \nonumber \\
		& \leq (x \alpha+1) \eta \e^{\kappa \alpha x}, \label{eq:bound_Lipschitz}
	\end{align}
	whence,
	\[
	\Ec{\abs{
			\frac{(R_1-\eta)_+}{R_1} F(\alpha (R_1-\eta))
			- F(\alpha R_1) 
	}} \le \Ec{(1+R_1^{-1})e^{\kappa R_1}} \eta.
	\]
 This proves \eqref{lf}, since the last expectation is finite.
	Finally, the third step consists in the bound $\abs{\E[F(\alpha R_1)] - \rho(F)} \leq C (1-\alpha)$, which follows easily from Assumption \ref{ass2}.
	The result follows from this combined with \eqref{le} and \eqref{lf}.
\end{proof}

\begin{lem} \label{lem:Fh}
	Let $\alpha \in [0,2)$ and $\kappa \geq 0$.
	Let $F\colon \R \to \R$ be a function satisfying Assumption \ref{ass1_alpha}. 
	Let $h \in [0,1)$ and recall $F^h$ is defined in \eqref{eq:def F^h}.
	Then, $F^h$ is twice differentiable on $(0,\infty)$ and there exists $C = C(\alpha,\kappa,h)>0$ such that $(F^h)''(x) \leq C \e^{C x}$ for any $x>0$.
\end{lem}

\begin{proof}
	For $y \in \R^*$, we write $\sinhc(y) \coloneqq \sinh(y)/y$, which is extended at 0 by continuity into an analytic function on $\R$.
	First recalling the definition of $F^h$ and then using the density of $R_1$ given in \eqref{eq:density-of-R_1-under-P_x}, we get, for any $x>0$
	\begin{align}
		F^h(x) & = \Eci{x\sqrt{h/(1-h)}}{F(R_1 \sqrt{1-h} )} \nonumber \\
		& = \sqrt{\frac{2}{\pi}} \e^{-x^2 h/[2(1-h)]} 
		\int_0^\infty z^2 \e^{-z^2/2}
		\sinhc\left(\frac{zx \sqrt{h}}{\sqrt{1-h}} \right) 
		F\left(z\sqrt{1-h}\right) \diff z \nonumber \\
		& \eqqcolon \sqrt{\frac{2}{\pi}} \e^{-x^2 h/[2(1-h)]}  G(x).
		\label{eq:Fh_and_G}
	\end{align}
	A power series expansion shows that, for $k\in\{0,1,2\}$ and any $y>0$, $\sinhc^{(k)}(y) \leq \e^y$. This domination, together with the fact that $F$ satisfies Assumption \ref{ass1_alpha}, is enough to prove that $G$ is twice differentiable on $(0,\infty)$, and we get, for $k\in\{0,1,2\}$ and $x>0$
	\begin{align*}
	\abs{G^{(k)}(x)} 
	& \leq \frac{h^{k/2}}{(1-h)^{(k+\alpha)/2}} 
	\int_0^\infty z^{2+k-\alpha} \e^{-z^2/2}
	\exp\left(\frac{zx \sqrt{h}}{\sqrt{1-h}} + \kappa z \sqrt{1-h} \right) 
	\diff z.
	\end{align*}
	Using that $z^{2+k-\alpha} \e^{\kappa z \sqrt{1-h}} \leq C \e^{(\kappa+1) z}$ and then the formula for the Laplace transform of the normal distribution, we get
	\begin{align*}
	\abs{G^{(k)}(x)} 
	& \leq C \exp\left( \frac{1}{2} \left( \frac{x \sqrt{h}}{\sqrt{1-h}} + \kappa + 1\right)^2 \right)
	\leq C \exp\left( \frac{x^2 h}{2(1-h)} + Cx \right).
	\end{align*}
	Coming back to \eqref{eq:Fh_and_G}, this yields the result.
\end{proof}

\section{Two constants depending of \texorpdfstring{$F$}{f} are identical}
\label{section:two-constants-are-identical}

We prove here the relation \eqref{eq:two-constants-are-identical}. 
It could actually be deduced from the rest of our proof: if it was wrong, the $(\log t)/\sqrt{t}$ corrective term appearing in Theorem \ref{theorem-complete} would depend on the choice of $\beta_t$, which is not possible since $Z_{at}(F) - \rho(F)$ does not depend on $\beta_t$. However, we provide a direct proof here.

\begin{proof}[Proof of Equation \eqref{eq:two-constants-are-identical}]
Recall $\mathscr R F(u) = \Ec{F(\sqrt{1-u} R_1) \1_{u<1} - F(R_1)}$ for $u \geq 0$.
Hence, we have $\int_1^\infty \mathscr R F(u) \frac{\diff u}{u^{3/2}} = -2 \rho(F)$ and we now focus on the integral from 0 to 1.
We first introduce artificially a factor $\sqrt{1-u}$ and integrate by part
\begin{align*}
\int_0^1 \mathscr R F(u) \frac{\diff u}{u^{3/2}}
& = \int_0^1 \sqrt{1-u} \mathscr R F(u) 
	\frac{\diff u}{u^{3/2}\sqrt{1-u}} \\
& = \left[ \sqrt{1-u} \mathscr R F(u) 
	\left( \frac{-2\sqrt{1-u}}{\sqrt{u}} \right) \right]_0^1 \\
	&\quad+ \int_0^1 \left( 
	\frac{-\mathscr R F(u)}{2\sqrt{1-u}} 
	+ \sqrt{1-u}\, (\mathscr R F)'(u) 
	\right) 
	\frac{2\sqrt{1-u}}{\sqrt{u}} \diff u.
\end{align*}
We have $(1-u) \mathscr R F(u) \to 0$ as $u \to 1$ and $\mathscr R F(u)/\sqrt{u} \to 0$ as $u \to 0$ as consequences of Assumptions \ref{ass1_alpha}-\ref{ass2_alpha},
so the boundary terms are zero and we get
\begin{align*}
\int_0^1 \mathscr R F(u) \frac{\diff u}{u^{3/2}}
& = - \int_0^1 \Ec{F(\sqrt{1-u} R_1) + \sqrt{1-u} R_1 F'(\sqrt{1-u} R_1)} \frac{\diff u}{\sqrt{u}} 
+ \int_0^1 \frac{\rho(F)}{\sqrt{u}} \diff u.
\end{align*}
The second term equals $2 \rho(F)$, so we focus on the first one. By \eqref{eq:density-of-R_1-under-P}, it equals
\begin{align*}
& - \int_0^1 \int_0^\infty 
\left( F(\sqrt{1-u} z) + \sqrt{1-u} z F'(\sqrt{1-u} z) \right)
\sqrt{\frac{2}{\pi}} z^2 \e^{-z^2/2} \diff z
\frac{\diff u}{\sqrt{u}}  \\
& = - \int_0^\infty \sqrt{\frac{2}{\pi}} y^2 \e^{-y^2/2}
\left( F(y) + y F'(y) \right)
\left(
\int_0^1 \e^{-uy^2/2(1-u)} \frac{\diff u}{(1-u)^{3/2} \sqrt{u}} 
\right)
\diff y,
\end{align*}
with the change of variables $y = \sqrt{1-u} z$.
Then, with $v = y \sqrt{u/1-u} \Leftrightarrow u = v^2/(y^2+v^2)$, we have
\[
\int_0^1 \e^{-uy^2/2(1-u)} \frac{\diff u}{(1-u)^{3/2} \sqrt{u}} 
= \int_0^\infty \e^{-v^2/2} \frac{2}{y} \diff v
= \frac{\sqrt{2\pi}}{y}.
\]
Thus, we finally get
\begin{align*}
\int_0^1 \mathscr R F(u) \frac{\diff u}{u^{3/2}}
& = 2 \rho(F) - \int_0^\infty \sqrt{\frac{2}{\pi}} y^2 \e^{-y^2/2}
\left( F(y) + y F'(y) \right)
\frac{\sqrt{2\pi}}{y}
\diff y \\
& = 2 \rho(F) - \sqrt{2\pi} \Ec{F'(R_1) + \frac{F(R_1)}{R_1}}
\end{align*}
and the result follows.
\end{proof}

\section{Parameters of the limit in the case of the additive martingale}
\label{app:additive_martingale}

We compute here the values of $\sigma_F$, $\beta_F$ and $\mu_F$ defined in Remark~\ref{rem:one-dimensional}, in the case where 
\[
\mathscr R F(r) = \sqrt{\frac 2 \pi} \left(\frac{\1_{r<1}}{\sqrt{1-r}}-1\right).
\]
We have
\begin{align*}
\sigma_F 
&= \int_0^\infty \abs{\mathscr R F(r)} \frac{\sqrt \pi}{2\sqrt 2} \frac{\diff r}{r^{3/2}}\\
&= \int_0^1 \left(\frac{1}{\sqrt{1-r}}-1\right) \frac{\diff r}{2r^{3/2}}
+ \int_1^\infty \frac{\diff r}{2r^{3/2}}\\
&= \left[ \frac{r-1+\sqrt{r-1}}{\sqrt{r}\sqrt{r-1}} \right]_0^1 + 1 \\
&= 2.
\end{align*}
Similarly, we get $\int_0^\infty \mathscr R F(r) r^{-3/2}\diff r=0$ and therefore $\beta_F=0$.
Finally, 
\begin{align*}
\mu_F 
& = \int_0^\infty \mathscr R F(r)(\log\abs{\mathscr R F(r)}+\mu_Z) \frac{1}{\sqrt{2\pi}} \frac{\diff r}{r^{3/2}} \\
& =\left( \log \sqrt{\frac 2 \pi}  + \mu_Z \right)  \int_0^\infty \mathscr R F(r) \frac{1}{\sqrt{2\pi}} \frac{\diff r}{r^{3/2}}
+ \int_0^1 \mathscr R F(r) \log \left(\frac{1}{\sqrt{1-r}}-1\right) \frac{1}{\sqrt{2\pi}} \frac{\diff r}{r^{3/2}} \\
& = 0 + \frac{1}{\pi} \int_0^1 \left(\frac{1}{\sqrt{1-r}}-1\right) \log \left(\frac{1}{\sqrt{1-r}}-1\right) \frac{\diff r}{r^{3/2}} \\
& = \frac{2}{\pi} \int_0^\infty \frac{\log u}{\sqrt{u} (2+u)^{3/2}} \diff u,
\end{align*}
with the change of variables $u= \frac{1}{\sqrt{1-r}}-1$, and the last integral can be calculated to get
\begin{align*}
\mu_F & = \frac{2}{\pi} \left[ \sqrt{\frac{u}{2+u}} \log u - \log \left( 1 + \sqrt{\frac{u}{2+u}} \right) + \log \left( 1 - \sqrt{\frac{u}{2+u}} \right) \right]_0^\infty 
= - \frac{2\log 2}{\pi}.
\end{align*}

\addcontentsline{toc}{section}{References}
\bibliographystyle{abbrv}
\bibliography{biblio}

\end{document}